\documentclass[a4paper,11pt]{article}
\usepackage{amsmath,amssymb,amsthm,MnSymbol}
\usepackage{paralist,enumitem,hyperref,url,bm,placeins}
\usepackage{color,graphicx}
\usepackage[margin=30mm]{geometry}

\newcommand{\N}{\mathbb{N}}
\newcommand{\R}{\mathbb{R}}
\newcommand{\loc}{\mathrm{loc}}
\newcommand{\nloc}{\mathrm{nloc}}

\newcommand{\bv}{\bm{v}}
\newcommand{\bz}{\bm{z}}
\newcommand{\pd}{\partial}
\renewcommand{\div}{\, \mathrm{div} \,}
\newcommand{\I}{\mathbb{I}}
\newcommand{\bo}{\bm{\omega}}
\newcommand{\bnu}{\bm{n}}

\newcommand{\eps}{\varepsilon}
\newcommand{\D}{\mathrm{D}}
\newcommand{\W}{\mathrm{W}}
\newcommand{\curl}{\, \mathrm{curl}\, }

\newcommand{\pdnu}{\pd_{\bnu}}
\newcommand{\bu}{\bm{u}}

\newcommand{\inn}[2]{\langle #1, #2 \rangle}

\newcommand{\HH}{\mathbb{H}}

\newcommand{\VV}{\mathbb{V}}

\newcommand{\bw}{\bm{w}}

\newcommand{\bA}{\bm{A}}

\theoremstyle{plain}
\newtheorem{thm}{Theorem}[section]
\newtheorem{prop}[thm]{Proposition}
\newtheorem{lem}[thm]{Lemma}

\newtheorem{remark}{Remark}[section]

\numberwithin{equation}{section}

\title{Two phase micropolar fluid flow with nonlocal energies: Existence theory, nonpolar limits and nonlocal-to-local convergence}
\author{Kin Shing Chan \footnotemark[1] \and Kei Fong Lam \footnotemark[1]} 
\date{ }

\begin{document}

\maketitle

\renewcommand{\thefootnote}{\fnsymbol{footnote}}
\footnotetext[1]{Department of Mathematics, Hong Kong Baptist University, Kowloon Tong, Hong Kong \tt(24482951@life.hkbu.edu.hk,akflam@hkbu.edu.hk)}

\begin{abstract}
We study a nonlocal variant of a thermodynamically consistent phase field model for binary mixtures of micropolar fluids, i.e., fluids exhibiting internal rotations. The model is described by a Navier--Stokes--Cahn--Hilliard system that extends the earlier nonlocal variants of the model introduced by Abels, Garcke and Gr\"un for binary Newtonian fluid mixtures with unmatched densities. We establish the global 3D weak existence and global 2D strong well-posedness, followed by the weak convergence of the nonlocal model to its local counterpart as the nonlocal interaction kernel approaches the Dirac delta distribution. In the two dimensional setting we provide consistency estimates between strong solutions of the nonlocal micropolar model and strong solutions of nonlocal variants of the Abels--Garcke--Gr\"un model and Model H.
\end{abstract}



\section{Introduction}
The micropolar fluid model is a simple but significant extension of the well-established Navier--Stokes equations to describe the motion of fluidic bodies where each particle possess an internal angular momentum separated from the gross external angular momentum arising from rigid motions of the whole fluid. This model can be used to represent fluids containing randomly oriented particles suspended in a viscous medium, in which the deformation of the particles are neglected. Common examples include ferrofluids, bubbly flows, blood flows and liquid crystals. Introduced by A.C.~Eringen in \cite{Eringen}, see also \cite{Luka}, the model for isotropic, isothermal and incompressible micropolar fluids arises from the balance of mass, momentum and moment of momentum:
\begin{subequations}\label{microNS}
\begin{alignat}{2}
\div \bu & = 0, \label{MNS:div} \\
\pd_t \bu + (\bu \cdot \nabla) \bu & =  (\eta + \eta_r) \Delta \bu - \nabla p + 2 \eta_r \curl \bo, \label{MNS:mom} \\
\pd_t \bo + (\bu \cdot \nabla) \bo & = (c_0 + c_d - c_a) \nabla \div \bo + (c_a + c_d) \Delta \bo + 2 \eta_r(\curl \bu - 2 \bo), \label{MNS:w}
\end{alignat}
\end{subequations}
where $\bu$ denote the fluid velocity, $p$ the pressure, $\eta$ the viscosity, $\bo$ the micro-rotation, along with $c_0$, $c_a$, $c_d$ as the coefficients of angular viscosities, and $\eta_r$ as the dynamic micro-rotation viscosity. The fluid mass density has been set to $1$, and the non-symmetry of the stress tensor can be attributed to the appearance of the term $2 \eta_r \curl \bo$ in the momentum balance \eqref{MNS:mom}. An interesting observation arises from the formal limit $\eta_r \to 0$, which serves to decouple the Navier--Stokes component \eqref{MNS:div}-\eqref{MNS:mom} with the micro-rotation equation \eqref{MNS:w}. In particular, $\bo$ serves as a background vector field and no longer influence the dynamics of the flow, which is now governed by the classical Navier--Stokes equations.  In this case the fluid is termed \textit{nonpolar}, as the associated stress tensor is symmetric. This allows us to view the dynamic micro-rotation visocsity $\eta_r$ as a measure of the deviation of micropolar fluids from the conventional Navier--Stokes model, see e.g.~\cite{ChanLamstrsol,Luka2D}. Another interesting observation is that in the formal limit $\eta_r \to \infty$ we obtain from \eqref{MNS:w} the relation $\bo = \frac{1}{2} \curl \bu$, which connects the micro-rotation with the well-known fluid vorticity. Some recent results in this direction can be found in \cite{Guterres2}.

Motivated by studies and situations in which mixtures of immiscible micropolar fluids are utilized in real-world applications \cite{TAriMTNS,MDevAR,HRamkSM,PKYadSJBDS}, the authors proposed in \cite{CHLS} a diffuse interface model for two-phase micropolar fluids with unmatched densities. The model, as an extension of the one proposed by Abels, Garcke and Gr\"un (AGG) in \cite{AGGmodel} for two-phase Newtonian flow with unmatched densities, reads as follows:
\begin{subequations}\label{local:model:equ}
\begin{alignat}{2}
& \label{local:model:equ:div:0} \div \bu = 0, \\[1ex]
& \label{local:model:equ:bu} \partial_{t}(\rho(\phi) \bu) + \div(\rho(\phi) \bu \otimes \bu) - \div\left( 2\eta(\phi)\D\bu + 2\eta_{r}(\phi)\W\bu \right) \\[1ex]
\notag & \quad = - \nabla p +  \div\left( \rho'(\phi) m(\phi) \nabla \mu \otimes \bu \right) + \mu\nabla\phi + 2\curl\left( \eta_{r}(\phi)\bo \right), \\[1ex]
& \label{local:model:equ:bo} \partial_{t}(\rho(\phi) \bo) + \div(\rho(\phi) \bu \otimes \bo) - \div\left( c_{0}(\phi)(\div \bo)\mathbb{I} + 2c_{d}(\phi)\D\bo + 2c_{a}(\phi)\W\bo \right) \\[1ex] 
\notag & \quad = \div\left( \rho'(\phi) m(\phi) \nabla \mu \otimes \bo \right) + 2\eta_{r}(\phi)(\curl \bu - 2\bo), \\[1ex] 
& \label{local:model:equ:CH} \partial_{t}\phi + \bu \cdot \nabla\phi = \div\left( m(\phi)\nabla\mu \right), \\[1ex] 
& \label{local:model:equ:mu} \mu = -\sigma \eps \Delta \phi + \frac{\sigma}{\eps} F'(\phi). 
\end{alignat}
\end{subequations}
In the above $\bu$ is a mixture velocity taken as the solenoidal volume averaged velocity \cite{AGGmodel,FBoyer1,HDingPSCS}, $\bo$ is the mixture micro-rotation, $\phi$ is the phase field variable taken as the difference in volume fractions of the two micropolar fluids, $\mu$ is the associated chemical potential and $p$ is the pressure, along with symmetric and antisymmetric rate of strain tensors:
\[
\D \bu = \frac{1}{2} (\nabla \bu + (\nabla \bu)^{\top}), \quad \W \bu = \frac{1}{2} (\nabla \bu - (\nabla \bu)^{\top}).
\]
Although the structure of \eqref{local:model:equ:div:0}--\eqref{local:model:equ:bo} resembles \eqref{microNS}, one key difference lies in the fluid density $\rho$ which is now an affine linear function of $\phi$:
\begin{align}\label{rho:phi}
\rho(\phi) = \frac{\overline{\rho}_1 + \overline{\rho}_2}{2} +  \frac{\overline{\rho}_1 - \overline{\rho}_2}{2}\phi \quad \text{ for } \phi \in [-1,1],
\end{align}
with $\overline{\rho}_1$ and $\overline{\rho}_2$ as the constant mass densities of fluid $1$ and $2$, respectively. Similarly, the viscosity $\eta(\phi)$, the dynamic micro-rotation viscosity $\eta_r(\phi)$, and the coefficients of angular viscosities $c_0(\phi)$, $c_a(\phi)$ and $c_d(\phi)$ are now functions of $\phi$. It is common to choose a similar form as \eqref{rho:phi} to interpolate between two sets of values $\{\eta_{i}, \eta_{r,i}, c_{0,i}, c_{d,i}, c_{a,i}\}_{i=1,2}$ corresponding to the individual constituent fluids. Furthermore, if the constant $\eta_{r,1} = 0$, then fluid 1 will be considered as \textit{nonpolar}. 

Besides the coupling of the micropolar fluid model to a convective Cahn--Hilliard system \eqref{local:model:equ:CH}-\eqref{local:model:equ:mu} with double well potential $F$ (where for later mathematical treatment we choose $F$ to be a singular  potential so that the phase field variable $\phi$ would be naturally bounded in $[-1,1]$), thus leading to the presence of capillary forces $\mu \nabla \phi$ driven by surface tension $\sigma$, the use of the volume averaged velocity under the context of unmatched fluid densities $\overline{\rho}_1 \neq \overline{\rho}_2$ brings about an additional transport mechanism by the relative flux $\bm{J} = - \rho'(\phi) m(\phi)\nabla\mu$ related to the diffusion of the components \cite{AGGmodel}. Lastly, the parameter $0 < \eps \ll 1$ is a measure of the thickness of the interfacial region $\{|\phi| < 1\}$ separating the region occupied by fluid 1 $\{\phi = -1\}$ and the region occupied by fluid 2 $\{\phi = 1\}$. In the case of matched densities $\overline{\rho}_1 = \overline{\rho}_2$, the Navier--Stokes--Cahn--Hilliard component \eqref{local:model:equ:div:0}-\eqref{local:model:equ:bu}, \eqref{local:model:equ:CH}-\eqref{local:model:equ:mu} reduces to the well-known Model H of Hohenberg and Halperin \cite{ModelH.der.,ModelH}. We call the system \eqref{local:model:equ} as the (local) Micropolar Abels--Garcke--Gr\"un (MAGG) model.

Concerning the mathematical analysis of these diffuse interface models for two phase flow, under standard no-slip and no-flux boundary conditions, we refer to \cite{HAbel.EandU.ModelH,AGAMRT} for well-posedness and regularity of solutions to Model H.  For the AGG model we mention \cite{HAbelDDHG,HAbelDDHG2} for global weak solutions, \cite{CHNS3Dstrsol,CHNS2Dstrsol} for local strong solutions, \cite{2023CHNS} for asymptotic stabilization and regularity, as well as \cite{CCGMG1,CCGMG2,CCGMG3} for long time behavior in the context of attractors. The extension to the multicomponent setting with different densities is explored in \cite{AGP}. The inclusion of moving contact line dynamics can be incorporated via a generalized Navier boundary condition (see for instance \cite{Heida,Qian,Ren}) and a dynamic boundary condition of Allen--Cahn type. We refer to \cite{GGM,GGW} for the existence of weak solutions to Model H and to AGG with moving contact line dynamics, to \cite{GLW,AG3,KnopfStange} for the recent analysis where a Cahn--Hilliard type dynamic boundary condition is employed, and to \cite{KnopfStangeNSCH} where a surface Navier--Stokes system is used to account for viscous and convective effects at the domain boundary.

On the other hand, for the MAGG model, global weak solutions in three dimensions with moving contact line dynamics were established in \cite{CHLS}, and local strong well-posedness in three dimensions with standard no-slip, no-spin and no-flux boundary conditions was established in \cite{ChanLamstrsol}. The latter contribution enables a comparison between the MAGG model with its nonpolar variants, where consistency estimates between local strong solutions of the MAGG model and AGG model in terms of the dynamic micro-rotation viscosity $\eta_r$ have been derived.

The main focus of this paper is to investigate a nonlocal variant of the MAGG model, first derived in \cite{CHLS} when a nonlocal approximation of the Ginzburg--Landau functional proposed by Giacomin and Lebowitz \cite{Gia1,Gia2} was employed. The resulting nonlocal MAGG (nMAGG) model is identical to \eqref{local:model:equ} except that \eqref{local:model:equ:mu} is replaced by 
\begin{align}\label{nonlocal:model:equ:mu}
\mu = \eps \sigma \big ( a \phi - K \star \phi) + \frac{1}{\eps} F'(\phi)    
\end{align}
where for a prescribed symmetric convolution kernel $K$, we have
\[
a(x) = (K \star 1)(x) = \int_\Omega K(x-y) dy, \quad (K \star \phi)(x) = \int_\Omega K(x-y) \phi(y) dy.
\]
The associated total energy of the system $\mathcal{E}_{\nloc}(\bu,\bo,\phi)$ is the sum of the kinetic energy, rotation energy and nonlocal Ginzburg--Landau energy:
\begin{equation} \label{intro:nloc:ene}
\begin{aligned}
\mathcal{E}_{\nloc}(\bu,\bo,\phi) &:= \int_{\Omega}\frac{\rho(\phi)}{2} (|\bu|^2+|\bo|^2)\, dx + G_{\nloc}(\phi), \\
G_{\nloc}(\phi) & := e(\phi) + \int_\Omega F(\phi) \, dx \\
e(\phi) &:=  \frac{1}{4}\int_{\Omega}\int_{\Omega}K(x-y)(\phi(x)-\phi(y))^2\, dx \,dy = \frac{1}{2} \int_\Omega a|\phi|^2 - \phi (K \star \phi) \, dx
\end{aligned}
\end{equation}
The mathematical analysis of the nonlocal Model H (nModelH) and the nonlocal AGG (nAGG) model, as well as their variants, have been studied by many authors. We refer to \cite{nlocfrac,SFrigeri1,SFrigeri2,FrigeriGrasselliRocca} for global weak solutions, to \cite{CCGAGMG} for global well-posedness and long-time behavior, and to \cite{AbelsTera,AbelsTeraSing,Hurm} for nonlocal-to-local asymptotics when a suitably scaled family of kernels $(K_\kappa)_{\kappa > 0}$ approaches the Dirac delta distribution as $\kappa \to 0$, see e.g.~\eqref{scaled:kernel}. Our aim in this work is to develop similar analytical results for the nMAGG model, which we furnish with the following standard set of no-slip, no-spin and no-flux boundary conditions if $\Omega$ is a bounded domain:
\begin{align}\label{std:bdy}
\bu = \bm{0}, \quad \bo = \bm{0}, \quad \pdnu \mu = \nabla \mu \cdot \bm{n} = 0 \text{ on } \pd \Omega,
\end{align}
where $\bm{n}$ is the outward unit normal vector field on $\pd \Omega$, or with periodic boundary conditions if $\Omega$ is the torus $\mathbb{T}^d$. We remark that the current state-of-the-art involving the nonlocal-to-local convergence for strong solutions with convergence rates, established in \cite{Hurm}, is only for the Model H, as several key estimates involved in deriving strong solutions for the nAGG model \cite{CCGAGMG} and for the nMAGG model (Section~\ref{sec:str}) are not independent of the scaling parameter $\kappa$ in \eqref{scaled:kernel}. Deriving similar convergence rates for the nonlocal-to-local convergence for strong solutions to the nAGG and nMAGG models will be the subject of future research.

Our main contributions are summarized as follows:
\begin{itemize}
    \item The existence of weak solutions to the nMAGG model in 3D on $[0,T]$ for arbitrary $0 < T < \infty$, extending the result of \cite{SFrigeri1} to the micropolar setting;
    \item The nonlocal-to-local convergence of weak nMAGG solutions to weak nAGG solutions when the convolution kernel $K$ is scaled appropriately, extending the result of \cite{AbelsTera} to the micropolar setting;
    \item The well-posedness of strong solutions to the nMAGG model in 2D on $[0,T]$ for arbitrary $0 < T < \infty$, extending the result of \cite{CCGAGMG} to the micropolar setting;
    \item Consistency estimates between nMAGG and the nonlocal nonpolar variants (nAGG and nModelH) in terms of $\eta_r$, providing a nonlocal analogue of the consistency estimates for local models in \cite{ChanLamstrsol}.
\end{itemize}
We remark that local strong well-posedness for the nAGG and nMAGG models remains an open problem, as the semi-Galerkin approach of \cite{CCGAGMG} and the regularity assumptions on the initial conditions listed in Theorem \ref{thm:strwellposed} seem insufficient in the 3D setting. However, with the latest advancements in maximal $L^p$-regularity theory for the nonlocal Cahn--Hilliard equation, see e.g.~\cite{maxLpregnlocCH}, it seems feasible to establish local strong well-posedness to the nAGG and nMAGG models in 3D when more regular initial data are prescribed. This would be a subject of future research.



The rest of this paper is organized as follows: In Section \ref{sec:prelim} we introduce the notation, problem setting and essential preliminary results. Then, in Section \ref{sec:weak} we establish the global existence of weak solutions for both two and three spatial dimensions, while tackling the nonlocal-to-local asymptotics in a weak sense in Section \ref{sec:weakconv}. Restricting to two spatial dimensions thereafter, in Section \ref{sec:str} we establish the global well-posedness of strong solutions, which in turn allows us to develop consistency estimates between the nMAGG model with other nonlocal nonpolar models in Section \ref{sec:consistency}.

\section{Preliminaries}\label{sec:prelim}

\paragraph{Notation and convention.}
For the subsequent sections we make use of the Einstein summation convention and neglect the basis vector elements. 
\begin{itemize}
\item For a vector $\bv = (v_i)$ and a second order tensor $\bm{A} = (A_{ij})$, the gradient $\nabla \bv$ and divergence $\div \bm{A}$ are defined as
\[
(\nabla \bv)_{ij} = \pd_i v_j = \frac{\pd v_j}{\pd x_i}, \quad (\div \bm{A})_j = \pd_i A_{ij}.
\]
Meanwhile the vector $\bm{v} \cdot \bm{A}$ is defined as
\[
(\bm{v} \cdot \bm{A})_j = v_i A_{ij}.
\]
The Frobenius product $\bm{A}: \bm{B}$ of two second order tensors $\bm{A}$ and $\bm{B}$ is defined as
$\bm{A} : \bm{B} = A_{ij} B_{ij}$. For two vectors $\bm{v} = (v_i)$ and $\bm{w} = (w_i)$ we denote $\bm{v} \otimes \bm{w}$ to be the second order tensor $(\bm{v} \otimes \bm{w})_{ij} = v_i w_j$. The symbol $\mathbb{I}$ will be used to denote the identity matrix, the identity tensor or the identity mapping.
\item In three spatial dimensions, the entries of the third order Levi-Civita tensor $\bm{\eps} = (\eps_{ijk})$ are defined as
\[
\eps_{ijk} = \begin{cases}
1 & \text{ if } (i,j,k) \text{ is } (1,2,3), (2,3,1) \text{ or } (3,1,2), \\
-1 & \text{ if } (i,j,k) \text{ is } (3,2,1), (1,3,2) \text{ or } (2,1,3), \\
0 & \text{ if } i = j, \text{ or } j = k, \text{ or } k = i.
\end{cases}
\]
Then, in three spatial dimensions, the following properties are valid:
\begin{align}\label{eps:prop}
\eps_{ljk} \eps_{mjk} = 2 \delta_{lm}, \quad \eps_{jik} \eps_{jlm} = \delta_{il} \delta_{km} - \delta_{im} \delta_{kl}.
\end{align}
The cross product $\bm{a} \times \bm{b}$ between two vectors $\bm{a}$ and $\bm{b}$, as well as the curl of a vector $\bv = (v_1, v_2, v_3)^{\top}$ are defined as
\[
(\bm{a} \times \bm{b})_j = \eps_{jkl} a_k b_l, \quad (\curl \bv)_j = \eps_{jkl} \pd_k v_l = (\nabla \times \bv)_j.
\]
Then, we have the following integration by parts formula involving the curl operator:
\begin{equation}\label{IBP:curl}
\begin{aligned}
\int_\Omega \curl \bm{a} \cdot \bm{b} \, dx & = \int_{\Omega} \bm{a} \cdot \curl \bm{b} \, dx + \int_{\pd \Omega} (\bm{a} \times \bm{b}) \cdot \bnu \, dS \\
& = \int_{\Omega} \bm{a} \cdot \curl \bm{b} \, dx - \int_{\pd \Omega}  (\bm{a} \times \bnu) \cdot \bm{b} \, dS.
\end{aligned}
\end{equation}
The antisymmetric tensor $\W \bm{v} = \frac{1}{2} (\nabla \bm{v} - (\nabla \bm{v})^{\top})$ satisfies the following relation:
\begin{equation}\label{Wv:curl}
\begin{aligned}
|\curl \bm{v} |^2 & = \eps_{jik} \pd_i v_k \eps_{jlm} \pd_l v_m = (\delta_{il} \delta_{km} - \delta_{im} \delta_{kl}) \pd_i v_k \pd_l v_m \\
& = \pd_i v_k \pd_i v_k - \pd_i  v_k \pd_k v_i = \frac{1}{2}(\pd_i v_k - \pd_k v_i)(\pd_i v_k - \pd_k v_i) \\
& = 2 \W \bm{v} : \W \bm{v}.
\end{aligned}
\end{equation}
\item In two spatial dimensions, the entries of the second order Levi-Civita tensor $\bm{\eps} = (\eps_{ij})$ are defined as
\[
\eps_{ij} = \begin{cases}
1 & \text{ if } (i,j) = (1,2), \\
-1 & \text{ if } (i,j) = (2,1), \\
0 & \text{ if } i = j.
\end{cases}
\]
Then, in two spatial dimensions, the following properties are valid:
\[
\eps_{ij} \eps_{ik} =  \delta_{jk}, \quad \eps_{ij} \eps_{kl} = \delta_{ik} \delta_{jl} - \delta_{il} \delta_{jk}.
\]
The curl of a two dimensional vector $\bm{v} = (v_1, v_2)^{\top}$ is the scalar quantity
\[
\curl_2 \bm{v} = \pd_1 v_2 - \pd_2 v_1,
\]
and we define the rotation vector associated to a scalar $w$ as
\[
\curl_1 w = \begin{pmatrix}
    \pd_2 w \\ - \pd_1 w
\end{pmatrix}.
\]
Then, we have the following integration by parts formula:
\begin{align}\label{IBP:curl:2D}
        \int_\Omega \curl_2 \bm{v} w \, dx = \int_\Omega \bm{v} \cdot \curl_1 w \, dx + \int_{\pd \Omega} w \bm{v} \cdot \bm{n}^{\perp} \, dS,
\end{align}
where $\bm{n}^{\perp} = (-n_2, n_1)^{\top}$ is an unit vector orthogonal to the outer unit normal $\bm{n} = (n_1, n_2)^{\top}$ to $\pd \Omega$. Furthermore, for a scalar $w$ and a vector $\bm{v}$, we have the relation
\begin{equation}\label{2D:curlu:w}
\begin{aligned}
& \int_\Omega 2 \W \bm{v} : \W \bm{v} - 2w \curl_2 \bm{v} - 2\bm{v} \cdot \curl_1 w + 4 w^2 \, dx \\
& \quad = \int_\Omega (\pd_1 v_2 - \pd_2 v_1)^2 - 2 w (\pd_1 v_2 - \pd_2 v_1) - 2 v_1 \pd_2 w + 2v_2 \pd_1 w + 4 w^2 \, dx \\
& \quad = \int_\Omega |\curl_2 \bm{v} - 2 w|^2 \, dx - \int_{\pd \Omega} 2 w \bm{v} \cdot \bm{n}^{\perp} \, dS.
\end{aligned}
\end{equation}
\end{itemize}

\subsection{Equations in the two-dimensional setting}\label{sec:2Dcase}
Following standard convention, see e.g.~\cite{Luka2D}, we restrict the three-dimensional dynamics to the $(x_1,x_2)$-plane, with the flow independent of the $x_3$-coordinate, while having the axes of micro-rotation parallel to the $x_3$-axis. Namely, $\bu$ and $\bo$ reduces to $\bu = (u_1(t,x_1,x_2), u_2(t,x_1,x_2),0)^{\top}$ and $\bo = (0,0,\omega(t,x_1,x_2))^{\top}$, while $\phi$ and $\mu$ are scalar functions of $t$, $x_1$ and $x_2$. In a bounded domain $\Omega \subset \R^2$, with an abuse of notation we write $\bu = (u_1, u_2)^{\top}$, so that the Navier--Stokes component of the two-dimensional model arising from considering only the first and second components of \eqref{local:model:equ:bu}, which reads as:
\begin{align*}
& \pd_{t}(\rho(\phi) \bu) + \div(\rho(\phi) \bu \otimes \bu) - \div\left( 2\eta(\phi)\D\bu + 2\eta_{r}(\phi)\W\bu \right)  \\[1ex] 
& \quad = - \nabla p + \div\left( \rho'(\phi) m(\phi) \nabla \mu \otimes \bu \right)+ \mu\nabla\phi + 2\curl_1\left( \eta_{r}(\phi) \omega \right),
\end{align*}
where $\nabla$ and $\div$ denote the two-dimensional gradient and divergence operators. Then, we take the third component of \eqref{local:model:equ:bo}, which now reads for the scalar quantity $\omega$:
\begin{align*}
& \pd_{t}(\rho(\phi) \omega) + \div(\rho(\phi) \omega \bu ) - \div\left( (c_{d}(\phi) + c_{a}(\phi)) \nabla \omega \right) \\
& \quad = \div \left( \omega \rho'(\phi) m(\phi) \nabla \mu \right) + 2\eta_{r}(\phi)(\curl_2 \bu - 2\omega). 
\end{align*}
It is interesting to note that the term involving $c_0(\phi) \div \bo$ drops out under this two-dimensional setting. For simplicity, we express the sum $c_d + c_a$ as $c_{d,a}$. Hence, the nMAGG model in two dimensions reads as
\begin{subequations}\label{nonlocal:model:equ:2D}
\begin{alignat}{2}
& \label{2D:nonlocal:model:equ:div:0} \div \bu = 0, \\[1ex]
& \label{2D:nonlocal:model:equ:bu} \partial_{t}(\rho(\phi) \bu) + \div(\rho(\phi) \bu \otimes \bu) - \div\left( 2\eta(\phi)\D\bu + 2\eta_{r}(\phi)\W\bu \right)  \\[1ex] 
\notag & \quad = - \nabla p + \div\left( \rho'(\phi) m(\phi) \nabla \mu \otimes \bu \right)\mu\nabla\phi + 2\curl_1\left( \eta_{r}(\phi) \omega \right), \\[1ex]
& \label{2D:nonlocal:model:equ:bo} \partial_{t}(\rho(\phi) \omega) + \div(\rho(\phi) \omega \bu) - \div  (c_{d,a}(\phi) \nabla \omega) \\[2ex]
\notag & \quad = \div ( \omega \rho'(\phi)m(\phi) \nabla \mu) + 2\eta_{r}(\phi)(\curl_2 \bu - 2\bo), \\[1ex] 
& \label{2D:nonlocal:model:equ:CH} \partial_{t}\phi + \bu \cdot \nabla\phi = \div\left( m(\phi)\nabla\mu \right), \\[1ex] 
& \label{2D:nonlocal:model:equ:mu} \mu = \sigma \eps  (a\phi - K \star \phi) + \frac{\sigma}{\eps} F'(\phi). 
\end{alignat}
\end{subequations}

\subsection{Functional setting}

\paragraph{Notation for Banach, Bochner, Lebesgue and Sobolev spaces.} For a real Banach space $X$ and its topological dual $X^*$, the corresponding duality pairing is denote by $\inn{f}{g}_X$ for $f \in X^*$ and $g \in X$. The continuous embedding of $X$ into $Y$ and the compact embedding of $X$ into $Y$ are denoted by $X \subset Y$ and $X \Subset Y$ respectively.

For $1 \leq p \leq \infty$, the Bochner space $L^p(0,T;X)$ denotes the set of all strongly measurable $p$-integrable functions (if $p < \infty$) or essentially bounded functions (if $p = \infty$) on the time interval $[0,T]$ with values in the Banach space $X$. The space $W^{1,p}(0,T;X)$ denotes all $u \in L^p(0,T;X)$ with vector-valued distributional derivative $\frac{du}{dt} \in L^p(0,T;X)$. When $p = 2$ we use the notation $H^1(0,T;X) = W^{1,2}(0,T;X)$. The Banach space of all bounded and continuous functions $u:[0,T] \to X$ is denoted by $C^0([0,T];X)$ and equipped with the supremum norm, while $C_w([0,T];X)$ is the space of bounded and weakly continuous functions $u:[0,T] \to X$, i.e., $\inn{f}{u} : [0,T] \to \R$ is continuous for all $f \in X^*$. We employ the notation $C^\infty_0(0,T;X)$ to denote the vector space of all smooth and compactly support (in time) functions $u:(0,T) \to X$.

Let $\Omega \subset \R^d$, $d\in\{2,3\}$, denote either a bounded domain or the $d$-dimensional torus $\mathbb{T}^d$. We denote by $L^p(\Omega)$ and $W^{k,p}(\Omega)$ for $p\in [1,\infty]$ and $k \in [0,\infty)$ to be the Lebesgue and Sobolev spaces over $\Omega$, respectively. When $p = 2$, these are Hilbert spaces and we use the notation $H^k(\Omega) = W^{k,2}(\Omega)$. The $L^2(\Omega)$-inner product is denoted by $(\cdot,\cdot)$ with associated norm $\| \cdot \|$, while the norm of $W^{k,p}(\Omega)$ is denoted by $\| \cdot \|_{W^{k,p}}$. For vectorial functions, we use $L^p(\Omega;\R^d)$ and $W^{k,p}(\Omega;\R^d)$ to denote the corresponding vectorial Lebesgue and Sobolev spaces. For convenience we use the notation $Q := \Omega \times [0,T]$, $\Sigma := \pd \Omega \times [0,T]$, as well as $(\cdot,\cdot)_Q$ to denote the $L^2(Q)$-inner product, i.e., 
\[
(f,g)_Q := \int_0^T \int_\Omega fg \, dx \, dt.
\]
\paragraph{Spaces with zero mean value.} The generalized mean value of $f$ is defined as
\begin{align*}
\overline{f} = \begin{cases}
\frac{1}{|\Omega|}\int_{\Omega} f \, dx & \text{ if } f\in L^1 (\Omega)\\[1ex]
\frac{1}{|\Omega|}\inn{f}{1}_{H^{1}(\Omega)} & \text{ if } f\in H^1(\Omega)^*,
\end{cases}
\end{align*}
where $|\Omega|$ denotes the $d$-dimensional Lebesgue measure of $\Omega$. We introduce the following spaces of functions with zero generalized mean value:
\[
H^1_{(0)}(\Omega) := \{ f \in H^1(\Omega) \, : \, \overline{f} = 0 \}, \quad H^1_{(0)}(\Omega)^* := \{ f \in H^1(\Omega)^* \, : \, \overline{f} = 0 \},
\]
which are closed linear subspaces of $H^1(\Omega)$ and $H^1(\Omega)^*$, respectively, and hence they are also Hilbert spaces.

\paragraph{Spaces with zero Neumann boundary conditions.} In light of the zero Neumann boundary conditions for $\mu$, we introduce the Hilbert spaces
\[
H^k_n(\Omega) := \{ f \in H^k(\Omega) \, : \, \pdnu f = 0 \text{ on } \pd \Omega \} \text{ for } k \geq 1.
\]

\paragraph{Spaces for velocity and micro-rotation.} For the velocity field, if $\Omega$ is a bounded domain, we define for $p \in (1,\infty)$,
\begin{align*}
\mathbb{L}^p_{\sigma} & := \left \{ \bu \in L^{p}(\Omega;\R^d) \, : \, \div \bu = 0\text{ in } \Omega, \ \bu \cdot \bnu = 0 \text{ on } \partial\Omega \right \}, \\
\mathbb{W}^{1,p}_{\sigma} & : = \left\{ \bu \in W^{1,p}(\Omega;\R^d) \, : \,\div \bu = 0\text{ in } \Omega, \ \bu = \bm{0}\text{ on } \partial\Omega \right\}, 
\end{align*}
as the completion of the space of divergence-free vector fields in $C^\infty_0(\Omega;\R^d)$ with respect to the norms of $L^p(\Omega;
\R^d)$ and $W^{1,p}(\Omega;\R^d)$. Then, for $k \in \N$, $k \geq 2$ we set $\mathbb{W}^{k,p}_\sigma := W^{k,p}(\Omega;\R^d) \cap \mathbb{W}^{1,p}_\sigma$. When $p = 2$, we use the notation
\[
\HH_{\sigma} = \HH^0_{\sigma} := \mathbb{L}^2_\sigma, \quad \HH^k_\sigma := \mathbb{W}^{k,2}_\sigma.
\]
If $\Omega$ is the torus $\mathbb{T}^d$, we modify the definition to 
\begin{align*}
\mathbb{L}^p_{\sigma} & := \left \{ \bu \in L^{p}(\Omega;\R^d) \, : \, \div \bu = 0\text{ in } \Omega, \ \overline{\bu} = 0 \right \}, \\
\mathbb{W}^{1,p}_{\sigma} & : = \left\{ \bu \in W^{1,p}(\Omega;\R^d) \, : \,\div \bu = 0\text{ in } \Omega, \ \overline{\bu} = 0 \right\}.
\end{align*}
For the micro-rotation, if $\Omega$ is a bounded domain, we define for $k \in \N$, $k \geq 2$,
\begin{align*}
\HH = \HH^0 := L^{2}(\Omega;\R^d), \quad \HH^1 := \left \{ \bo \in \HH \, : \, \bo = \bm{0} \text{ on } \partial\Omega \right\}, \quad 
\HH^{k} = H^{k}(\Omega;\R^d) \cap \HH.
\end{align*}
If $\Omega$ is the torus $\mathbb{T}^d$, we modify the definition to $\HH^k = H^k(\Omega;\R^d)$.

\paragraph{The Stokes operator.}
We define the Stokes operator $\bm{A}_S : \HH^1_{\sigma} \to (\HH^1_{\sigma})^*$ by
\[
\langle \bm{A}_S \bu, \bm{v} \rangle_{\HH^1_{\sigma}} := (\nabla \bu, \nabla \bm{v}) \quad \forall \bm{v} \in \HH^1_{\sigma}.
\]
It is a continuous linear isomorphism with domain $D(\bm{A}_S) = \HH^2_{\sigma}$, and using the Leray projection $\mathbb{P} : L^2(\Omega;\R^d) \to \HH_\sigma$, we have the representation $\bm{A}_S = - \mathbb{P} \Delta$ with $\Delta$ as the standard (vectorial) Laplace operator. Its inverse $\bm{A}_S^{-1} : \HH_\sigma \to \HH_\sigma$ is a self-adjoint compact operator, and by spectral decomposition we can define fractional operators $\bm{A}_S^q$ for all $q \in \R$ with domains $D(\bm{A}_S^{q}) \subset H^{2q}(\Omega;\R^d)$. Furthermore, for $\bu, \bm{v} \in \HH_\sigma$, we set
\[
(\bu, \bm{v})_\sigma := (\nabla \bm{A}_S^{-1} \bu, \nabla \bm{A}_S^{-1} \bm{v}),
\]
which defines an inner product on $(\HH^1_\sigma)^*$. The induced norm 
\[
\| \bu \|_\sigma := \| \nabla \bm{A}_S^{-1} \bu \|
\]
is equivalent to the standard norm on $(\HH^1_\sigma)^*$. By Poincar\'e's inequality there exists a positive constant $C_{S,1}$ such that 
\begin{align}\label{Stokes:Poin}
    \| \bm{A}_S^{-1} \bu \|_{\HH^1_{\sigma}} \leq C_{S,1} \| \bu \|_{\sigma} \quad \forall \bu \in \HH_{\sigma},
\end{align}
and by the regularity theory involving the Stokes operator, there exists a positive constant $C_{S,2}$ such that 
\begin{align}\label{Stokes:reg}
    \| \bm{A}_S^{-1} \bu \|_{\HH^2_\sigma} \leq C_{S,2} \| \bu \| \quad \forall \bu \in \HH_\sigma.
\end{align}

\paragraph{The Laplace operator.} For the micro-rotation we use the notation $\bm{A}_1$ as the vectorial Dirichlet--Laplacian with domain $D(\bm{A}_1) = \HH^2$ if $\Omega$ is a bounded domain, and as the vectorial Laplacian with periodic boundary conditions if $\Omega$ is the torus $\mathbb{T}^d$. It is clear that its inverse $\bm{A}_1^{-1} : \HH \to \HH$ is a self-adjoint compact operator. Analogous to the setting of the Stokes operator, the equivalent $(\HH^1)^{*}$ norm is given by 
\begin{align}
 \|\bo\|_{\#} := \|\nabla \bA^{-1}_{1} \bo\|
\end{align}
and the $\HH^1$- and the $\HH^2$-estimates for $\bA^{-1}_{1} \bo$ are given by
\begin{align}
&\|\bA^{-1}_{1} \bo\|_{\HH^1} \leq C_{1,1} \| \bo \|_{\#} \quad \forall \bo \in (\HH^1)^*, \\
&\|\bA^{-1}_{1} \bo\|_{\HH^2} \leq C_{1,2} \| \bo \| \quad \forall \bo \in \HH,
\end{align}
for some constant $C_{1,1}, C_{1,2} > 0$.
For the phase field variable we need the scalar Neumann Laplacian $A_{N}$ defined as
\[
\langle A_N f, g \rangle_{H^1(\Omega)} := (\nabla f, \nabla g) \quad \forall g \in H^1(\Omega)
\]
with domain $D(A_N) = H^2_n(\Omega)$. Its inverse $A_{N}^{-1} : H^1_{(0)}(\Omega)^* \to H^1(\Omega)$ provides an equivalent norm $\sqrt{ \| f - \overline{f}\|_*^2 + |\overline{f}|^2}$ for any $f \in H^1_{(0)}(\Omega)^*$, where
\begin{align}
 \|h\|_{*} := \|\nabla A^{-1}_{N} h\| \quad \forall h \in H^1_{(0)}(\Omega)^*,
\end{align}
to the usual $H^1(\Omega)^*$-norm. Moreover, analogous $H^{1}(\Omega)$- and $H^{2}(\Omega)$- estimates for $A_N^{-1}h$ are given by
\begin{align}
& \| A^{-1}_{N} h\|_{H^1} \leq C_{N,1} \|h\|_{*} \quad \forall h \in H^1_{(0)}(\Omega)^*, \\
& \| A^{-1}_{N} h\|_{H^2} \leq C_{N,2} \|h\| \quad \forall f \in L^2(\Omega) \text{ and } \overline{f} = 0,
\end{align}
for some constants $C_{N,1}, C_{N,2} > 0$.\\

\noindent We introduce the unbounded linear scalar operator $B : H^1(\Omega) \to H^1(\Omega)^*$ by 
\[
\langle B f, g \rangle_{H^1(\Omega)} := (\nabla f, \nabla g) + (f,g).
\]
If $\Omega$ is a bounded domain, then we may identify $B$ as the operator $-\Delta_N + \mathbb{I}$ with domain $D(B) = H^2_n(\Omega)$. Its inverse $B^{-1} : L^2(\Omega) \to L^2(\Omega)$ is a self-adjoint compact operator on $L^2(\Omega)$.

\subsection{Main assumptions}\label{allass}
\begin{enumerate}[label=$(\mathrm{A \arabic*})$, ref = $\mathrm{A \arabic*}$]
\item \label{gen:ass1} $\Omega \subset \R^{d}$, $d \in \{2,3\}$, is either a bounded domain with $C^\infty$ boundary or the torus $\mathbb{T}^d = (\R / \mathbb{Z})^d$.
\item \label{gen:ass2} The kernel $K$ is symmetric, i.e., $K(x) = K(-x)$ for all $x \in \R^d$, and satisfies
$K \in W^{1,1}(\R^d)$ with $a(x) := \int_{\Omega}K(x - y)\, dy = (K \star 1)(x) \geq 0$ for a.e.~$x \in \Omega$.
\item \label{gen:ass3} The mobility function $m$ and the viscosity functions $\eta$, $\eta_r$ $c_{0}$, $c_{d}$, $c_{a}$ belong to $W^{1,\infty}(\R)$ and satisfy 
\begin{equation*}
f_* \leq f(s) \leq f^* \quad \forall s\in \R,
\end{equation*}
for $f \in \{ m, \eta, c_0, c_d, c_a\}$ where $f_*$ and $f^*$ are positive constants, while $\eta_r$ satisfies
\[
\eta_{r,*} \leq \eta_r(s) \leq \eta_{r}^* \quad \forall s \in \R,
\]
with $\eta_r^* > 0$ and $\eta_{r,*} \geq 0$.
\end{enumerate}
Having $\eta_{r,*} = 0$ allows for the possibility that one of the constituent fluid in the mixture can be nonpolar.  Concerning the potential function $F$, we always have in mind the logarithmic double-well potential $F$ written as the decomposition $F = F_{1} + F_{2}$, where
\begin{align}\label{log:pot}
F_{1}(s) = \frac{\theta}{2}\left( (1 + s)\log(1 + s) + (1 - s)\log(1 - s) \right), \quad F_{2}(s) = - \frac{\theta_{c}}{2}s^{2}
\end{align}
with constants $0 < \theta < \theta_c$. It can be shown that \eqref{log:pot} fulfills the following assumptions:

\begin{enumerate}[label=$(\mathrm{B \arabic*})$, ref = $\mathrm{B \arabic*}$]
\item \label{sing:pot:ass1} $F_{1} \in C^{p}(-1,1)$ for some fixed even integer $p \geq 4$ and $F_{2} \in C^{2,1}([-1,1])$ such that 
\[
\lim_{s \to \pm 1} F'(s) = \pm \infty.
\]
We extend $F(s) = +\infty$ for $s \notin [-1,1]$, and without loss of generality, we set $F(0) = 0$, $F'(0) = 0$, and so $F(s) \geq 0$ for all $s \in [-1,1]$.
\item \label{sing:pot:ass2} There exist constants $\delta_{0} > 0$ and $\eps_{0} > 0$ such that
\begin{equation*}
F_{1}^{(p)}(s) := \frac{d^p F_1}{d s^p}(s) \geq \delta_{0} \quad \forall s \in ( - 1, - 1 + \eps_{0}] \cup [ 1 - \eps_{0},1 ),
\end{equation*}
and $F_1^{(p)}$ is non-decreasing in $[1-\eps_0,1)$ and non-increasing in $(-1,-1+\eps_0]$. 

Furthermore, there exists $\tilde{\eps}_{0} > 0$ such that, for each $k = 0,1,\dots,p$ and each $j = 0,1,\dots,\frac{p - 2}{2}$
\begin{subequations}
\begin{alignat*}{3}
& F_{1}^{(k)}(s) \geq 0 &&\quad \forall s \in [1 - \tilde{\eps}_{0},1 ) \\[1ex]
& F_{1}^{(2j + 2)}(s) \geq 0, \quad F_{1}^{(2j + 1)}(s) \leq 0 && \quad \forall s \in ( - 1, - 1 + \tilde{\eps}_{0} ].
\end{alignat*}
\end{subequations}
\item \label{sing:pot:ass5} There exists a constant $\hat{\alpha} > 0$ such that
\begin{equation*}
F''(s) + a(x) \geq \hat{\alpha} \quad \forall s \in ( -1,1) \text{ and for a.e.~} x \in \Omega.
\end{equation*}
\item \label{sing:pot:ass6} There exists $\eps_1 > 0$, $r \in [2,\infty)$ and a continuous function $h:(0,1) \to \R_+$, $h(\delta) = o(\delta^{4/r})$ as $\delta \to 0^+$ such that $F''$ is monotone non-decreasing on $[1-\eps_1, 1)$ and
\begin{subequations}
\begin{alignat*}{2}
& F''(s) \leq Ce^{C\left| F'(s) \right|^{\beta}} && \quad \forall s \in ( - 1,1), \\[1ex]
& F''(1 - 2\delta)h(\delta) \geq 1 &&\quad  \forall\delta \leq \frac{\eps_1}{2},
\end{alignat*}
\end{subequations}
for some constant $C>0$ and $\beta \in [1,2)$.
\end{enumerate}


For the nonlocal-to-local convergence, we make further assumptions on the structure of the kernel $K$, which we take as
\begin{align}\label{scaled:kernel}
K_\kappa(x) := \frac{\gamma_\kappa(|x|)}{|x|^2} \quad \forall x \in \Omega, \ \forall \kappa > 0,
\end{align}
where $\gamma_\kappa$ satisfies
\begin{enumerate}[label=$(\mathrm{C \arabic*})$, ref = $\mathrm{C \arabic*}$]
\item \label{nonl:to:loc:weak:ass} For any $\kappa > 0$, let $\gamma_\kappa \in L^1_{\loc}(\R; [0,\infty))$ be a family of functions satisfying $\gamma_\kappa(-r) = \gamma_\kappa(r)$ for all $r \in \R$ and 
\begin{equation*}
\begin{aligned}
\int_{0}^{+ \infty}{\gamma_{\kappa}(r)r^{d - 1}}dr = \frac{2}{C_{d}} & \quad \forall \kappa >0, \\ 
 \lim_{\kappa \rightarrow 0^{+}} \int_{\delta}^{+ \infty}\gamma_{\kappa}(r)r^{d - 1}dr = 0 &  \quad \forall \delta > 0,
\end{aligned}
\end{equation*}
where the constant $C_d$ is defined as the following integral over the $(d-1)$-dimensional unit sphere $\mathbb{S}^{d-1}$:
\[
C_{d} = \int_{\mathbb{S}^{d - 1}}^{}\left| \bm{e}_{1} \cdot \sigma \right|^{2}d\mathcal{H}^{d - 1}(\sigma),
\]
with $(d-1)$-dimensional Hausdorff measure $\mathcal{H}^{d-1}$ and $\bm{e}_1 = (1, 0, \dots,0)^{\top} \in \R^{n}$ is the first canonical unit vector.
\end{enumerate}
If $\Omega$ is the torus $\mathbb{T}^d$, we further assumption that $\gamma_\kappa$ is compactly supported in $[0,1)$ for all $\kappa > 0$.

\section{Global weak existence}\label{sec:weak}
Since the values of $\sigma$ and $\eps$ has no significance in the forthcoming analysis, we set their values to equal 1. The main result of this section is the existence of global-in-time weak solutions to the nMAGG model which in strong formulation reads as
\begin{subequations}\label{nonl:model:equ}
\begin{alignat}{2}
& \label{3D;nonlocal:model:equ:div:0} \div \bu = 0, \\[1ex]
& \label{3D;nonlocal:model:equ:bu} \partial_{t}(\rho(\phi) \bu) + \div(\rho(\phi) \bu \otimes \bu) - \div\left( 2\eta(\phi)\D\bu + 2\eta_{r}(\phi)\W\bu \right)\\[1ex] 
\notag & \quad = -  \nabla p + \div\left( \tfrac{\overline{\rho}_1 - \overline{\rho}_2}{2} m(\phi) \nabla \mu \otimes \bu  \right) + \mu\nabla\phi + 2\curl\left( \eta_{r}(\phi)\bo \right), \\[1ex]
& \label{3D:nonlocal:model:equ:bo} \partial_{t}(\rho(\phi) \bo) + \div(\rho(\phi)  \bu \otimes \bo) - \div\left( c_{0}(\phi)(\div \bo)\mathbb{I} + 2c_{d}(\phi)\D\bo + 2c_{a}(\phi)\W\bo \right) \\[1ex] 
\notag & \quad = \div\left( \tfrac{\overline{\rho}_1 - \overline{\rho}_2}{2} m(\phi) \nabla \mu \otimes \bo \right) + 2\eta_{r}(\phi)(\curl \bu - 2\bo), \\[1ex] 
& \label{3D:nonlocal:model:equ:CH} \partial_{t}\phi + \bu \cdot \nabla\phi = \div\left( m(\phi)\nabla\mu \right), \\[1ex] 
& \label{3D:nonlocal:model:equ:mu} \mu = a \phi - K \star \phi + F'(\phi),
\end{alignat}
\end{subequations}
furnished with no-slip $\bu = \bm{0}$, no-spin $\bo = \bm{0}$ and no-flux $m(\phi) \pdnu \mu = 0$ boundary conditions on $\pd \Omega$. 

\begin{thm}\label{thm:nMAGG:weaksol}
Suppose that \eqref{gen:ass1}--\eqref{gen:ass3} and \eqref{sing:pot:ass1}--\eqref{sing:pot:ass5} hold for some fixed even integer $p \geq 4$. Let $\bu_{0} \in \mathbb{H}_{\sigma}$, $\bo_{0}\in \mathbb{H}$ and $\phi_{0} \in L^{\infty}(\Omega)$ such that $F( \phi_{0} ) \in L^{1}(\Omega)$ and $| \overline{\phi_{0}}| < 1$. Then, for all $T > 0$, there exists a weak solution $(\bu,\bo,\phi, \mu)$ to the nMAGG model on the time interval $[0,T]$ corresponding to initial data $(\bu_{0},\bo_{0},\phi_{0})$ in the following sense:
\begin{itemize}
\item Regularity
\begin{subequations} \label{nloc:weak:sol:reg}
\begin{alignat}{2}
& \bu \in C_{w} ([0,T];\HH_\sigma) \cap L^{2}( 0,T;\HH_\sigma^{1}) \text{ with } \partial_{t}(\rho(\phi) \bu) \in L^{\frac{4}{3}}( 0,T; D(\bm{A}_S)^{*} ), \\[1ex]
& \bo \in C_{w}([0,T];\HH ) \cap L^{2}( 0,T;\HH^{1}) \text{ with } \partial_{t}(\rho(\phi) \bo) \in L^{\frac{4}{3}}( 0,T; D(\bm{A}_1)^*) ),\\[1ex]
& \phi \in L^{\infty}( 0,T;L^{p}(\Omega)) \cap L^{2} ( 0,T;H^{1}(\Omega)) \text{ with } \partial_{t}\phi \in L^{2}( 0,T;H^{1}(\Omega)^{*} ), \\[1ex]
& \phi \in L^{\infty}(Q) \text{ such that } | \phi(x,t) | < 1 \text{ a.e.~} (x,t) \in Q, \\[1ex]
& \mu \in L^{2}( 0,T;H^{1}(\Omega)).
\end{alignat}
\end{subequations}
\item Equations
\begin{subequations}\label{weak:nloc:1}
\begin{alignat}{2}
& \label{weak:form:bu:sec2} \left\langle \partial_{t}(\rho(\phi) \bu),\bv \right\rangle_{D(\bm{A}_S)} - (\rho(\phi) \bu \otimes \bu,\nabla \bv) \\[1ex]
\notag & \qquad + \left( 2\eta(\phi)\D\bu,\D\bv \right) + \left( 2\eta_{r}(\phi)\W\bu,\W\bv \right) - \left( \bm{J} \otimes \bu ,\nabla \bv \right) \\[1ex] 
\notag  & \quad  = - (\phi\nabla\mu,\bv) + \left( 2\curl (\eta_{r}(\phi)\bo),\bv \right), \\[1ex]
& \label{weak:form:bo:sec2} \left\langle \partial_{t}(\rho(\phi) \bo),\bz \right\rangle_{D\left(\bm{A}_1\right)} - (\rho(\phi) \bu \otimes \bo,\nabla \bz) + \left( c_{0}(\phi)\div \bo,\div \bz \right) \\[1ex]
\notag & \qquad + \left( 2c_{d}(\phi)\D\bo,\D\bz \right) + \left( 2c_{a}(\phi)\W\bo,\W\bz \right) - \left( \bm{J} \otimes \bo, \nabla \bz \right) \\[1ex] 
\notag & \quad = \left( 2\eta_{r}(\phi)\curl\bu,\bz \right) - \left( 4\eta_{r}(\phi)\bo,\bz \right), \\[1ex]
\label{weak:form:phi} & \left\langle \partial_{t}\phi,\psi \right\rangle_{H^1(\Omega)} + \left( m(\phi)\nabla\mu,\nabla\psi \right) = (\phi \bu,\nabla\psi),
\end{alignat}
\end{subequations}
for all $\psi \in H^{1}(\Omega)$, $\bm{v} \in D(\bm{A}_S)$, $\bz \in D(\bm{A}_1)$ and for almost any $t \in (0,T)$, along with 
\begin{subequations}\label{weak:nloc:2}
\begin{alignat}{2}
\label{weak:nloc:mu} \mu & = a \phi - K \star \phi + F'(\phi) && \quad \text{ a.e.~in } Q, \\[1ex]
\label{weak:nloc:J} \bm{J} & = -\tfrac{\overline{\rho}_1 - \overline{\rho}_2}{2} m(\phi) \nabla \mu && \quad \text{ a.e.~in } Q, \\[1ex]
\label{weak:nloc:rho} \rho(\phi) & = \tfrac{\overline{\rho}_1 - \overline{\rho}_2}{2} \phi + \tfrac{\overline{\rho}_1+ \overline{\rho}_2}{2} &&  \quad \text{ a.e.~in } Q.
\end{alignat}
\end{subequations}
\item Energy inequality
\begin{equation} \label{nloc:ene:ineq:sec2}
\begin{aligned}
& \mathcal{E}_{\nloc}(\bu(t), \bo(t), \phi(t)) \\
& \qquad  + \int_s^t \int_\Omega m(\phi) |\nabla \mu|^2 + 2 \eta(\phi) |\D \bu|^2 + 4 \eta_r(\phi)  | \tfrac{1}{2} \curl u - \bo |^2  \, dx \,  d\tau \\
& \qquad + \int_s^t \int_\Omega 2 c_0(\phi) |\div \bo|^2 + 2 c_d(\phi) |\D \bo|^2 + 2 c_a(\phi)|\W \bo|^2 \, dx \, d \tau \\
& \quad \leq \mathcal{E}_{\nloc}(\bu(s), \bo(s), \phi(s))
\end{aligned}
\end{equation}
for almost all $s \in [0,T)$, including $s = 0$, and for all $t \in [s,T)$, where $\mathcal{E}_{\nloc}$ is defined in \eqref{intro:nloc:ene}.
\item 
  Initial conditions
\begin{equation}
\left( \bu,\bo,\phi \right)|_{t=0} = \left( \bu_{0},\bo_{0},\phi_{0} \right)
\end{equation}
\end{itemize}
\end{thm}
\begin{remark}
The regularities stated in the above theorem are for the three-dimensional case. Analogous global weak existence also hold in the two-dimensional setting, where one would obtain slightly better regularity for the time derivatives:
\[
\pd_t(\rho(\phi) \bu) \in L^{2-s}(0,T;D(\bm{A}_S)^*), \quad \pd_t(\rho(\phi) \omega) \in L^{2-s}(0,T;H^{-1}(\Omega))
\]
for $0 < s \leq 1$. Note that the micro-rotation $\bo$ is a scalar $\omega$ in the two-dimensional model.
\end{remark}

We mainly follow the ideas in \cite{SFrigeri1} that employ a two-level approximation, which involves first approximating the singular potential $F$ by a sequence of regular potentials $F_{\epsilon} : \R \to \R$, and inserting additional terms into the nMAGG model. This approximate system reads in strong formulation as
\begin{subequations}\label{app:model1}
\begin{alignat}{2}
\div \bu  & = 0, \\[1ex]
\pd_t \phi + \bu \cdot \nabla \phi & = \div (m(\phi) \nabla \mu), \label{app:model:phi} \\[1ex]
\mu & = a\phi - K \star \phi  + F_{\epsilon}'(\phi) - \delta \Delta \phi, \label{app:model:mu} \\[1ex]
\pd_t (\rho \bu) + \div (\rho \bu \otimes \bu) & = - \nabla p + \mu \nabla \phi +  \div (2 \eta(\phi) \D \bu + 2 \eta_r(\phi) \W \bu) + \delta \bm{A}_S^3 \bu\label{app:model:u} \\[1ex]
\notag & \quad + 2 \curl(\eta_r(\phi) \bo) + \div (\rho'(\phi) m(\phi) \nabla \mu \otimes \bu )+ \frac{1}{2} R\bu  \\[1ex]
\pd_t (\rho \bo) + \div (\rho \bu \otimes  \bo ) & = \div (c_0(\phi) (\div \bo) \I + 2c_d(\phi) \D \bo +2 c_a(\phi) \W \bo)  \label{app:model:w}  \\[1ex]
\notag & \quad  + 2 \eta_r(\phi)(\curl \bv - 2 \bo) + \div (\rho'(\phi) m(\phi) \nabla \mu \otimes \bo) \\
\notag & \quad +  \frac{1}{2} R \bo + \delta \bm{A}_1^3 \bo,
\end{alignat}
\end{subequations}
where $R = - m(\phi) \nabla \mu \cdot \nabla \rho'(\phi)$. The additional terms involving $R$ in \eqref{app:model:u} and \eqref{app:model:w} arise from the preservation of the natural energy identity when $\phi$ is no longer confined to the physically relevant interval $[-1,1]$, see e.g.~\cite{CHLS} for further details. Notice that when $\phi \in [-1,1]$, we obtain $\rho'(\phi) = \frac{\overline{\rho}_1 - \overline{\rho}_2}{2}$ is a constant, and thus $R = 0$.  In order to deal with this nonlinear term, we introduce regularisation terms $-\delta \Delta \phi$ in \eqref{app:model:mu}, $\delta \bm{A}_S^3 \bu$ in \eqref{app:model:u} and $\delta \bm{A}_1^3 \bo$ in \eqref{app:model:w}.

If $\Omega$ is a bounded domain, then due to the presence of $\Delta \phi$ in \eqref{app:model:mu} we also need to consider a boundary condition for $\phi$, which we take as the no-flux boundary condition $\pdnu \phi = 0$ on $\pd \Omega$, in addition to the no-slip $\bu = \bm{0}$, no-spin $\bo = \bm{0}$ and no-flux $m(\phi)\pdnu \mu = 0$ boundary conditions.

We further consider a sequence of functions $\phi_{0,\delta} \in H^2_n(\Omega)$ defined as $\phi_{0,\delta} := (\mathbb{I} + \sqrt{\delta} B)^{-1} \phi_0$, where we recall $B = - \Delta + \mathbb{I}$.  Then, we furnish the initial data $(\bu_0, \bo_0, \phi_{0,\delta})$ to the approximate system \eqref{app:model1}.  

The proof of Theorem \ref{thm:nMAGG:weaksol} proceeds as follows:
\begin{itemize}
\item Step 1: Establish the existence of a global weak solution $(\bu_{\epsilon,\delta}, \bo_{\epsilon,\delta}, \phi_{\epsilon,\delta}, \mu_{\epsilon,\delta})$ to \eqref{app:model1}.
\item Step 2: Pass to the limit $\epsilon \to 0$ to obtain the singular potential $F$, and a family of functions $(\bu_\delta, \bo_\delta, \phi_\delta, \mu_\delta)_{\delta>0}$ such that $|\phi_\delta| < 1$ a.e.~in $Q$. Consequently, the terms involving $R = - m(\phi) \nabla \mu \cdot \nabla \rho'(\phi)$ in \eqref{app:model:u} and \eqref{app:model:w} vanish.
\item Step 3: Pass to the limit $\delta \to 0$ to obtain a weak solution to the nMAGG model \eqref{nonl:model:equ}.
\end{itemize}

Let us remark that since the equation for the micro-rotation $\bo$ is structurally similar to the equation for the velocity $\bu$, both variables share similar regularities in space and in time. Along with the fact that $\bo$ does not appear explicitly in the convective Cahn--Hilliard component for $(\phi, \mu)$, the arguments in Steps 2 and 3 are almost identical to those in \cite{SFrigeri1} for the nAGG model, while Step 1 requires some minor modifications that are displayed below.  We focus only on the three-dimensional case where $\Omega$ is a bounded domain, as the arguments for the corresponding two-dimensional setting and for the case where $\Omega = \mathbb{T}^d$ are similar.

\begin{remark}
One possibility to have micro-rotational influence of the convective Cahn--Hilliard component for $(\phi, \mu)$ is to consider a dependence on the micro-rotation magnitude in the mobility function, i.e., $m = m(\phi, |\bo|)$. The analysis presented below can be adapted to this setting with minor modifications in light of the boundedness and positivity assumption \eqref{gen:ass3} on $m$ and the strong (hence almost everywhere) convergence involving approximations of $\bo$.
\end{remark}

To establish the existence of weak solutions for the approximate model \eqref{app:model1} it is sufficient to impose the following assumptions of $F_{\epsilon}$ and on the density function $\rho$.

\begin{enumerate}[label=$(\mathrm{D \arabic*})$, ref = $\mathrm{D \arabic*}$]
\item \label{ass:rho:approx} The mass density function $\rho \in C^2(\R)$ is bounded from below by a positive constant $\rho_0$ and satisfies $\rho, \rho', \rho'' \in L^\infty(\R)$, along with 
\[
\rho(s) = \frac{\overline{\rho}_1 - \overline{\rho}_2}{2} s + \frac{\overline{\rho}_1 + \overline{\rho}_2}{2} \quad \text{ for } s \in [-1,1].
\]
\item \label{reg:pot:ass1} The approximate potential function $F_{\epsilon} \in C^{4,1}(\R)$ and there exist constants $\widehat{c}_0 > 0$, $\widehat{c}_1 > 0$, $\widehat{c}_2 > 0$, $\widehat{c}_3 > 0$ and $\widehat{c}_4 \geq 0$, as well as exponents $q \geq 3$ and $r \in (1,2]$ such that 
\begin{align*}
F_{\epsilon}''(s) + a(x) \geq \widehat{c}_0 & \quad \forall s \in \R, \quad \text{ for a.e.~} x \in \Omega, \\
F_{\epsilon}''(s) + a(x) \geq \widehat{c}_1 |s|^{q-2} - \widehat{c}_2 & \quad \forall s \in \R, \quad \text{ for a.e.~} x \in \Omega, \\
|F_{\epsilon}'(s)|^r \leq \widehat{c}_3 |F_{\epsilon}(s)| + \widehat{c}_4 & \quad \forall s \in \R.
\end{align*}
\end{enumerate}
Note that the last condition of \eqref{reg:pot:ass1} implies that $F_{\epsilon}$ has polynomial growth of order $r' = \frac{r}{r-1} \in [2,\infty)$. We refer to \eqref{polyapprox:F} below for a possible choice of $F_{\epsilon}$ that fulfills \eqref{reg:pot:ass1}. We further collect some useful preliminary results.

\begin{lem}[Lemma 2 of \cite{SFrigeri1}]\label{lem:Cw}
Let $X$ and $Y$ be two Banach spaces such that $Y \hookrightarrow X$ and $X^* \hookrightarrow Y^*$ densely. Then, $L^\infty(0,T;Y) \cap C^0([0,T];X) \hookrightarrow C_w([0,T];Y)$.
\end{lem}

\begin{lem}[Lemma 3 of \cite{SFrigeri1}]\label{lem:Energyineq}
Let $\mathcal{E} : [0,T) \to \R$, $0 < T \leq \infty$, be a lower semicontinuous function and let $\mathcal{D} : (0,T) \to \R$ be an integrable function. Assume that 
\[
\mathcal{E}(0) w(0) + \int_0^T \mathcal{E}(r) w'(r) \, dr \geq \int_0^T \mathcal{D}(r) w(r) \, dr
\]
holds for all $w \in W^{1,1}(0,T)$ with $w(T) = 0$ and $w \geq 0$.  Then, 
\[
\mathcal{E}(t) + \int_s^t \mathcal{D}(r) \, dr \leq \mathcal{E}(s)
\]
for almost all $s \in [0,T)$, including $s = 0$, and for all $t \in [s,T)$.
\end{lem}

\begin{lem}\label{lem:exist:weaksol}
Suppose that \eqref{gen:ass1}-\eqref{gen:ass3} and \eqref{ass:rho:approx}-\eqref{reg:pot:ass1} hold. Let $\bu_{0} \in \HH_\sigma$, $\bo_{0}\in\HH$ and $\phi_{0,\delta} \in H^1(\Omega)$ such that $F_{\epsilon}(\phi_{0,\delta}) \in L^{1}(\Omega)$. Then, for any $T > 0$ there exists a weak solution $(\bu, \bo, \phi, \mu)$ to \eqref{app:model1} corresponding to initial data $(\bu_0, \bo_0, \phi_{0,\delta})$ such that the following properties are satisfied:
\begin{itemize}
\item Regularity
\begin{equation}\label{Gal:regularity}
\begin{aligned}
\bu & \in C_w([0,T];\HH_\sigma) \cap L^2(0,T;D(\bm{A}_S^{3/2})) \text{ with } \pd_t (\rho \bu)  \in L^{\frac{30}{29}}(0,T;D(\bm{A}_S^{3/2})^*), \\
\bo & \in C_w([0,T];\HH) \cap L^2(0,T;D(\bm{A}_1^{3/2})) \text{ with } \pd_t (\rho \bo) \in L^{\frac{30}{29}}(0,T;D(\bm{A}_1^{3/2})^*), \\
\phi &\in L^\infty(0,T;H^1(\Omega)) \cap L^2(0,T;H^2_n(\Omega)) \text{ with } \pd_t \phi  \in L^2(0,T;H^1(\Omega)^*), \\
\mu & \in L^2(0,T;H^1(\Omega)).
\end{aligned}
\end{equation}
\item Equations
\begin{subequations}\label{appsys:weakform}
\begin{alignat}{2}
0 & = \langle \pd_t (\rho \bu), \bm{v} \rangle_{D(\bm{A}_S^{3/2})}  - (\rho \bu \otimes \bu), \nabla \bm{v}) + (2 \eta(\phi) \D \bu, \D \bm{v})  \\
\notag & \quad + (2 \eta_r(\phi) \W \bu, \W \bm{v}) + \delta (\bm{A}_S^{3/2} \bu, \bm{A}_S^{3/2} \bm{v}) - (\bm{J} \otimes \bu, \nabla \bm{v}) \\
\notag & \quad + (\phi \nabla \mu, \bm{v}) - (2 \curl(\eta_r(\phi) \bo), \bm{v}) - \tfrac{1}{2}(R\bu, \bm{v}), \\[1ex]
0 & = \langle \pd_t (\rho \bo), \bm{z} \rangle_{D(\bm{A}_1^{3/2})} - (\rho \bu \otimes \bo, \nabla \bm{z}) + (c_0(\phi) \div \bo, \div \bm{z}) \\
\notag & \quad + (2 c_d(\phi) \D \bo, \D \bm{z}) + (2 c_a(\phi) \W \bo, \W \bm{z}) - (\bm{J} \otimes \bo, \nabla \bm{z}) \\
\notag & \quad - (2 \eta_r(\phi)(\curl \bu - 2 \bo), \bm{z}) - \tfrac{1}{2}(R \bo, \bm{z}) + \delta (\bm{A}_1^{3/2} \bo, \bm{A}_1^{3/2} \bm{z}), \\[1ex]
0 & = \langle \pd_t \phi , \psi \rangle_{H^1(\Omega)} + (m(\phi) \nabla \mu, \nabla \psi) - (\phi \bu, \nabla \psi), 
\end{alignat}
\end{subequations}
holding for a.e.~$t \in (0,T)$ and for all $\bm{v} \in D(\bm{A}_S^{3/2})$, $\bm{z} \in D(\bm{A}_1^{3/2})$, $\psi \in H^1(\Omega)$, along with
\begin{equation*}
\begin{alignedat}{2}
\mu & = a \phi - K \star \phi + F_{\epsilon}'(\phi) - \delta \Delta \phi && \quad \text{ a.e.~in } Q, \\
\bm{J} & = - \rho'(\phi) m(\phi) \nabla \mu && \quad \text{ a.e.~in } Q, \\
R & = - \rho''(\phi) m(\phi) \nabla \mu \cdot \nabla \phi && \quad \text{ a.e.~in } Q.
\end{alignedat}
\end{equation*}
\item Energy inequality
\begin{equation}\label{ene.ineq:weaksol}
\begin{aligned}
&  \int_\Omega \frac{\rho(\phi(t))}{2} (|\bu(t)|^2 + |\bo(t)|^2) + F_{\epsilon}(\phi(t)) + \frac{\delta}{2} |\nabla \phi(t)|^2 \, dx + e(\phi(t))\\ 
& \qquad + \int_{0}^{t} \int_\Omega m(\phi) |\nabla \mu|^2 + 2 \eta(\phi) |\D \bu|^2 + c_0(\phi) |\div \bo|^2 \, dx \, d\tau \\
& \qquad + \int_0^t \int_\Omega 2 c_d(\phi) |\D \bo|^2 + 2 c_a(\phi) |\W \bo|^2 + 4 \eta_r(\phi) |\tfrac{1}{2} \curl \bu - \bo |^2 \, dx \, d\tau \\
& \qquad + \delta \int_0^t \| \bm{A}_S^{3/2} \bu \|^2 + \| \bm{A}_1^{3/2} \bo \|^2 \, d \tau \\
& \quad \leq \int_\Omega \frac{\rho(\phi_{0,\delta})}{2} (|\bu_0|^2 + |\bo_0|^2) + F_{\epsilon}(\phi_{0,\delta}) + \frac{\delta}{2} |\nabla \phi_{0,\delta}|^2 \, dx + e(\phi_{0,\delta})
\end{aligned}
\end{equation}
holds for almost all $t \in (0,T)$.
\item Initial conditions
\[
(\bu, \bo, \phi) \vert_{t=0} = (\bu_0, \bo_0, \phi_{0,\delta}).
\]
\end{itemize}
\end{lem}

\begin{proof}
We consider a Galerkin approximation with the family $\{\bm{X}_j\}_{j \in \N}$ of eigenfunctions of the operator $\bm{A}_S$, the family $\{\bm{Y}_j\}_{j \in \N}$ of eigenfunctions of the operator $\bm{A}_1$, and the family $\{z_j\}_{j\in \N}$ of eigenfunctions of the operator $B$. Setting $\mathbb{V}_{u,n} := \mathrm{span}\{\bm{X}_1, \dots, \bm{X}_n\}$, $\mathbb{V}_{w,n} := \mathrm{span}\{\bm{Y}_1, \dots, \bm{Y}_n\}$ and $\mathbb{V}_n := \mathrm{span}\{z_1, \dots, z_n\}$ as the associated finite dimensional subspaces with orthogonal projection operators $P_{u,n}$, $P_{w,n}$ and $P_n$, respectively, we seek Galerkin solutions of the form
\begin{align*}
\bu_n(t) & := \sum_{j=1}^n a_j^n(t) \bm{X}_j, \quad \bo_n(t) := \sum_{j=1}^n b_j^n(t) \bm{Y}_j, \\
\phi_n(t) &:= \sum_{j=1}^n c_j^n(t) z_j, \quad \mu_n(t) := \sum_{j=1}^n d_j^n(t) z_j,
\end{align*}
that solve the approximating problem
\begin{subequations}\label{Galerkin}
\begin{alignat}{2}
\label{Gal:u} 0 & = (\pd_t(\rho_n \bu_n), \bm{X}_k) - (\rho_n \bu_n \otimes \bu_n, \nabla \bm{X}_k) + (2 \eta_n \D \bu_n, \D \bm{X}_k) + (2 \eta_{r,n} \W \bu_n, \W \bm{X}_k) \\
\notag & \quad + \delta (\bm{A}_S^{3/2} \bu_n, \bm{A}_S^{3/2} \bm{X}_k) - ((\bm{J}_n \otimes \bu_n, \nabla \bm{X}_k) + (\phi_n \nabla \mu_n, \bm{X}_k) - (2 \curl(\eta_{r,n} \bo_n), \bm{X}_k) \\
\notag & \quad + \tfrac{1}{2}(\rho_n'' m(\phi_n) (\nabla \phi_n \cdot \nabla \mu_n)\bu_n, \bm{X}_k) + \tfrac{1}{2}(\rho_n' (P_n(\bu_n \cdot \nabla \phi_n) - \bu_n \cdot \nabla \phi_n) \bu_n, \bm{X}_k) \\
\notag & \quad + \tfrac{1}{2} (\rho_n' (\div (m(\phi_n) \nabla \mu_n) - P_n(\div (m(\phi_n) \nabla \mu_n))) \bu_n, \bm{X}_k), \\[1ex]
\label{Gal:w} 0 & = (\pd_t (\rho_n \bo_n), \bm{Y}_k) - (\rho_n \bu_n \otimes \bo_n, \nabla \bm{Y}_k)  + (2 c_{d,n} \D \bo_n, \D \bm{Y}_k)  + (2 c_{a,n} \W \bo_n, \W \bm{Y}_k) \\
\notag & \quad + (c_{0,n} \div \bo_n, \div \bm{Y}_k) - (\bm{J}_n \otimes \bo_n, \nabla \bm{Y}_k) - (2 \eta_{r,n}(\curl \bu_n - 2 \bo_n), \bm{Y}_k)  \\
\notag & \quad  + \tfrac{1}{2}(\rho_n'' m(\phi_n)(\nabla \phi_n \cdot \nabla \mu_n) \bo_n, \bm{Y}_k) + \tfrac{1}{2}(\rho_n' (P_n(\bu_n \cdot \nabla \phi_n) - \bu_n \cdot \nabla \phi_n) \bo_n, \bm{X}_k) \\
\notag & \quad +  \tfrac{1}{2} (\rho_n' (\div (m(\phi_n) \nabla \mu_n) - P_n(\div (m(\phi_n) \nabla \mu_n))) \bo_n, \bm{Y}_k) + \delta (\bm{A}_1^{3/2} \bo_n, \bm{A}_1^{3/2} \bm{Y}_k) , \\[1ex]
\label{Gal:phi} 0 & = (\pd_t \phi_n , z_k) + (m(\phi_n) \nabla \mu_n, \nabla z_k) - (\phi_n \bu_n, \nabla z_k),  \\[1ex]
\label{Gal:mu} \mu_n & = P_n(a \phi_n - K \star \phi_n + F_{\epsilon}'(\phi_n) - \delta \Delta \phi_n), \\[1ex]
\bm{J}_n & = - \rho_n' m(\phi_n) \nabla \mu_n,
\end{alignat}
\end{subequations}
where we used the notation $f_n = f(\phi_n)$ for $f \in \{\rho, \rho', \rho'', \eta, \eta_r, c_0, c_d, c_a\}$, and set $\bu_n(0) = P_{u,n}( \bu_0)$, $\bw_n(0) = P_{w,n}(\bo_0)$ and $\phi_n(0) = P_n(\phi_{0,\delta})$.  The additional terms involving $\bu_n \cdot \nabla \phi_n$ and $\div (m(\phi_n) \nabla \mu_n)$ in \eqref{Gal:u} and \eqref{Gal:w} are introduced to ensure the validity of an energy identity. 

The system \eqref{Galerkin} can be expressed as a system of ordinary differential equations with right-hand sides depending continuously on the unknowns $\{a_j^n, b_j^n, c_j^n\}$ (note that \eqref{Gal:mu} allows us to express $\{d_j^n\}$ as functions of $\{c_j^n\}$). Hence, by the theory of ordinary differential systems, there exists a time $T_n \in (0,+\infty]$ such that this system admits solutions $\bm{a}^n := (a_1^n, \dots, a_n^n)$, $\bm{b}^n := (b_1^n, \dots, b_n^n)$, $\bm{c}^n := (c_1^n, \dots, c_n^n)$ and $\bm{d}^n := (d_1^n, \dots, d_n^n)$ on $[0,T_n)$ with $\bm{a}^n, \bm{b}^n, \bm{c}^n, \bm{d}^n \in C^1([0,T_n);\R^n)$.

Next, we multiply \eqref{Gal:u} with $a_k^n$, \eqref{Gal:w} with $b_k^n$, \eqref{Gal:phi} with $d_k^n$, summing over $k = 1, \dots, n$ and employing the following identities for $\bm{g}_n \in \{\bu_n,\bo_n\}$:
\begin{align*}
 (\pd_t (\rho_n \bm{g}_n), \bm{g}_n) & = \frac{d}{dt} \int_\Omega \frac{\rho_n}{2} |\bm{g}_n|^2 \, dx + \int_\Omega \pd_t \rho_n |\bm{g}_n|^2 \, dx \\
& = \frac{d}{dt} \int_\Omega \frac{\rho_n}{2} |\bm{g}_n|^2 \, dx + \frac{1}{2} \int_\Omega \rho_n' |\bm{g}_n|^2 P_n(\div (m(\phi_n) \nabla \mu_n) - \bu_n \cdot \nabla \phi_n) \, dx,
\end{align*}
which comes from rewriting \eqref{Gal:phi} as $\pd_t \phi_n = P_n(\div (m(\phi_n) \nabla \mu_n) - \nabla \phi_n \cdot \bu_n)$, and
\begin{align*}
 - (\rho_n \bu_n \otimes \bm{g}_n, \nabla \bm{g}_n)  = \int_\Omega (\nabla \rho_n \cdot \bu_n) \frac{|\bm{g}_n|^2}{2} \, dx = \frac{1}{2} \int_\Omega \rho_n' (\bu_n \cdot \nabla \phi_n) |\bm{g}_n|^2 \, dx,
\end{align*}
as well as
\begin{align*}
& - (\bm{J}_n \otimes \bm{g}_n, \nabla \bm{g}_n) = \int_\Omega |g_n|^2 \div \bm{J}_n + \frac{1}{2} \bm{J}_n \cdot \nabla |\bm{g}_n|^2 \, dx \\
& \quad = \int_\Omega \frac{1}{2} |\bm{g}_n|^2 \div \bm{J}_n \, dx = - \frac{1}{2} \int_\Omega \div (\rho_n'm(\phi_n) \nabla \mu_n) |\bm{g}_n|^2 \, dx \\
& \quad = - \frac{1}{2} \int_\Omega \rho_n'' m(\phi_n) (\nabla \phi_n \cdot \nabla \mu_n) |\bm{g}_n|^2 + \rho_n' \div (m(\phi_n) \nabla \mu_n) |\bm{g}_n|^2 \, dx.
\end{align*}
After some calculations we obtain
\begin{equation}\label{Gal:energyId}
\begin{aligned}
& \frac{d}{dt} \int_\Omega \frac{\rho_n}{2}\Big (|\bu_n|^2 + |\bo_n|^2\Big) + \frac{a}{2}|\phi_n|^2 + F_{\epsilon}(\phi_n) + \frac{\delta}{2}|\nabla \phi_n|^2 - \frac{1}{2} \phi_n (K \star \phi_n) \, dx \\
& \qquad + \int_\Omega m(\phi_n)|\nabla \mu_n|^2 + 2 \eta_n |\D \bu_n|^2 + 4 \eta_{r,n} |\tfrac{1}{2} \curl \bu_n - \bo_n|^2 \, dx \\
& \qquad + \int_\Omega c_{0,n}|\div \bo_n|^2 + 2 c_{d,n} |\D \bo_n|^2 + 2 c_{a,n}|\W \bo_n|^2 \, dx \\
& \qquad + \delta \| \bm{A}_S^{3/2} \bu_n \|^2 + \delta \| \bm{A}_1^{3/2} \bo_n \|^2 \\
& \quad = 0.
\end{aligned}
\end{equation}
We used the relation \eqref{Wv:curl} to combine the terms involving $\W \bu_n : \W \bu_n$, $\curl \bu_n \cdot \bo_n$ and $|\bo_n|^2$ to obtain a dissipative term $|\frac{1}{2}\curl \bu_n -  \bo_n|^2$. For the two-dimensional setting, the relation \eqref{2D:curlu:w} is used instead.

For the remainder of the proof we denote positive constants independent of $n$ and $\delta$ by the symbol $C$. Any dependence of the constants on $\delta$ will be indicated expressed as $C_\delta$.

As $\phi_{0,\delta} \in H^2_n(\Omega)$, we have $P_n(\phi_{0,\delta}) \to \phi_{0,\delta}$ in $L^\infty(\Omega)$. Using also $\| P_{u,n}(\bu_0) \| \leq \| \bu_0 \|$, $\| P_{w,n}(\bo) \| \leq \| \bo_0 \|$ and $\| P_n(\phi_{0,\delta}) \|_{H^1} \leq C \| \phi_{0,\delta} \|_{H^1}$, we find from \eqref{Gal:energyId} that there exists a positive constant $C$ independent of $n$ such that
\begin{align*}
& \| \bu_n \|_{L^\infty(0,T;\HH_{\sigma})\cap L^2(0,T;\HH^1_{\sigma})} + \delta^{1/2}  \|  \bu_n \|_{L^2(0,T;D(\bm{A}_S^{3/2}))} \\
& \quad + \| \bo_n \|_{L^\infty(0,T;\HH) \cap L^2(0,T;\HH^1)} + \delta^{1/2} \| \bo_n \|_{L^2(0,T;D(\bm{A}_1^{3/2}))} \\
& \quad + \|\nabla \mu_n \|_{L^2(Q)} + \| F_{\epsilon}(\phi_n) \|_{L^\infty(0,T;L^1(\Omega))} + \delta^{1/2} \| \phi_n \|_{L^\infty(0,T;H^1(\Omega))} \leq C.
\end{align*}
From the polynomial growth of $F_{\epsilon}$ we also infer that 
\[
\| \phi_n \|_{L^\infty(0,T;L^p)} \leq C \quad \forall p \in [2,\infty).
\]
As these uniform estimates are valid in the time interval $[0,T]$ for arbitrary $T > 0$, it holds that the existence time $T_n$ can be extended to a fixed but arbitrary $T>0$. Next, notice that $\Delta \phi_n \in \mathbb{V}_n$ and so upon testing \eqref{Gal:mu} with $- \Delta \phi_n$ yields
\begin{align*}
(\nabla \mu_n, \nabla \phi_n) & = (\nabla \phi_n, (a + F_{\epsilon}''(\phi_n)) \nabla \phi_n + \phi_n \nabla a - \nabla K \star \phi_n) + \delta \| \Delta \phi_n \|^2 \\
& \geq \widehat{c}_0 \| \nabla \phi_n \|^2 - 2\| \nabla K \|_{L^1} \| \nabla \phi_n \| \| \phi_n \| + \delta \| \Delta \phi_n \|^2 \\
& \geq \frac{\widehat{c}_0}{2} \| \nabla \phi_n \|^2 + \delta \| \Delta \phi_n \|^2 - C\| \phi_n \|^2.
\end{align*}
Using H\"older's and Young's inequalities on the left-hand side, we infer from elliptic regularity that
\[
\| \phi_n \|_{L^2(0,T;H^1)} + \delta^{1/2} \| \phi_n \|_{L^2(0,T;H^2)} \leq C.
\]
Then, observing
\[
\Big | \int_\Omega \mu_n \, dx \Big | \leq \| F_{\epsilon}'(\phi_n) \|_{L^1} \leq \widehat{c}_3 \| F_{\epsilon}(\phi_n) \|_{L^1} + \widehat{c}_{4},
\]
we see that $\overline{\mu}_n$ is bounded in $L^\infty(0,T)$, and thus by the Poincar\'e inequality
\[
\| \mu_n \|_{L^2(0,T;H^1(\Omega))} \leq C.
\]
Next, using the uniform boundedness of $\phi_n$ in $L^\infty(0,T;L^p(\Omega))$ for $p \geq 3$ and of $\bu_n$ in $L^2(0,T;\HH^1_{\sigma})$, for an arbitrary $\psi \in H^1(\Omega)$, we have
\begin{align*}
|(\pd_t \phi_n, \psi)| & = |(\pd_t \phi_n, P_n(\psi))| = |-(m(\phi_n) \nabla \mu_n, \nabla P_n(\psi)) + (\phi_n \bu_n, \nabla P_n(\psi))| \\
& \leq C (\| \nabla \mu_n \| + \| \phi_n \|_{L^3} \| \bu_n \|_{L^6}) \| \nabla \psi \| \\
& \leq C (\| \nabla \mu_n \| + \| \nabla \bu_n \|) \| \psi \|_{H^1}.
\end{align*}
This yields
\[
\| \pd_t \phi_n \|_{L^2(0,T;H^1(\Omega)^*)} \leq C.
\]
Invoking the Gagliardo--Nirenburg inequality for three dimensions, we infer the following interpolation embeddings valid for $s,r \in (4, \infty]$
\begin{align}
\notag & L^\infty(0,T;L^2(\Omega)) \cap L^2(0,T;H^3(\Omega)) \hookrightarrow L^s(0,T;W^{\frac{6}{s},2}(\Omega)) \hookrightarrow L^s(0,T;L^{\frac{2s}{s-4}}(\Omega)), \\
\label{SobEmd2} & L^\infty(0,T;H^1(\Omega)) \cap L^2(0,T;H^2(\Omega)) \hookrightarrow L^r(0,T;W^{1+\frac{2}{r},2}(\Omega)) \hookrightarrow L^r(0,T;L^{\frac{6r}{r-4}}(\Omega)),
\end{align}
and with the choices $s = 6$ and $r = 10$, respectively, we obtain the following bounds
\[
\| \bu_n \|_{L^6(Q)} + \| \bo_n \|_{L^6(Q)} + \| \phi_n \|_{L^{10}(Q)} \leq C_\delta.
\]
Using the Sobolev embedding $W^{\frac{6-s}{s},2}(\Omega) \hookrightarrow L^{\frac{6s}{5s-12}}(\Omega)$ with the choice $s = \frac{18}{5}$, and  $W^{\frac{2}{r},2}(\Omega) \hookrightarrow L^{\frac{6r}{3r-4}}(\Omega)$ with the choice $r = \frac{10}{3}$ we deduce that 
\[
\|\nabla \bu_n \|_{L^{\frac{18}{5}}(Q)} + \|\nabla \bo_n \|_{L^{\frac{18}{5}}(Q)} + \| \nabla \phi_n \|_{L^{\frac{10}{3}}(Q)} \leq C_\delta.
\]
For an arbitrary $\bm{v} \in D(\bm{A}_S^{3/2})$, we consider multiplying \eqref{Gal:u} with the coefficients $\{\zeta_{v,k}\}_{k=1}^n$ where $P_{u,n}(\bm{v}) = \sum_{k=1}^n \zeta_{v,k} \bm{X}_k$. Upon summing over $k$ from $1$ to $n$ we have
\begin{align*}
& |(\pd_t(P_{u,n}(\rho_n \bu_n)), \bm{v})| \\
& \quad \leq |(\rho_n \bu_n\otimes \bu_n, \nabla P_{u,n}(\bm{v}))| + 2|(\eta_n \D \bu_n, \D P_{u,n}(\bm{v}))| + 2|(\eta_{r,n} \W \bu_n, \W P_{u,n}(\bm{v}))|  \\
& \qquad + \delta |(\bm{A}_S^{3/2} \bu_n, \bm{A}_S^{3/2} P_{u,n}(\bm{v}))| + |(\bm{J}_n \otimes \bu_n, \nabla P_{u,n}(\bm{v}))| + |(\phi_n \nabla \mu_n, P_{u,n}(\bm{v}))| \\
& \qquad + 2|(\eta_{r,n} \bo_n, \curl P_{u,n}(\bm{v}))|+ \tfrac{1}{2} |(\rho_n'' m(\phi_n)(\nabla \phi_n \cdot \nabla \mu_n) \bu_n, P_{u,n}(\bm{v}))| \\
& \qquad + \tfrac{1}{2}|(\rho_n' (P_n(\bu_n \cdot \nabla \phi_n) - \bu_n \cdot \nabla \phi_n) \bu_n, P_{u,n}(\bm{v}))| \\
& \qquad + \tfrac{1}{2}|(\rho_n'(\div (m(\phi_n) \nabla \mu_n) - P_n(\div (m(\phi_n) \nabla \mu_n))) \bu_n, P_{u,n}(\bm{v}))| \\
& \quad =: I_1 + \cdots + I_{10}.
\end{align*}
The terms $I_3$ and $I_{7}$ that are absent in \cite{SFrigeri1} can be estimated as
\[
I_3 + I_7 \leq C (\| \nabla \bu_n \| + \| \bo_n \|) \| P_{u,n}(\bm{v}) \|_{H^1} \leq C( \| \nabla \bu_n \| + \| \bo_n \|) \| \bm{v} \|_{D(\bm{A}_S^{1/2})}.
\]
Together with \cite[(4.42)-(4.49)]{SFrigeri1} for the other terms, we infer that 
\begin{align*}
\|\pd_t (P_{u,n}(\rho_n \bu_n)) \|_{D(\bm{A}_S^{3/2})^*} & \leq  C \Big ( \| \bu_n \|^2 + \| \nabla \bu_n \| + \delta \| \bm{A}_S^{3/2} \bu_n \| + \| \bu_n \| \| \nabla \mu_n \| + \| \bo_n \|\Big ) \\
& \quad  + C \Big (\| \nabla \phi_n \|_{L^{\frac{10}{3}}} \| \nabla \mu_n \| \| \bu_n \|_{L^6} + \| \bu_n \|_{L^6} \| \nabla \phi_n \|_{L^3} \| \bu_n \| \Big ) \\
& \quad + C\| \nabla \mu_n \| \Big (\| \nabla \bu_n \| + \| \nabla \phi_n \|_{L^{\frac{10}{3}}} \| \bu_n \|_{L^6} \Big ).
\end{align*}
Since $\| \nabla \phi_n\|_{L^{\frac{10}{3}}} \| \nabla \mu_n \| \| \bu_n \|_{L^6} \in L^{\frac{30}{29}}(0,T)$, we see that
\[
\| \pd_t (P_{u,n}(\rho_n \bu_n) ) \|_{L^{\frac{30}{29}}(0,T;D(\bm{A}_S^{3/2})^*)} \leq C_\delta.
\]
Similarly, thanks to the boundedness of $\nabla \phi_n$ in $L^{10/3}(Q)$ and $\bu_n$ in $L^6(Q)$, 
\begin{align*}
\| \nabla (\rho_n \bu_n) \|_{L^{\frac{15}{7}}(Q)} & \leq \| \rho_n \nabla \bu_n \|_{L^{\frac{15}{7}}(Q)} + \| \rho_n' \nabla \phi_n \cdot \bu_n \|_{L^{\frac{15}{7}}(Q)} \\
& \leq C \| \nabla \bu_n \|_{L^{\frac{15}{7}}(Q)} + C \| \nabla \phi_n \|_{L^{\frac{10}{3}}(Q)} \| \bu_n \|_{L^6(Q)},
\end{align*}
and so
\[
\| \rho_n \bu_n \|_{L^{\frac{15}{7}}(0,T;W^{1,\frac{15}{7}}(\Omega))} \leq C_\delta.
\]
Consequently, we have
\[
\| P_{u,n}(\rho_n \bu_n) \|_{L^2(0,T;H^1(\Omega))} \leq C_\delta.
\]
Via a similar calculation we also infer for the micro-rotation
\[
\| \rho_n \bo_n \|_{L^{\frac{15}{7}}(0,T;W^{1,\frac{15}{7}}(\Omega))} + \| P_{w,n}(\rho_n \bo_n) \|_{L^2(0,T;H^1(\Omega)) \cap W^{1,\frac{30}{29}}(0,T;D(\bm{A}_1^{3/2})^*)} \leq C_\delta.
\]
Based on the above uniform estimates, we collect the following compactness assertions along a nonrelabelled subsequence $n \to \infty$: By the Aubin--Lions lemma, there exists a function $\phi \in L^\infty(0,T;H^1(\Omega)) \cap L^2(0,T;H^2_n(\Omega)) \cap H^1(0,T;H^1(\Omega)^*)$ such that 
\begin{align*}
\phi_n \to \phi \text{ strongly in } L^2(0,T;H^{2-\gamma}(\Omega)) \text{ for } \gamma \in (0,2) \text{ and a.e.~in } Q,
\end{align*}
as well as a limit function $\mu \in L^2(0,T;H^1(\Omega))$ such that 
\[
\mu = a \phi - K \star \phi + F_{\epsilon}'(\phi) - \delta \Delta \phi \text{ a.e.~in } Q
\]
satisfying
\[
\div (m(\phi_n) \nabla \mu_n) \to \div (m(\phi) \nabla \mu) \text{ weakly in } L^2(0,T;H^1(\Omega)^*)
\]
and thus
\[
P_n(\div (m(\phi_n) \nabla \mu_n)) \to \div (m(\phi_n) \nabla \mu) \text{ weakly in } L^2(0,T;H^1(\Omega)^*).
\]
Next, there exist functions $\bu \in L^\infty(0,T;\HH_\sigma) \cap L^2(0,T;D(\bm{A}_S^{3/2})) \cap L^6(Q)^d$ as well as $\bm{\xi} \in L^2(0,T;H^1(\Omega;\R^d)) \cap W^{1,\frac{30}{29}}(0,T;D(\bm{A}_S^{3/2})^*)$ such that 
\[
P_{u.n}(\rho_n \bu_n) \to \bm{\xi} \text{ strongly in } L^2(Q).
\]
From the Lebesgue dominated convergence theorem and \eqref{ass:rho:approx} we deduce that $\rho_n \to \rho(\phi)$ strongly in $L^q(Q)$ for any $q \in [2,\infty)$. Choosing $q = 3$ and together with the weak convergence of $\bu_n$ in $L^6(Q)^d$, we get $\rho_n \bu_n \to \rho(\phi) \bu$ weakly in $L^2(Q)^d$, which allows us to identify $\bm{\xi}$ as $\rho(\phi) \bu$. Consequently, it holds that
\[
\pd_t P_{u,n}(\rho_n \bu_n) \to \pd_t (\rho(\phi) \bu) \text{ weakly in } L^{\frac{30}{29}}(0,T;D(\bm{A}_S^{3/2})^*),
\]
as well as
\[
\| \rho_n |\bu_n|^2 \|_{L^1(Q)} = \int_Q P_{u,n}(\rho_n \bu_n) \cdot \bu_n \, dx \, dt \to \int_Q \rho(\phi) |\bu|^2 \, dx \, dt.
\]
This $L^2(Q)^d$-norm convergence of $(\rho_n)^{1/2} \bu_n$ yields the strong convergence $(\rho_n)^{1/2} \bu_n \to (\rho(\phi))^{1/2} \bu$ in $L^2(Q)^d$, and when combined with the strong convergence of $(\rho_n)^{-1/2} \to \rho(\phi)^{-1/2}$ in $L^{q}(Q)$ for $q \in [2,\infty)$ due to \eqref{ass:rho:approx} we arrive at
\[
\bu_n \to \bu \text{ strongly in } L^{2-\gamma}(Q)^d  \text{ for } \gamma \in (0,1] \text{ and a.e.~in } Q.
\]
Boundedness of $\bu_n$ in $L^6(Q)^d$ further provides via interpolation of Lebesgue spaces
\[
\bu_n \to \bu \text{ strongly in } L^{6-\gamma}(Q)^d \text{ for } \gamma \in (0,5].
\]
Similarly, strong convergence of $\nabla \phi_n$ in $L^2(0,T;H^{1-\gamma}(\Omega;\R^d))$ and boundedness in $L^{10/3}(Q)^d$ yields
\[
\nabla \phi_n \to \nabla \phi \text{ strongly in } L^{\frac{10}{3}-\gamma}(Q)^d \text{ for } \gamma \in (0,\tfrac{7}{3}].
\]
Hence, we obtain $\bu_n \cdot \nabla \phi_n \to \bu \cdot \nabla \phi$ strongly in $L^{\frac{15}{7}-\gamma}(Q)$ for $\gamma \in (0, \tfrac{8}{7}]$ and therefore
\[
P_{n}(\bu_n \cdot \nabla \phi_n) \to \bu \cdot \nabla \phi \text{ strongly in } L^{\frac{15}{7}-\gamma}(Q) \text{ for } \gamma \in (0, \tfrac{8}{7}].
\]
Choosing $\gamma = \frac{1}{7}$, we have strong convergence in $L^2(Q)$ so that for arbitrary $\zeta(t) \in C^\infty_c(0,T)$ the penultimate term on the right-hand side of \eqref{Gal:u} converges to zero:
\[
\int_Q \rho_n'(P_n(\bu_n \cdot \nabla \phi_n) - \bu_n \cdot \nabla \phi_n) \bu_n \cdot \zeta(t) \bm{X}_k \, dx \, dt \to 0
\]
for all $k = 1, \dots, n$. Meanwhile, choosing $\gamma < \frac{1}{7}$, and using $\rho_n'' m(\phi_n) \to \rho''(\phi) m(\phi)$ strongly in $L^{(30-14\gamma)/(1-7\gamma)}(Q)$ leads to the strong convergence of $\rho_n'' m(\phi_n)(\bu_n  \cdot \bm{v}) \nabla \phi_n \to \rho''(\phi) m(\phi) (\bu \cdot \bm{v}) \nabla \phi$ in $L^2(Q)$. Combining with the weak convergence of $\nabla \mu_n$ in $L^2(Q)^d$ it holds that 
\begin{align*}
& \frac{1}{2}\int_Q \rho_n'' m(\phi_n) (\nabla \phi_n \cdot \nabla \mu_n) \bu_n \cdot \zeta(t) \bm{X}_k \, dx \, dt \\
& \quad \to \frac{1}{2}\int_Q \rho''(\phi) m(\phi) (\nabla \phi \cdot \nabla \mu) \bu \cdot \zeta(t) \bm{X}_k \, dx \, dt = - \int_Q R \bu \cdot \zeta(t) \bm{X}_k \, dx \, dt.
\end{align*}
For the last term on the right-hand side of \eqref{Gal:u} we aim to establish the strong convergence $\rho_n'\bu_n \cdot \bm{v} \to \rho'(\phi) \bu \cdot \bm{v}$ in $L^2(0,T;H^1(\Omega))$ for $\bm{v} \in D(\bm{A}_S^{3/2}) \hookrightarrow H^3(\Omega;\R^3)$. A short calculation shows that 
\begin{align*}
& \| \rho_n'\bu_n \cdot \bm{v} - \rho'(\phi_n) \bu \cdot \bm{v} \|_{L^2(0,T;H^1(\Omega))} \\
& \quad \leq \|(\rho_n' \bu_n - \rho'(\phi) \bu) \cdot \bm{v} \|_{L^2(Q)} + \| (\rho_n'' \nabla \phi_n - \rho''(\phi) \nabla \phi) \bu_n \cdot \bm{v} \|_{L^2(Q)} \\
& \qquad + \| \rho''(\phi) \nabla \phi (\bu_n - \bu) \cdot \bm{v} \|_{L^2(Q)} + \| (\rho'_n - \rho'(\phi) )(\nabla \bu_n)\bm{v} \|_{L^2(Q)} \\
& \qquad + \| \rho'(\phi) \nabla (\bu_n - \bu) \bm{v} \|_{L^2(Q)} + \| (\rho_n' \bu_n - \rho'(\phi) \bu) \nabla \bm{v} \|_{L^2(Q)}\\
& \quad \leq C \Big (\| \rho_n' \bu_n - \rho'(\phi) \bu \|_{L^2(Q)} + \| \rho_n'' - \rho''(\phi) \|_{L^{30}(Q)} \| \nabla \phi_n \|_{L^{10/3}(Q)} \| \bu_n \|_{L^6(Q)} \\
& \qquad + \| \nabla \phi_n - \nabla \phi \|_{L^4(Q)} \| \bu_n \|_{L^4(Q)} + \| \nabla \phi \|_{L^4(Q)} \| \bu_n - \bu \|_{L^4(Q)} \\
& \qquad + \| \rho_n' - \rho'(\phi) \|_{L^{18/4}(Q)}\| \nabla \bu_n \|_{L^{18/5}(Q)} + \| \nabla (\bu_n - \bu) \|_{L^2(Q)} \\
& \qquad + \| \rho_n' \bu_n - \rho'(\phi) \bu \|_{L^2(Q)} \Big ) \| \bm{v} \|_{D(\bm{A}_S^{3/2})}.
\end{align*}
By the strong convergences of $\rho_n'$ and $\rho_n''$ in $L^q(Q)$ for any $q \in [2,\infty)$, as well as the strong convergence $\nabla \bu_n \to \nabla \bu$ in $L^2(Q)^{d \times d}$ via the Aubin--Lions lemma and the boundedness of $\bu_n$ in $L^2(0,T;D(\bm{A}_S^{3/2})) \cap W^{1,\frac{30}{29}}(0,T;D(\bm{A}_S^{3/2})^*)$ we find that 
\[
\rho_n' \bu_n \cdot \bm{v} \to \rho'(\phi) \bu \cdot \bm{v} \text{ strongly in } L^2(0,T;H^1(\Omega)).
\]
Hence, when combining with the weak convergence $P_n(\div (m(\phi_n) \nabla \mu_n)) \to \div (m(\phi) \nabla \mu)$ in $L^2(0,T;H^1(\Omega)^*)$ established earlier, we see that the last term on the right-hand side of \eqref{Gal:u} converges to zero as $n \to \infty$.

For the micro-rotation $\bo_n$, analogous weak/strong convergence assertions hold. Then, via a standard argument of multiplying \eqref{Gal:u}--\eqref{Gal:phi} by $\zeta(t) \in C^\infty_c(0,T)$ and passing to the limit $n \to \infty$ we obtain the variational equalities in \eqref{appsys:weakform}. 

The assertion $\bu \in C_w([0,T];\HH_\sigma)$ comes from apply Lemma \ref{lem:Cw} to $\bm{U} := \rho(\phi) \bu \in L^\infty(0,T;\HH_\sigma) \cap W^{1,\frac{30}{29}}(0,T;D(\bm{A}_S^{3/2})^*)$, and noticing that $\rho(\phi)^{-1} \in L^\infty(0,T;H^1(\Omega)) \cap H^1(0,T;H^1(\Omega)^*) \Subset C^0([0,T];L^2(\Omega))$ due to \eqref{ass:rho:approx}, we have $\bu = \rho(\phi)^{-1} \bm{U} \in C_w([0,T];\HH_\sigma)$. The argument to show $\bo \in C_w([0,T];\HH)$ is similar.

For the attainment of initial conditions, we note that from the compact embedding $L^\infty(0,T;H^1(\Omega)) \cap H^1(0,T;H^1(\Omega)^*) \Subset C^0([0,T];L^2(\Omega))$ that $\phi(0) \in L^2(\Omega)$ is well-defined. Let $\psi \in H^1(\Omega)$ and $\zeta(t) \in C^\infty_c(0,T)$ with $\zeta(T) = 0$ be arbitrary. Fix $K \in \mathbb{N}$ and denote by $\{\alpha_k\}_{k=1}^K$ the coefficients of $P_K(\psi)$. Upon multiplying \eqref{Gal:phi} by $\alpha_k \zeta(t)$, summing from $k = 1$ to $K$, and integrating in time leads to 
\begin{align*}
 0 = \int_0^T  \zeta(t)  (m(\phi_n) \nabla \mu_n - \phi_n \bu_n, \nabla P_K(\psi)) - \zeta'(t) (P_n(\phi_{0,\delta}) - \phi_n, P_K(\psi)) \, dt.
\end{align*}
Passing to the limit $n \to \infty$ and then $K \to \infty$ yields
\begin{align*}
0 & = \int_0^T \zeta(t) (m(\phi) \nabla \mu - \phi \bu, \nabla \psi) - \zeta'(t) (\phi_{0,\delta} - \phi, \psi) \, dt \\
& = \int_0^T \zeta(t) \langle \pd_t \phi, \psi \rangle_{H^1(\Omega)} - \zeta'(t) (\phi_{0,\delta} - \phi, \nabla \psi) \, dt = \int_0^T \zeta'(t) (\phi(0) - \phi_{0,\delta}, \psi) \, dt. 
\end{align*}
Since $\psi$ and $\zeta$ are arbitrary, this shows that $\phi(0) = \phi_{0,\delta}$. The argument for establishing $\bu(0) = \bu_0$ and $\bo(0) = \bo_0$ are similar, see e.g.~\cite{SFrigeri1}.

Let us rewrite \eqref{Gal:energyId} as
\[
\frac{d}{dt} E_n(t) + D_n(t) = 0.
\]
Multiplying the above by an arbitrary $w \in W^{1,1}(0,T)$ such that $w(T) = 0$ and $w \geq 0$, integrating by parts in time yields
\begin{align*}
\int_0^T D_n(t) w(t) \, dt = \int_0^T w'(t) E_n(t) \, dt + w(0) E_n(0) = 0.
\end{align*}
Passing to the limit $n \to \infty$ with the above weak/strong convergences for $(\bu_n, \bo_n, \phi_n, \mu_n)$ and the weak lower semicontinuity of the norms leads to 
\begin{align*}
& \int_0^T w(t) \int_\Omega m(\phi) |\nabla \mu|^2 + 2 \eta(\phi) |\D \bu|^2 + 4 \eta_r(\phi) |\tfrac{1}{2} \curl \bu - \bo|^2 \, dx \, dt \\
& \qquad + \int_0^T w(t) \int_\Omega c_0(\phi) |\div \bo|^2 + 2 c_d(\phi) |\D \bo|^2 + 2 c_a(\phi) |\W \bo|^2 \, dx \, dt \\
& \qquad + \int_0^T w(t) \int_\Omega  + \delta |\bm{A}_S^{3/2} \bu|^2 + \delta |\bm{A}_1^{3/2} \bo|^2 \, dx \, dt \\
& \quad \leq \int_0^T w'(t) \Big [ \int_\Omega \frac{\rho(\phi)}{2} \Big ( |\bu|^2 + |\bo|^2 \Big ) + F_{\epsilon}(\phi) + \frac{\delta}{2} |\nabla \phi|^2 \, dx + e(\phi) \Big ] \, dt \\
& \qquad w(0) \Big [ \int_\Omega \frac{\rho(\phi_{0,\delta})}{2} \Big ( |\bu_0|^2 + |\bo_0|^2 \Big ) + F_{\epsilon}(\phi_{0,\delta}) + \frac{\delta}{2} |\nabla \phi_{0,\delta}|^2 \, dx + e(\phi_{0,\delta}) \Big ]
\end{align*}
holding for all $w \in W^{1,1}(0,T)$ with $w(T) = 0$ and $w \geq 0$. Then, the energy inequality \eqref{ene.ineq:weaksol} follows from applying Lemma \ref{lem:Energyineq}.
\end{proof}

With the establishment of Lemma \ref{lem:exist:weaksol}, Step 1 of the proof of Theorem \ref{thm:nMAGG:weaksol} is completed. Let us briefly comment on the remaining steps. In Step 2 we consider a polynomial approximation $F_{\epsilon} := F_{1,\epsilon} + F_{2,\epsilon}$ of the singular potential $F = F_1 + F_2$ defined as 
\begin{align}\label{polyapprox:F}
F_{1,\epsilon}^{(p)}(s) = \begin{cases}
F_1^{(p)}(1-\epsilon) & \text{ for } s \geq 1- \epsilon, \\
F_1^{(p)}(s) & \text{ for } |s| \leq 1- \epsilon, \\
F_1^{(p)}(\epsilon - 1) & \text{ for } s \leq -1 + \epsilon, 
\end{cases} \quad F_{2,\epsilon}''(s) = \begin{cases}
F_2''(1-\epsilon) & \text{ for } s \geq 1- \epsilon, \\
F_2''(s) & \text{ for } |s| \leq 1- \epsilon, \\
F_2''(\epsilon - 1) & \text{ for } s \leq -1 + \epsilon, 
\end{cases}
\end{align} 
with $F_{1,\epsilon}^{(k)}(0) = F_{1}^{(k)}(0)$ for $k =0, 1, \dots, p-1$ and $F_{2,\epsilon}^{(k)}(0) = F_2^{(k)}(0)$ for $k = 0, 1$. By \cite[Lemmas 1 \& 2]{FrigeriGrasselli} and \cite[Proof of Theorem 2]{FrigeriGrasselliRocca} there exist constants $c_0 > 0$, $c_1 > 0$, $c_2 > 0$ and $\epsilon_0 \in (0,1]$ such that for all $\epsilon \in (0,\epsilon_0)$:
\begin{align*}
F_{\epsilon}''(s) + a(x) \geq c_0, \quad 
F_{\epsilon}(s) \geq c_1 |s|^p - c_2 \quad \forall s \in \R \text{ and a.e.~} x \in \Omega.
\end{align*}
Furthermore, by construction it holds that
\begin{align}\label{F0approx}
|F_{1,\epsilon}'(s)| \leq |F_1'(s)|, \quad |F_{1,\epsilon}(s)| \leq |F_1(s)| \quad \forall s \in (-1,1),
\end{align}
and hence for an initial condition $\phi_{0,\delta} \in H^1(\Omega)$ such that $F(\phi_{0,\delta}) \in L^1(\Omega)$ we find that the right-hand side of the energy inequality \eqref{ene.ineq:weaksol} for a weak solution $(\bu_{\epsilon,\delta}, \bo_{\epsilon,\delta}, \phi_{\epsilon,\delta}, \mu_{\epsilon,\delta})$ from Lemma \ref{lem:exist:weaksol} can be bounded uniformly in $\epsilon \in (0,1]$. 

The only main difference with Step 1 is estimating $F_{\epsilon}'(\phi_{\epsilon,\delta})$ in $L^2(0,T;L^1(\Omega))$ uniformly in $\epsilon$, and in turn the derivation of uniform $L^2(0,T;L^2(\Omega))$ bounds for $\mu_{\epsilon,\delta}$ requires a modified argument. As we will provide a similar reasoning in the next section we omit the details here. Let us remark that the compactness assertions for $\bu_{\epsilon,\delta}$ and $\bo_{\epsilon,\delta}$ are unchanged from those in Step 1.

Then, passing to the limit $\epsilon \to 0$ yields weak solutions $(\bu_\delta, \bo_\delta, \phi_\delta, \mu_\delta)$ with the same regularity as in \eqref{Gal:regularity} to the approximate system \eqref{app:model1} but with the singular potential $F$. In turn, this provides the boundedness property $\phi_{\delta} \in L^\infty(Q)$ with $|\phi_\delta| < 1$ a.e.~in $Q$, while the mass density function $\rho(\phi_\delta)$ coincides with the affine linear function $\frac{\overline{\rho}_1 - \overline{\rho}_2}{2} \phi_\delta + \frac{\overline{\rho}_1 + \overline{\rho}_2}{2}$, and the terms involving $R = -m(\phi_\delta) \nabla \mu_\delta \cdot \nabla \rho'(\phi_\delta)$ in \eqref{app:model:u} and \eqref{app:model:w} vanish. 

Lastly, in Step 3, from an analogous energy inequality to \eqref{ene.ineq:weaksol} (with $F_{\epsilon}$ replaced by the original singular potential $F$) for $(\bu_\delta, \bo_\delta, \phi_\delta, \mu_\delta)$ and a similar set of arguments in Step 2 allow us to pass to the limit $\delta \to 0$ to complete the proof. The main difference with Steps 1 and 2 is the estimate for the time derivatives $\pd_t (\rho(\phi_\delta) \bu_\delta)$ and $\pd_t(\rho(\phi_\delta) \bo_\delta)$. Let us briefly sketch the modifications. From the analogue of \eqref{weak:form:bu:sec2} for  $(\bu_\delta, \bo_\delta, \phi_\delta, \mu_\delta)$:
\begin{align*}
\langle \pd_t (\rho_\delta \bu_\delta), \bm{v} \rangle_{D(\bm{A}_S^{3/2})} & = (\rho_\delta \bu_\delta \otimes \bu_\delta, \nabla \bm{v}) - (2\eta_\delta \D \bu_\delta, \D \bm{v}) - (2 \eta_{r,\delta} \W \bu_\delta, \W \bm{v}) + (\bm{J}_\delta \otimes \bu_\delta, \nabla \bm{v})  \\
& \quad - \delta(\bm{A}_S^{3/2} \bu_\delta, \bm{A}_S^{3/2} \bm{v})- (\phi_\delta \nabla \mu_\delta, \bm{v}) + (2 \curl(\eta_{r,\delta} \bo_\delta), \bm{v})
\end{align*}
for a.e.~$t \in (0,T)$ and for all $\bm{v} \in D(\bm{A}_S^{3/2})$, where $\rho_\delta := \rho(\phi_\delta)$, $\eta_\delta = \eta(\phi_\delta)$, $\eta_{r,\delta} = \eta_r(\phi_\delta)$ and $\bm{J}_\delta = - \rho'(\phi_\delta) m(\phi_\delta) \nabla\mu_\delta$, we use the uniform estimates of $\bu_\delta$ in $L^\infty(0,T;\HH_\sigma) \cap L^2(0,T;\HH^1_{\sigma})$, of $\delta^{1/2} \bm{A}_S^{3/2} \bu_\delta$ in $L^2(Q)^d$, of $\bo_\delta$ in $L^\infty(0,T;\HH) \cap L^2(0,T;\HH^1)$, of $\phi_\delta$ in $L^\infty(Q)$ and of $\mu_\delta$ in $L^2(0,T;H^1(\Omega))$ to obtain 
\begin{align*}
\| \pd_t (\rho_\delta \bu_\delta) \|_{L^2(0,T;D(\bm{A}_S^{3/2})^*)} \leq C
\end{align*}
with a constant $C$ independent of $\delta$. Upon passing to the limit $\delta \to 0$ we obtain a quadruple of functions $(\bu, \bo, \phi, \delta)$ satisfying \eqref{weak:nloc:1}--\eqref{weak:nloc:2} but holding for all $\bm{v} \in D(\bm{A}_S^{3/2})$ and $\bm{z} \in D(\bm{A}_1^{3/2})$. The absence of the $\delta$ regularisation terms allows us to use a density argument to have \eqref{weak:form:bu:sec2} and \eqref{weak:form:bo:sec2} holding for all $\bm{v} \in D(\bm{A}_S)$ and $\bm{z} \in D(\bm{A}_1)$. Then, \textit{a posteriori} we find that $\pd_t(\rho \bu) \in L^{\frac{4}{3}}(0,T;D(\bm{A}_S)^*)$ and $\pd_t(\rho \bo) \in L^{\frac{4}{3}}(0,T;D(\bm{A}_1)^*)$ by a comparison of terms. The attainment of the initial conditions and derivation of the energy inequality \eqref{nloc:ene:ineq:sec2} proceed similarly as before.

\section{Nonlocal-to-local weak convergence}\label{sec:weakconv}
In this section we assume the nonlocal interaction kernel $K$ is chosen to be the form $K_\kappa(x) = \frac{\gamma_\kappa(|x|)}{|x|^2}$ where the family of functions $(\gamma_\kappa)_{\kappa >0}$ satisfies \eqref{nonl:to:loc:weak:ass}. We denote  
\begin{equation*}
\begin{alignedat}{3}
e_{\kappa}(\phi) & :=   \frac{1}{4}\int_{\Omega}\int_{\Omega}K_\kappa(x-y)(\phi(x)-\phi(y))^2\, dx \,dy  \quad && \text{ for } \phi \in L^2(\Omega), \\
e_0(\phi) & := \frac{1}{2} \int_\Omega |\nabla \phi|^2 \, dx \quad&&  \text{ for } \phi \in H^1(\Omega),
\end{alignedat}
\end{equation*}
and set $a_\kappa(x) := (K_\kappa \star 1)(x)$. Furthermore, we denote 
\begin{equation*}
\begin{alignedat}{4}
\mathcal{E}_{\kappa,\nloc}(\bu, \bo, \phi) & := \int_\Omega \frac{\rho(\phi)}{2} (|\bu|^2 + |\bo|^2) \, dx + G_{\kappa,\nloc}(\phi), && \quad G_{\kappa,\nloc}(\phi) && := e_\kappa(\phi) + \int_\Omega F(\phi) \,dx, \\
\mathcal{E}_{\loc}(\bu, \bo, \phi) & := \int_\Omega \frac{\rho(\phi)}{2} (|\bu|^2 + |\bo|^2) \, dx + G_{\loc}(\phi), && \quad \quad G_{\loc}(\phi) && := e_0(\phi) + \int_\Omega F(\phi) \, dx.
\end{alignedat}
\end{equation*}
For the nonlocal-to-local convergence of weak solutions, we report on the following two auxiliary results.
\begin{lem}[Lemma 3.3 of \cite{Davoli}] \label{sec2:lem:nloc:to:loc:ene}
For every $\phi,\zeta \in H^{1}(\Omega)$, it holds that
\begin{subequations}\label{nloc:to:loc:conv1}
\begin{alignat}{2}
 \lim_{\kappa \to 0} e_{\kappa}(\phi) & = e_0(\phi), \\[1ex]
 \lim_{\kappa \to 0} \int_{\Omega} (a_\eps \phi - K_\kappa \star \phi)(x)\zeta(x) \, dx &= \int_{\Omega} \nabla \phi(x) \cdot \nabla\zeta(x) \, dx.
\end{alignat}
\end{subequations}
Moreover, for every sequence $(\phi_\kappa)_{\kappa > 0} \subset L^2(\Omega)$ and $\phi \in L^2(\Omega)$ it holds that 
\begin{subequations}
\begin{alignat}{2}
\sup_{\kappa > 0} e_\kappa(\phi) < + \infty \quad & \Rightarrow \quad (\phi_\kappa)_{\kappa > 0} \text{ is relatively compact in } L^2(\Omega), \\[1ex]
\label{nloc:to:loc:4} \phi_\kappa \to \phi \text{ in } L^2(\Omega) \quad & \Rightarrow  \quad e_0(\phi) \leq \liminf_{\kappa \to 0} e_\kappa(\phi_\kappa).
\end{alignat}
\end{subequations}
\end{lem}

\begin{lem}[Lemma 3.4 of \cite{Davoli}]\label{sec2:lem:phi:e:str:con:L2}
For any $\delta > 0$, there exists $C_\delta > 0$ and $\kappa_\delta > 0$ such that for any $(\phi_\kappa)_{\kappa > 0} \subset L^2(\Omega)$,
\begin{equation}
\| \phi_{\kappa_1} - \phi_{\kappa_2} \|^2 \leq \delta (e_{\kappa_1}(\phi_{\kappa_1}) + e_{\kappa_2}(\phi_{\kappa_2})) + C_\delta \| \phi_{\kappa_1} - \phi_{\kappa_2} \|_{H^1(\Omega)^*}^2
\end{equation}
holds for any $\kappa_1$, $\kappa_2 \in (0,\kappa_\delta)$.
\end{lem}

A further property of the singular potential $F$ that is useful is the following, see e.g.~\cite[Proposition 4.1]{Miranvillebk}: For any $r \in (-1,1)$, there exist constants $c_r > 0$ and $C_r \geq 0$ such that the singular part $F_1$ of the potential $F$ satisfies
\begin{align}\label{F'abs:ineq}
c_r |F_1'(s)| \leq F_1(s)(s-r) + C_r \quad \forall s \in (-1,1).
\end{align}
The main result of this section is the convergence of weak solutions of the nMAGG model \eqref{nonl:model:equ} to the weak solutions of the local MAGG model \eqref{local:model:equ}.

\begin{thm}\label{thm:weak:cov}
Suppose that \eqref{gen:ass1}--\eqref{gen:ass3}, \eqref{sing:pot:ass1}--\eqref{sing:pot:ass5} and \eqref{nonl:to:loc:weak:ass} hold. For any $\kappa \in (0,1)$ let $\bu_{0,\kappa} \in \HH_\sigma$, $\bo_{0,\kappa} \in \HH$ and $\phi_{0,\kappa} \in L^\infty(\Omega)$ such that $F(\phi_{0,\kappa}) \in L^1(\Omega)$ and $|\overline{\phi_{0,\kappa}}| < 1$. Moreover, we assume that there are $\bu_0 \in \HH_\sigma$, $\bo_0 \in \HH$ and $\phi_0 \in H^1(\Omega)$ such that $\bu_{0,\kappa} \to \bu_0$ strongly in $\HH_\sigma$, $\bo_{0,\kappa} \to \bo_0$ strongly in $\HH$, $\phi_{0,\kappa} \to \phi_0$ strongly in $L^2(\Omega)$ and $\mathcal{E}_{\kappa,\nloc}(\bu_{0,\kappa}, \bo_{0,\kappa}, \phi_{0,\kappa}) \to \mathcal{E}_{\loc}(\bu_0, \bo_0, \phi_0)$ as $\kappa \to 0$. If $(\bu_\kappa, \bo_\kappa, \phi_\kappa, \mu_\kappa)$ is a weak solution to the nMAGG model \eqref{nonl:model:equ} with kernel $K_\kappa(x) = \frac{\gamma_\kappa(|x|)}{|x|^2}$ corresponding to initial data $(\bu_{0,\kappa}, \bo_{0,\kappa}, \phi_{0,\kappa})$, then along a nonrelabelled subsequence $\kappa \to 0$,
\begin{align*}
\bu_\kappa \to \bu & \text{ weakly* in } L^\infty(0,T;\HH_\sigma) \cap L^2(0,T;\HH^1_{\sigma}), \\
\bu_\kappa \to \bu & \text{ strongly in } L^2(0,T;\HH_\sigma) \text{ and a.e.~in } Q, \\
\bo_\kappa \to \bo & \text{ weakly* in } L^\infty(0,T;\HH) \cap L^2(0,T;\HH^1), \\
\bo_\kappa \to \bo & \text{ strongly in } L^2(0,T;\HH) \text{ and a.e.~in } Q, \\
\mu_\kappa \to \mu & \text{ weakly in } L^2(0,T;H^1(\Omega)), \\
\phi_\kappa \to \phi & \text{ strongly in } C^0([0,T];L^2(\Omega)) \text{ and a.e.~in } Q,
\end{align*}
where $(\bu, \bo, \phi, \mu)$ is a weak solution to the MAGG model \eqref{local:model:equ} in the following sense:
\begin{itemize}
\item Regularity
\begin{equation*}
\begin{alignedat}{3}
\bu & \in C_w([0,T]; \HH_\sigma) \cap L^2(0,T;\HH^1_{\sigma}), \\
\bo & \in C_w([0,T]; \HH) \cap L^2(0,T;\HH^1), \\
\phi & \in C_w([0,T];H^1(\Omega)) \cap L^2(0,T;H^2_n(\Omega)) \text{ s.t. }  |\phi| < 1 \text{ a.e.~in } Q, \\
\mu & \in L^2(0,T;H^1(\Omega)).
\end{alignedat}
\end{equation*}
\item Equations
\begin{subequations}\label{weak:loc:1}
\begin{alignat}{2}
\label{weak:loc:u} 0 &=  \langle \pd_t(\rho(\phi) \bu), \bm{v} \rangle_{D(\bm{A}_S)}+ (\rho \bu \otimes \bu, \nabla \bm{v})  + (2 \eta(\phi) \D \bu, \D \bm{v}) \\
\notag & \quad + (2 \eta_r(\phi) \W \bu, \W \bm{v})  - (2\eta_r(\phi) \bo, \curl \bm{v})_Q   - (( \bu \otimes \bm{J}), \nabla \bm{v}) - (\mu \nabla \phi, \bm{v}), \\
\label{weak:loc:w} 0 & = \langle \pd_t(\rho(\phi) \bo), \bm{z} \rangle_{D(\bm{A}_1)} + (\rho \bu \otimes \bo, \nabla \bm{z})+ (c_0(\phi) \div \bo, \div \bm{z})   \\
\notag & \quad + (2 c_d(\phi) \D \bo, \D \bm{z}) + (2 c_a(\phi) \W \bo, \W \bm{z})  - (2 \eta_r(\phi) (\curl \bu - 2 \bo), \bm{z}) \\
\notag & \quad - ((\bo \otimes \bm{J}), \nabla \bm{z}), \\
\label{weak:loc:phi} 0 & = \langle \pd_t\phi, \psi \rangle_{H^1(\Omega)} - (\phi \bu, \nabla \psi) + (m(\phi)\nabla \mu, \nabla \psi), 
\end{alignat}
\end{subequations}
holding for all $\psi \in H^1(\Omega)$, $\bm{v} \in D(\bm{A}_S)$, $\bm{z} \in D(\bm{A}_1)$ and for almost any $t \in (0,T)$, along with 
\begin{subequations}\label{weak:loc:2}
\begin{alignat}{2}
\label{weak:loc:mu} \mu & = - \Delta \phi + F'(\phi) && \quad \text{ a.e.~in } Q, \\[1ex]
\label{weak:loc:J} \bm{J} & = -\tfrac{\overline{\rho}_1 - \overline{\rho}_2}{2} m(\phi) \nabla \mu && \quad \text{ a.e.~in } Q, \\[1ex]
\label{weak:loc:rho} \rho(\phi) & = \tfrac{\overline{\rho}_1 - \overline{\rho}_2}{2} \phi + \tfrac{\overline{\rho}_1+ \overline{\rho}_2}{2} && \quad \text{ a.e.~in } Q.
\end{alignat}
\end{subequations}
\item Energy inequality
\begin{equation}\label{energy:loc:ineq}
\begin{aligned}
& \mathcal{E}_{\loc}(\bu(t), \bo(t), \phi(t)) \\
& \qquad + \int_s^t \int_\Omega m(\phi) |\nabla \mu|^2 + 2 \eta(\phi) |\D \bu|^2 + 4 \eta_r(\phi) | \tfrac{1}{2} \curl \bu - \bo|^2 \, dx \, d\tau \\
& \qquad + \int_s^t \int_\Omega c_0(\phi) |\div \bo|^2 + 2 c_d(\phi) |\D \bo|^2 + 2 c_a(\phi) |\W \bo|^2 \, dx \, d \tau \\
& \quad \leq \mathcal{E}_{\loc}(\bu(s), \bo(s), \phi(s))
\end{aligned}
\end{equation}
holds for all $t \in [s,T)$ and almost all $s \in [0,T)$ including $s = 0$.
\item Initial conditions
\[
(\bu, \bo, \phi) \vert_{t = 0} = (\bu_0, \bo_0, \phi_0).
\]
\end{itemize}
\end{thm}

\begin{proof}
We adapt the arguments in \cite{AbelsTera} for the nonlocal-to-local weak convergence of the AGG model.\\

\noindent Step 1 (Uniform estimates): From the energy inequality \eqref{nloc:ene:ineq:sec2} and the assumptions on the initial data, there exists a positive constant $C$ independent of $\kappa$ such that 
\begin{align*}
& \|e_\kappa(\phi_\kappa)\|_{L^\infty(0,T)} + \| \bu_{\kappa} \|_{L^\infty(0,T;\HH_\sigma) \cap L^2(0,T;\HH^1_{\sigma})}\\
& \quad + \| \bo_\kappa \|_{L^\infty(0,T;\HH) \cap L^2(0,T;\HH^1)} + \| \nabla \mu_\kappa \|_{L^2(0,T;L^2)} \leq C.
\end{align*}
As $|\phi_\kappa| < 1$ a.e.~in $Q$, it holds that $\phi_\kappa$ is bounded in $L^\infty(0,T;L^2(\Omega))$, as well as
\[
\| \pd_t \phi_\kappa \|_{L^2(0,T;H^1(\Omega)^*)} \leq  \| \nabla \mu_\kappa \|_{L^2(0,T;L^2)} + \| \phi_\kappa \bu _\kappa \|_{L^2(0,T;L^2)} \leq C.
\]
Since $\overline{\phi}_\kappa(t) = \overline{\phi}_0$ for all $t \in (0,T]$, we consider for a.e.~$t \in (0,T)$, the unique solution $q_{\kappa}(t) \in H^1_{(0)}(\Omega)$, i.e., $q_\kappa(t) \in H^1(\Omega)$ with $\overline{q}_\kappa(t) = 0$, satisfying
\[
(m(\phi_\kappa(t)) \nabla q_\kappa(t) \cdot \nabla \zeta) = (\phi_\kappa(t) - \overline{\phi}_0, \zeta) \quad \forall \zeta \in H^1(\Omega).
\]
As $m$ is strictly bounded from below, we find that there exists a positive constant $C$ independent of $\kappa$ such that 
\[
\|q_\kappa(t) \|_{H^1} \leq C \| \phi_\kappa(t) - \overline{\phi}_0 \| \leq C.
\]
In particular, $q_\kappa \in L^\infty(0,T;H^1_{(0)}(\Omega))$. Then, choosing $\psi = q_\kappa(t)$ in \eqref{weak:form:phi} and testing \eqref{weak:nloc:mu} with $\phi_\kappa(t) - \overline{\phi}_0$ yields after summing:
\begin{align*}
0 & = \langle \pd_t \phi_\kappa, q_\kappa \rangle_{H^1(\Omega)} - (\phi_\kappa \bu_\kappa, \nabla q_\kappa) + (\phi_\kappa - \overline{\phi}_0, \mu_\kappa) \\
& = \langle \pd_t \phi_\kappa, q_\kappa \rangle_{H^1(\Omega)} - (\phi_\kappa \bu_\kappa, \nabla q_\kappa) + (F_1'(\phi_\kappa), \phi_\kappa - \overline{\phi}_0) + (a_\kappa \phi_\kappa - K_\kappa \star \phi_\kappa + F_2'(\phi_\kappa), \phi_\kappa - \overline{\phi}_0) \\
& = \langle \pd_t \phi_\kappa, q_\kappa \rangle_{H^1(\Omega)} - (\phi_\kappa \bu_\kappa, \nabla q_\kappa) + (F_1'(\phi_\kappa), \phi_\kappa - \overline{\phi}_0) + 2e_\kappa(\phi_\kappa) + (F_2'(\phi_\kappa), \phi_\kappa - \overline{\phi}_0),
\end{align*}
where we have used the relation
\begin{align*}
& (a_\kappa \phi_\kappa - K_\kappa \star \phi_\kappa, \phi_\kappa - \overline{\phi}_0) \\
& \quad = \int_\Omega (K_\kappa \star 1) |\phi_\kappa|^2   - (K_\kappa \star \phi_\kappa) \phi_\kappa \, dx - \overline{\phi}_0 \int_\Omega (K_\kappa \star 1) \phi_\kappa - K_\kappa \star \phi_\kappa \, dx \\
& \quad = \frac{1}{2} \int_\Omega \int_\Omega K_\kappa(x-y) (\phi_\kappa(x) - \phi_\kappa(y))^2 \, dx \, dy = 2 e_\kappa(\phi_\kappa)
\end{align*}
thanks to the symmetry of $K_\kappa$. Invoking \eqref{F'abs:ineq} with $r = \overline{\phi}_0 \in (-1,1)$ and $s = \phi_\kappa$, the continuity of $F_2'$ and the boundedness of $\phi_\kappa$, there exist constants $c_1, C > 0$ and $c_2 \geq 0$ such that 
\begin{align*}
2 e_\kappa(\phi_\kappa) + c_1 \| F_1'(\phi_\kappa) \|_{L^1} \leq c_2 + (\| \pd_t \phi_\kappa \|_{H^1(\Omega)^*} + \| \bu_\kappa \| \big )\| q_\kappa \|_{H^1} + C.
\end{align*}
Neglecting the non-negative term $e_\kappa(\phi_\kappa)$ on the left-hand side we find that $F_1'(\phi_\kappa)$ is bounded in $L^2(0,T;L^1(\Omega))$, whence the mean value of $\mu_\kappa$ can now be estimated as
\begin{align*}
    \|\overline{\mu}_\kappa \|_{L^2(0,T)} \leq C \| F'(\phi_\kappa) \|_{L^2(0,T;L^1(\Omega))} \leq C.
\end{align*}
By the Poincar\'e inequality we then deduce that 
\[
\| \mu_\kappa \|_{L^2(0,T;H^1(\Omega))} \leq C.
\]
Then, testing \eqref{weak:nloc:mu} with $F_1'(\phi_\kappa)$ and using the symmetry of $K_\kappa$:
\begin{align*}
& \| F_1'(\phi_\kappa) \|^2 + \frac{1}{2} \int_\Omega \int_\Omega K_\kappa(x-y) (\phi_\kappa(x) - \phi_\kappa(y))(F_1'(\phi_\kappa(x)) - F_1'(\phi_\kappa(y))) \, dy \, dx \\
& \quad = (\mu_\kappa - F_2'(\phi_\kappa), F_1'(\phi_\kappa)) \leq \frac{1}{2} \| F_1'(\phi_\kappa) \|^2 + C(\| \mu_\kappa \|^2 + \| F_2'(\phi_\kappa) \|^2).
\end{align*}
By the monotonicity of $F_1'$, the second term on the left-hand side is non-negative.  Then, integrating over $(0,T)$ yields
\[
\| F_1'(\phi_\kappa) \|_{L^2(Q)} \leq C.
\]
By a comparison of terms in \eqref{weak:nloc:mu} we then find that 
\[
\| a_\kappa \phi_\kappa - K_\kappa\star \phi_\kappa \|_{L^2(Q)} \leq C.
\]
From \eqref{SobEmd2} we have $L^\infty(0,T;L^2(\Omega)) \cap L^2(0,T;H^1(\Omega)) \subset L^r(0,T;H^{2/r}(\Omega))$ and for $r = \frac{10}{3}$ it holds that 
\[
\| \bu_\kappa \|_{L^{\frac{10}{3}}(Q)} + \| \bo_\kappa \|_{L^{\frac{10}{3}}(Q)} \leq C.
\]
We now focus on the time derivative $\pd_t(\rho(\phi_\kappa) \bo_\kappa)$ and remark that the argument for $\pd_t(\rho(\phi_\kappa)\bu_\kappa)$ is similar. From the boundedness of $\bm{J}_\kappa := -\rho'(\phi_\kappa) m(\phi_\kappa) \nabla \mu_\kappa$ in $L^2(Q)^d$ we first infer that $\bm{J}_\kappa \otimes \bo_\kappa$ is bounded in $L^{\frac{8}{7}}(0,T;L^{\frac{4}{3}}(\Omega)^{d \times d})$. Then, with the boundedness of $\rho(\phi_\kappa) \bu_\kappa \otimes \bo_\kappa$ in $L^2(0,T;L^{3/2}(\Omega)^{d \times d})\cap L^{\frac{5}{3}}(Q)^{d \times d}$ and of $\bo_\kappa$ in $L^2(0,T;\HH^1)$ we deduce that 
\[
\| \pd_t(\rho(\phi_\kappa) \bo_\kappa) \|_{L^{\frac{8}{7}}(0,T;(W^{1,4}(\Omega)^d)^*)} \leq C.
\]

\noindent Step 2 (Compactness): The weak/weak* compactness assertions for $\bu_\kappa$, $\bo_\kappa$, $\mu_\kappa$ follow from the application of the Banach--Alaoglu theorem. By the compact embedding of $L^2(\Omega) \Subset H^1(\Omega)^*$ and the Aubin--Lions lemma, we have that $\phi_\kappa \to \phi$ strongly in $C^0([0,T];H^1(\Omega)^*)$.  Then, Lemma \ref{sec2:lem:phi:e:str:con:L2} and the boundedness of $e_\kappa(\phi_\kappa)$ yields the strong convergence of $\phi_\kappa$ in $C^0([0,T];L^2(\Omega))$.

Since the compactness arguments for $(\bu_\kappa, \phi_\kappa, \mu_\kappa)$ can be found in \cite{AbelsTera}, let us focus more on the micro-rotation $\bo_\kappa$. Denoting by $(\bu, \bo, \phi, \mu)$ the corresponding weak limits of $(\bu_\kappa, \bo_\kappa, \phi_\kappa, \mu_\kappa)$ along a nonrelablled subsequence $\kappa \to 0$, we first observe the weak/weak* convergence $\bo_\kappa \to \bo$ in $L^\infty(0,T;\HH) \cap L^2(0,T;\HH^1) \cap L^{\frac{10}{3}}(Q)$. Combining with the strong convergence $\rho(\phi_\kappa) \to \rho(\phi)$ in $L^q(Q)$ for any $q \in [2,\infty)$ we deduce, in a similar fashion as in the proof of Lemma \ref{lem:exist:weaksol}, that
\begin{equation*}
\begin{alignedat}{3}
   \rho(\phi_\kappa) \bo_\kappa & \to \rho(\phi) \bo && \quad \text{ weakly in } L^2(Q), \\
   \bo_\kappa & \to \bo && \quad \text{ strongly in } L^{\frac{10}{3}-r}(Q) \text{ for } r \in (0,\tfrac{7}{3}] \text{ and a.e.~in } Q.
\end{alignedat}
\end{equation*}
From the compact embedding of $\HH \Subset (\HH^1)^*$ and the continuous embedding $(\HH^1)^* \subset (W^{1,4}(\Omega)^d)^*$, the Aubin--Lions lemma asserts the strong convergence $\rho(\phi_\kappa) \bo_\kappa \to \rho(\phi) \bo$ in $L^2(0,T;(\HH^1)^*)$. Together with the weak convergence of $\bo_\kappa$ to $\bo$ in $L^2(0,T;\HH^1)$ we have the norm convergence of $(\rho(\phi_\kappa))^{1/2} \bo_\kappa$ to $(\rho(\phi))^{1/2} \bo$ in $L^2(Q)^d$. This then yields the strong convergence of $(\rho(\phi_\kappa))^{1/2} \bo_\kappa$ in $L^2(Q)^d$, and subsequently
\[
\bo_\kappa \to \bo \text{ strongly in } L^2(Q)^d.
\]
By a similar argument we have $\bu_\kappa \to \bu$ strongly in $L^2(Q)^d$. From these we obtain the strong convergence of $\bu_\kappa \otimes \bo_\kappa \to \bu \otimes \bo$ in $L^1(Q)^{d \times d}$, and thus $\rho(\phi_\kappa) \bu_\kappa \otimes \bo_\kappa \to \rho(\phi_\kappa) \bu \otimes \bo$ weakly in $L^{\frac{5}{3}-s}(Q)^{d \times d}$ for $s \in (0,\frac{2}{3}]$, as well as $\bm{J}_\kappa \otimes \bo_\kappa \to \bm{J} \otimes \bo$ weakly in $L^1(Q)^{d \times d}$, where $\bm{J} = -\rho'(\phi) m(\phi) \nabla \mu$. Then, we can pass to the limit $\kappa \to 0$ in \eqref{weak:form:bo:sec2} for $(\bu_\kappa, \bo_\kappa, \phi_\kappa, \mu_\kappa)$ to obtain the analogue of \eqref{weak:form:bo:sec2} for the limit functions. Passing to the limit $\kappa \to 0$ in the Navier--Stokes equations \eqref{weak:form:bu:sec2} and the phase field equation \eqref{weak:form:phi} can be done in a similar fashion. 

It remains to study the limit of \eqref{weak:nloc:mu} as $\kappa \to 0$. The boundedness of $F_1'(\phi_\kappa)$ and $a_\kappa \phi_\kappa - K_\kappa \star \phi_\kappa$ in $L^2(Q)$ yield limit functions $\xi$ and $\zeta$ such that $F_1'(\phi_\kappa) \to \xi$ weakly and $a_\kappa \phi_\kappa - K_\kappa \star \phi_\kappa \to \zeta$ weakly in $L^2(Q)$. From the strong convergence of $\phi_\kappa \to \phi$ in $C^0([0,T];L^2(\Omega))$ we have $F_1'(\phi_\kappa) \to F_1'(\phi)$ and $F_2'(\phi_\kappa) \to F_2'(\phi)$ almost everywhere in $Q$. The former enables us to identify $\xi$ as $F_1'(\phi)$ via a standard argument, see e.g.~\cite[p.~1093]{AbelsBosia}, while the latter provides $F_2'(\phi_\kappa) \to F_2'(\phi)$ strongly in $L^2(Q)$. Hence, passing to the limit $\kappa \to 0$ in \eqref{weak:nloc:mu} leads to 
\[
\mu = \zeta + F'(\phi) \text{ a.e.~in } Q.
\]
In order to identify $\zeta$ we apply \eqref{nloc:to:loc:4} of Lemma \ref{sec2:lem:nloc:to:loc:ene} and recall the uniform boundedness of $e_\kappa(\phi_\kappa)$ in $L^\infty(0,T)$ to deduce that 
\[
\| e_0(\phi) \|_{L^\infty(0,T)} \leq \liminf_{\kappa \to 0} \| e_\kappa(\phi_\kappa)\|_{L^\infty(0,T)} \leq C.
\]
From this we see that $\phi \in L^\infty(0,T;H^1(\Omega))$. Furthermore, for any $\psi \in L^2(0,T;H^1(\Omega))$, the convexity of $e_\kappa(\cdot)$ yields that
\[
e_\kappa(\psi) \geq e_\kappa(\phi_\kappa) + \int_\Omega (a_\kappa \phi_\kappa- (K_\kappa \star \phi_\kappa)) (\psi - \phi_\kappa) \, dx.
\]
Integrating over $(0,T)$ and using the strong convergence of $\phi_\kappa \to \phi$ in $C^0([0,T];L^2(\Omega))$ and the weak convergence of $a_\kappa \phi_\kappa - K_\kappa \star \phi_\kappa \to \zeta$ in $L^2(Q)$, we have by \eqref{nloc:to:loc:conv1} and \eqref{nloc:to:loc:4} that
\[
\int_0^T \frac{1}{2} \|\nabla \psi\|^2 - \frac{1}{2} \|\nabla \phi\|^2 \, dt =  \int_0^T e_0(\psi) \, dt - \int_0^T e_0(\phi) \, dt \geq \int_0^T \int_\Omega \zeta (\psi - \phi) \, dx \, dt.
\]
As $\psi \in L^2(0,T;H^1(\Omega))$ is arbitrary, choosing $\psi = \phi + h \chi(t) \varphi(x)$ for $h \in \R$, $\chi \in C^0([0,T])$ and $\varphi \in H^1(\Omega)$ yields after sending $h \to 0$:
\begin{align*}
\int_0^T \chi(t) \int_\Omega \nabla \phi \cdot \nabla \varphi \, dx \, dt = \int_0^T \chi(t) \int_\Omega \zeta \varphi \, dx \, dt.
\end{align*}
This shows that $\zeta = - \Delta \phi$ a.e.~in $Q$ and by elliptic regularity we have $\phi \in L^2(0,T;H^2_n(\Omega))$.

To finish, the energy inequality \eqref{energy:loc:ineq} can be obtained by passing to the limit $\kappa \to 0$ in \eqref{nloc:ene:ineq:sec2}, while utilising the assumption on the initial energy $\mathcal{E}_{\kappa,\nloc}(\bu_{0,\kappa}, \bo_{0,\kappa}, \phi_{0,\kappa}) \to \mathcal{E}_{\loc}(\bu_0, \bo_0, \phi_0)$, the weak lower semicontinuity of the norms, as well as the properties $\bu_\kappa(t) \to \bu(t)$ in $\HH_\sigma$, $\bo_\kappa(t) \to \bo(t)$ in $\HH$ and $\phi_{\kappa}(t) \to \phi(t)$ in $L^2(\Omega)$ for a.e.~$t \in (0,T)$.
\end{proof}

\section{Strong well-posedness in two dimensions}\label{sec:str}
In this section we consider the two-dimensional setting $d = 2$ and set the mobility $m(\phi)$ to be a constant, which we take to be equal to unity. To establish global strong solutions to the nMAGG model \eqref{nonlocal:model:equ:2D}, we follow the approach of \cite{CCGAGMG}, which utilises a similar strategy to \cite{CHNS3Dstrsol,CHNS2Dstrsol} for the AGG model and to \cite{ChanLamstrsol} for the MAGG model. The result is formulated as follows.

\begin{thm}\label{thm:strwellposed}
Suppose that \eqref{gen:ass1}--\eqref{gen:ass3} and \eqref{sing:pot:ass1}--\eqref{sing:pot:ass6} hold, along with $\bu_0 \in \HH^1_{\sigma}$, $\omega_0 \in H^1_0(\Omega)$ and $\phi_0 \in H^1(\Omega)$ with $|\overline{\phi}_0| < 1$, $F'(\phi_0) \in L^2(\Omega)$ and $F''(\phi_0) \nabla \phi_0 \in L^2(\Omega;\R^2)$. Then, for a fixed but arbitrary $T > 0$, there exists a global strong solution $(\bu, \bo, p, \phi, \mu)$ to the nMAGG model on the time interval $[0,T]$ in the following sense:
\begin{itemize}
    \item Regularity
    \begin{align*}
\bu & \in C([0,T];\HH^1_{\sigma}) \cap L^2(0,T;\HH^2_{\sigma}) \cap H^1(0,T;\HH_\sigma), \\
\omega & \in C([0,T];\HH^1) \cap L^2(0,T;\HH^2) \cap H^1(0,T;\HH), \\
p & \in L^2(0,T;H^1_{(0)}(\Omega)), \\
\phi & \in C_w([0,T];H^1(\Omega)) \cap L^{\frac{2p}{p-2}}(0,T;W^{1,p}(\Omega)), \quad \forall p \in (2,\infty), \\
\phi & \in L^\infty(0,T;L^\infty(\Omega)) \text{ such that } |\phi(x,t)| < 1 \text{ for a.e.~} x \in \Omega, \, \forall t \in [0,T], \\
\pd_t \phi & \in L^\infty(0,T;H^1(\Omega)^*) \cap L^2(0,T;L^2(\Omega)), \\
F'(\phi) & \in L^\infty(0,T;H^1(\Omega)), \\
\mu & \in C_w([0,T];H^1(\Omega)) \cap L^2(0,T;H^2(\Omega)) \cap H^1(0,T;H^1(\Omega)^*).
\end{align*}
\item $(\bu, \omega, p, \phi, \mu)$ satisfy \eqref{nonlocal:model:equ:2D} a.e.~in $\Omega \times (0,T)$ with mobility $m = 1$ and the boundary condition $\pdnu \mu = 0$ almost everywhere on $\pd \Omega \times (0,T)$.
\item Initial conditions
\[
(\bu, \omega, \phi) \vert_{t=0} = (\bu_0, \omega_0, \phi_0).
\]
\item If there exists $\delta_0 \in (0,1)$ such that $\| \phi_0 \|_{L^\infty(\Omega)} \leq 1-\delta_0$, then $(\bu, \omega,  \phi )$ is the unique strong solution and depends continuously on the initial data in the interval $[0,T]$.
\end{itemize}
\end{thm}

\begin{remark}
The assertions of Theorem \ref{thm:strwellposed} also hold if we consider periodic boundary conditions, i.e., $\Omega = \mathbb{T}^2$ is the 2-torus.
\end{remark}

Let us report on two essential 
preliminary results from \cite{CCGAGMG}. The first is a $L^4$-estimate for the pressure in the stationary Stokes problem. 

\begin{lem}[cf.~Lemma 3.1 of \cite{CCGAGMG}] \label{L4:pre:estim}
Let $\Omega \subset \R^{2}$ be a bounded domain with $C^2$ boundary. For $\bm{f} \in \HH_\sigma$, consider the solution $(\bu, p) \in \HH^2_{\sigma} \times H^1_{(0)}(\Omega)$ to the Stokes problem 
\begin{align*}
- \Delta \bu + \nabla p & = \bm{f} \text{ in } \Omega, \\
\bu & = \bm{0} \text{ on } \pd \Omega.
\end{align*}
Then, there exists a positive constant $C$ such that
\begin{equation}\label{pressure:L4}
\| p \|_{L^{4}(\Omega)} \leq C\| \nabla \bm{A}_S^{-1}\bm{f} \|^{\frac{1}{2}}\| \bm{f} \|^{\frac{1}{2}}, \quad \forall \bm{f} \in \HH_\sigma,
\end{equation}
where $\bm{A}_S$ is the Stokes operator with no-slip boundary conditions.
\end{lem}

The second result concerns the regularity of the two-dimensional nonlocal convective Cahn--Hilliard equation with singular potential and constant mobility:
\begin{subequations} \label{sec2prere:nlocCH}
\begin{alignat}{2}
0& = \pd_t\phi + \bm{v} \cdot \nabla\phi - \Delta\mu  && \quad \text{ in } \Omega \times (0,T), \\
\mu & = F'(\phi) + a\phi - K \star\phi && \quad \text{ in } \Omega \times (0,T),  \\
0 & = \pd_{\bnu}\mu  && \quad \text{ on } \partial\Omega \times (0,T),\\ 
\phi(0) & = \phi_0 && \quad \text{ in } \Omega,
\end{alignat}
\end{subequations}
where $\bm{v}$ is divergence-free. We remark that \eqref{sec2prere:nlocCH} differs with \cite[(4.1)]{CCGAGMG} by the appearance of the lower order term $a \phi$, and hence the assertions of the following result are modified to account for the difference.

\begin{lem}[cf.~Theorem 4.1 of \cite{CCGAGMG}] \label{sec2prere:str:sol:reg:CH}
Let \eqref{gen:ass1}--\eqref{gen:ass3}  and \eqref{sing:pot:ass1}--\eqref{sing:pot:ass6} hold and let $T>0$ be arbitrary. Assume that $\bm{v} \in L^{4}( 0,T;\mathbb{L}_{\sigma}^{4})$, $\phi_0 \in H^1(\Omega) \cap L^{\infty}(\Omega)$ with $|\overline{\phi}_0| < 1$, $F(\phi_{0}) \in L^1(\Omega)$, $F'(\phi_0) \in L^2(\Omega)$, $F''(\phi_0) \nabla \phi_0 \in L^2(\Omega;\R^2)$. Then, there exists a unique strong solution to \eqref{sec2prere:nlocCH} with
\begin{align*}
\phi & \in L^\infty(0,T;L^\infty(\Omega)) \text{ such that } |\phi(x,t)| < 1 \text{ for a.e.~} x \in \Omega, \ \forall t \in [0,T], \\
    \phi & \in L^\infty(0,T;H^1(\Omega)) \cap L^{2p/(p-2)}(0,T;W^{1,p}(\Omega)), \quad \forall p \in (2,\infty), \\
    \pd_t \phi & \in L^4(0,T;H^1(\Omega)^*) \cap L^2(0,T;L^2(\Omega)), \\
    \mu & \in L^\infty(0,T;H^1(\Omega)) \cap L^2(0,T;H^2_n(\Omega)) \cap H^1(0,T;H^1(\Omega)^*), \\
    F'(\phi) & \in L^\infty(0,T;H^1(\Omega)), \quad F''(\phi) \in L^\infty(0,T;L^p(\Omega)), \quad \forall p \in [2,\infty).
\end{align*}
Moreover, there exists a positive constant $C$ depending only on $K$ and $\Omega$ such that 
\begin{align}
& \| \nabla \mu \|_{L^2(Q)}^2 \leq C + \| \bm{v} \|_{L^2(Q)}^2, \label{str:sol:CH:grad:mu:L2L2} \\
\label{str:sol:CH:grad:mu:Linfty:L2:est} & \| \nabla \mu \|_{L^\infty(0,T;L^2(\Omega))} 
 \leq \Big ( \| \nabla \mu_0 \|+ C \big [ \| \bm{v} \|_{L^2(Q)} + \| \nabla \mu \|_{L^2(Q)} \big ] \Big ) \exp \Big ( C \| \bm{v} \|_{L^4(Q)}^4 \Big ) =: \Xi_1, \\
\notag & \| \pd_t \phi \|_{L^2(Q)}^2 + \| \nabla \mu \|_{L^2(0,T;H^1(\Omega))}^2 \\
\notag & \quad \leq C \Big ( \|\nabla \mu_0 \|^2 + C \big [ \| \bm{v} \|_{L^2(Q)}^2 + \| \nabla \mu \|_{L^2(Q)}^2 \big ] \Big )  \Big ( 1 + \| \bm{v} \|_{L^4(Q)}^4 \Big ) \exp \Big ( C \| \bm{v} \|_{L^4(Q)}^4 \Big ) \\
\label{str:sol:tdphi:L2L2:CH:grad:mu:L2H1:est} & \quad =: \Xi_2,
\end{align}
where $\mu_0 := F'(\phi_0) + a \phi_0 - K \star \phi_0$, 
and there exists a positive constant $C$ depending only on $K$, $\Omega$, $\Xi_1$, $\Xi_2$, $\overline{\phi}_0$ and $T$ such that 
\begin{subequations}\label{str:sol:phi:mu:F:bdds}
\begin{alignat}{2}
& \| \mu \|_{L^\infty(0,T;H^1(\Omega))} + \| \phi \|_{L^\infty(0,T;H^1(\Omega))} + \| F'(\phi) \|_{L^\infty(0,T;H^1(\Omega))} \leq C, \\
& \| F''(\phi) \|_{L^\infty(0,T;L^p(\Omega))} \leq C \quad \forall p \in [2, \infty), \\
& \| \phi \|_{L^{2p/(p-2)}(0,T;L^p(\Omega))} \leq C \quad \forall p \in (2,\infty), \\
& \| \mu \|_{L^2(0,T;H^2(\Omega)) \cap H^1(0,T;H^1(\Omega)^*)} \leq C, \\
& \| \pd_t \phi \|_{L^\infty(0,T;H^1(\Omega)^*)} \leq C \quad \text{ if } \bm{v} \in L^\infty(0,T;\HH_\sigma).
\end{alignat}
\end{subequations}
If in addition $\| \phi_0 \|_{L^\infty(\Omega)} \leq 1 - \delta_0$ for some $\delta_0 \in (0,1)$, then there exists $\delta > 0$ such that 
\[
\sup_{t \in [0,T]} \| \phi(t) \|_{L^\infty(\Omega)} \leq 1 - \delta.
\]
Consequently, it also holds that $\pd_t \mu \in L^2(0,T;L^2(\Omega))$ and $\mu \in C^0([0,T];H^1(\Omega))$.
\end{lem}

\subsection{Semi-Galerkin scheme}
We again focus on the case where $\Omega$ is a bounded domain and discuss the modifications needed for the case where $\Omega = \mathbb{T}^2$ afterwards. Recalling from the proof of Lemma \ref{lem:exist:weaksol} the family of eigenfunctions $\{\bm{X}_j\}_{j \in \N}$ of the Stokes operator $\bm{A}_S$, as well as the associated finite dimensional subspace $\mathbb{V}_{u,n}$ and projection operator $\mathbb{P}_{u,n}$, due to the $C^\infty$ regularity of $\pd \Omega$, it follows from regularity theory that $\bm{X}_j \in \HH^2_{\sigma}$ for all $j \in \N$. Furthermore, there exist positive constants $C_n$ depending on $n$ such that 
\begin{align}
\label{inv:est:Gal:app}
\| \bm{f} \|_{H^1} \leq C_n \| \bm{f} \| , \quad \| \bm{f} \|_{H^2} \leq C_n \| \bm{f} \| \quad \forall \bm{f} \in \mathbb{V}_{u,n}.
\end{align}
With an abuse of notation, we use $\{Y_j\}_{j \in \N}$ to denote the family of eigenfunctions corresponding to the scalar Dirichlet--Laplacian operator $A_1$. Denoting by $\mathbb{V}_{w,n}$ as the finite dimensional subspace spanned by the first $n$ eigenfunctions and $\mathbb{P}_{w,n}$ as the associated projection operator, by regularity theory we have that $Y_j \in H^2(\Omega) \cap H^1_0(\Omega)$ and similarly, there exist positive constants $C_n$ such that 
\[
\| f \|_{H^1} \leq C_n \| f \|_{L^2}, \quad \| f \|_{H^2} \leq C_n \| f \|_{L^2} \quad \forall f \in \mathbb{V}_{w,n}.
\]
We first state a continuity property for the strong solutions of the nonlocal convective Cahn--Hilliard equation where $\bm{v} \in C^0([0,T];\mathbb{V}_{\bu,n})$, which has been established in Section 5.2 of \cite{CCGAGMG}:
\begin{prop}
Let $\{\bm{v}_k\}_{k \in \N} \subset C^0([0,T];\mathbb{V}_{\bu,n})$ be a sequence such that $\bm{v}_k \to \bm{v}_*$ in $C^0([0,T];\mathbb{V}_{\bu,n})$, and let $(\phi_k, \mu_k)_{k \in \N}$ and $(\phi_*, \mu_*)$ denote the corresponding strong solutions to the nonlocal convective Cahn--Hilliard equation associated to $\bm{v}_k$ and $\bm{v}_*$, respectively. Then, there exists a positive constant $C$ independent of $k \in \N$ such that 
\begin{subequations}\label{CH:uni:est:Gal:diff:2Dstr}
\begin{alignat}{2}
\| \phi_k \|_{L^\infty(0,T;H^1(\Omega)) \cap H^1(0,T;L^2(\Omega)) \cap W^{1,\infty}(0,T;H^1(\Omega)^*)} & \leq C, \\
\| \phi_* \|_{L^\infty(0,T;H^1(\Omega)) \cap H^1(0,T;L^2(\Omega)) \cap W^{1,\infty}(0,T;H^1(\Omega)^*)} & \leq C, \\
\| \mu_k \|_{L^\infty(0,T;H^1(\Omega)) \cap L^2(0,T;H^2(\Omega)) \cap H^1(0,T;H^1(\Omega)^*)} 
& \leq C, \\
\| \mu_* \|_{L^\infty(0,T;H^1(\Omega)) \cap L^2(0,T;H^2(\Omega)) \cap H^1(0,T;H^1(\Omega)^*)} & \leq C, \\
\| F' (\phi_k) 
\|_{L^\infty(0,T;H^1(\Omega))} + \| F''(\phi_k) \|_{L^\infty(0,T;L^p(\Omega))} & \leq C, \\
\| F' (\phi_*) 
\|_{L^\infty(0,T;H^1(\Omega))} + \| F''(\phi_*) \|_{L^\infty(0,T;L^p(\Omega))} & \leq C,
\end{alignat}
\end{subequations}
for any $p \in [2,\infty)$. Moreover, as $k \to \infty$,
\begin{align}
\| \phi_k - \phi_* \|_{L^\infty(0,T;H^1(\Omega)^*)} + \| \phi_k - \phi_* \|_{L^4(0,T;L^4(\Omega))} + \| \mu_k - \mu_* \|_{L^2(0,T;L^2(\Omega))} \to 0.
\end{align}
\end{prop}
Fixing $n \in \N$, we seek an approximate solution $(\bu_n, \bo_n, \phi_n, \mu_n)$ to the following system
\begin{subequations} \label{equ:Gal:app:2Dstr}
\begin{alignat}{2}
& \label{equ:bu:Gal:app:2Dstr} \left( \rho\left( \phi_{n} \right)\pd_{t}\bu_{n},\bv \right) + \left( \rho\left( \phi_{n} \right)\left( \bu_{n} \cdot \nabla \right)\bu_{n},\bv \right) + \left( 2\eta\left( \phi_{n} \right)\D \bu_{n},\nabla \bv \right) \\ & \notag \qquad + \left( 2\eta_{r}\left( \phi_{n} \right)\W\bu_{n},\nabla \bv \right)  - \frac{\rho_{1} - \rho_{2}}{2}\left( \left( \nabla\mu_{n} \cdot \nabla \right)\bu_{n},\bv \right) \\ & \notag \quad = - \left( \phi_{n}\nabla\mu_{n},\bv \right) + \left( 2\curl_1 \left( \eta_{r}\left( \phi_{n} \right)\omega_{n} \right),\bv \right), \\ 
& \label{equ:bo:Gal:app:2Dstr} \left( \rho\left( \phi_{n} \right)\pd_{t}\omega_{n},z \right) + \left( \rho\left( \phi_{n} \right) ( \bu_{n} \cdot \nabla \omega_{n}),z \right)  \\ 
& \notag \qquad + ( c_{d,a}(\phi_n)  \nabla \omega_n, \nabla z) - \frac{\rho_{1} - \rho_{2}}{2} (  \nabla\mu_{n} \cdot \nabla \omega_{n},z) \\ & \notag \quad = \left( 2\eta_{r}(\phi_n )\curl_2 \bu_{n},z \right) - \left( 4\eta_{r} (\phi_{n})\omega_{n},z \right), \\
& \label{equ:CH:Gal:app:2Dstr} \pd_{t}\phi_{n} + \bu_{n} \cdot \nabla\phi_{n} = \Delta\mu_{n} \text{ a.e. in } \Omega, \\
& \label{equ:mu:Gal:app:2Dstr} \mu_{n} = F'\left( \phi_{n} \right) + a\phi_{n} - K \star\phi_{n} \text{ a.e. in } \Omega,
\end{alignat}
\end{subequations}
for all $t \in (0,T)$ and for all $\bv \in \VV_{\bu,n}$ and $z \in \VV_{\omega,n}$, together with the following initial-boundary conditions:
\begin{subequations} \label{bcic:Gal:app:2Dstr}
\begin{alignat}{2}
\bu_{n} = \bm{0}, \quad  \omega_{n} = 0, \quad \pdnu \mu_{n} = 0 & \quad \text{ on } \Sigma, \\
\bu_n(0) = \mathbb{P}_{u,n}(\bu_0), \quad \omega_n(0) = \mathbb{P}_{w,n}(\omega_0), \quad \phi_n(0) = \phi_0 & \quad \text{ in } \Omega.
\end{alignat}
\end{subequations}

\begin{prop}\label{E&U:Gal:app:str:sol:nloc:MAGG}
For any $T> 0$ and for all $n \in \N$, there exists at least one solution quadruple $(\bu_n, \omega_n, \phi_n, \mu_n)$ to \eqref{equ:Gal:app:2Dstr}--\eqref{bcic:Gal:app:2Dstr} with the following regularities:
\begin{align*}
\bu_n & \in H^1(0,T;\mathbb{V}_{\bu,n}), \\
\omega_n & \in H^1(0,T;\mathbb{V}_{\omega,n}), \\
\phi_n & \in L^\infty(Q) \text{ such that } |\phi_n(x,t)| < 1 \text{ a.e.~in } \Omega \times (0,T), \\
\phi_n & \in L^\infty(0,T;H^1(\Omega)) \cap L^{2p/(p-2)}(0,T;W^{1,p}(\Omega)), \quad \forall p \in (2,\infty), \\
\pd_t \phi_n & \in L^\infty(0,T;H^1(\Omega)^*) \cap L^2(0,T;L^2(\Omega)), \\
\mu_n & \in L^\infty(0,T;H^1(\Omega)) \cap L^2(0,T;H^2_n(\Omega)) \cap H^1(0,T;H^1(\Omega)^*), \\
F'(\phi_n) & \in L^\infty(0,T;H^1(\Omega)), \quad F''(\phi_n) \in L^\infty(0,T;L^p(\Omega)), \quad \forall p \in [2, \infty).
\end{align*}
\end{prop}

\begin{proof}
We fix a $\bm{U} \in H^1([0,T];\mathbb{V}_{u,n}) \subset L^4(0,T;\mathbb{L}_{\sigma}^4)$ and define $(\phi_U, \mu_U)$ to be the unique strong solution to the nonlocal convective Cahn--Hilliard equation \eqref{sec2prere:nlocCH} with $\bm{v} = \bm{U}$. Then, we consider Galerkin approximations $(\bu^n, \omega^n)_{n \in \N}$ of the form
\[
\bm{u}^n(x,t) = \sum_{j=1}^n a_j^n(t) \bm{X}_j(t), \quad \omega^n(x,t) = \sum_{j=1} b_j^n(t) Y_j(x), 
\]
satisfying
\begin{subequations} \label{equ:Gal:app:2Dstr:fixedpoint}
\begin{alignat}{2}
& \label{equ:bu:Gal:app:2Dstr:fixedpoint} \left( \rho (\phi_U) \pd_{t}\bu^n,\bm{X}_k \right) + \left( \rho(\phi_U) \left( \bm{U} \cdot \nabla \right)\bu^n,\bm{X}_k \right) + \left( 2\eta ( \phi_U )\D \bu^n,\nabla \bm{X}_k \right) \\ & \notag \qquad + \left( 2\eta_{r}(\phi_U) \W\bu^n,\nabla \bm{X}_k \right)  - \frac{\overline{\rho}_{1} - \overline{\rho}_{2}}{2}\left( \left( \nabla\mu_U \cdot \nabla \right)\bu^n,\bm{X}_k \right) \\ & \notag \quad = - \left( \phi_U\nabla\mu_U,\bm{X}_k \right) + \left( 2\curl_1 \left( \eta_{r}(\phi_U)\omega^n \right),\bm{X}_k \right), \\ 
& \label{equ:bo:Gal:app:2Dstr:fixedpoint} \left( \rho (\phi_U) \pd_{t}\omega^n,Y_k \right) + \left( \rho(\phi_U) ( \bm{U} \cdot \nabla \omega^n),Y_k \right)  \\ 
& \notag \qquad + ( c_{d,a}(\phi_U) \nabla \omega^n, \nabla Y_k) - \frac{\overline{\rho}_{1} - \overline{\rho}_{2}}{2} (  \nabla\mu_U \cdot \nabla \omega^n,Y_k) \\ & \notag \quad = \left( 2\eta_{r}(\phi_U )\curl_2 \bu^n,Y_k \right) - \left( 4\eta_{r} (\phi_U)\omega^n,Y_k \right), 
\end{alignat}
\end{subequations}
holding for all $k = 1, \dots, n$, and furnished with initial conditions $\bu^n(0) = \mathbb{P}_{\bu,n} (\bu_0)$ and $\omega^n(0) = \mathbb{P}_{\omega,n}(\omega_0)$.

By the continuity of $\rho$, $\eta$, $\eta_r$, $c_{d,a}$, as well as the fact that $\phi_U, \mu_U \in C_w([0,T];H^1(\Omega)) \cap C^0([0,T];L^p(\Omega))$ for any $p \in [1,\infty)$ by virtue of the regularity properties stated in Lemma \ref{sec2prere:str:sol:reg:CH}, we can reformulate \eqref{equ:Gal:app:2Dstr:fixedpoint} into a system of ordinary differential equations for the vectors of functions $\bm{A}^n(t) = (a_1^n(t), \dots, a_n^n(t))^{\top}$ and $\bm{B}^n(t) = (b_1^n(t), \dots, b_n^n(t))^{\top}$ with a right-hand side that depend continuously on $\bm{A}^n(t)$ and $\bm{B}^n(t)$, cf.~\cite{ChanLamstrsol,CCGAGMG}. Invoking the standard theory for ordinary differential systems we can deduce the existence of unique solutions $\bm{A}^n, \bm{B}^n \in C^1([0,T];\R^n)$ on the time interval $[0,T]$, leading to the existence of unique solutions $\bu^n \in C^1([0,T];\mathbb{V}_{\bu,n})$ and $\omega^n \in C^1([0,T];\mathbb{V}_{\omega,n})$ to \eqref{equ:Gal:app:2Dstr:fixedpoint}.

We introduce the mapping $\Lambda : \bm{U} \to \bu^n$, where $\bu^n$ is the first component of the unique solution pair $(\bu^n, \omega^n)$ to \eqref{equ:Gal:app:2Dstr:fixedpoint} with variables $(\phi_U, \mu_U)$, themselves as the unique solution pair to the nonlocal convective Cahn--Hilliard equation \eqref{sec2prere:nlocCH} corresponding to $\bm{U}$. Our aim is to identify a closed ball $S \subset C^0([0,T];\mathbb{V}_{\bu,n})$ such that $\Lambda : S \to S$, $\Lambda(\bm{U}) = \bu^n$, has at least one fixed point $\bu_n \in C^0([0,T];\mathbb{V}_{\bu,n})$. Such a fixed point generates a solution quadruple $(\bu_n, \omega_n, \phi_n, \mu_n)$ to the approximation system \eqref{equ:Gal:app:2Dstr}.

We start with some basic estimates. With the properties $\rho'(\phi_U) = \frac{\overline{\rho}_1 - \overline{\rho}_2}{2}$ and $\div \bm{U} = 0$, it holds that
\begin{align*}
& \int_\Omega (\pd_t \rho(\phi_U) + \div (\rho(\phi_U) \bm{U}) - \frac{\overline{\rho}_1 - \overline{\rho}_2}{2} \Delta \mu_U) (|\bu^n|^2 + |\omega^n|^2) \, dx \\
& \quad = \frac{\overline{\rho}_1 - \overline{\rho}_2}{2} \int_\Omega ( \pd_t \phi_U + \nabla \phi_U \cdot \bm{U} - \Delta \mu_U)(|\bu^n|^2 + |\omega^n|^2) \, dx = 0.
\end{align*}
Together with the identity
\begin{align*}
& \int_\Omega \rho(\phi_U) \frac{d}{dt} \frac{1}{2} |f^n|^2 + \rho(\phi_U) \bm{U} \cdot \nabla \frac{|f^n|^2}{2} \, dx \\
& \quad = \frac{d}{dt} \int_\Omega \frac{\rho(\phi_U)}{2} |f^n|^2 \, dx - \int_\Omega \frac{|f^n|^2}{2} (\pd_t \rho(\phi_U) + \bm{U} \cdot \nabla \rho(\phi_U)) \, dx
\end{align*}
valid for $f^n \in \{\bu^n, \omega^n\}$, we obtain the following identity after multiplying \eqref{equ:bu:Gal:app:2Dstr:fixedpoint} by $a_k^n$ and \eqref{equ:bo:Gal:app:2Dstr:fixedpoint} by $b_k^n$, and then summing over $k = 1, \dots, n$: 
\begin{equation}\label{diff:ineq:Gal:est:L2:bu:bo}
    \begin{aligned}
& \frac{d}{dt} \frac{1}{2} \int_\Omega \rho(\phi_U) (|\bu^n|^2 + |\omega^n|^2) \, dx + \int_\omega 4 \eta_r(\phi_U) |\tfrac{1}{2} \curl_2 \bu^n - \omega^n|^2 \, dx \\
& \qquad + \int_\Omega 2 \eta(\phi_U) |\D \bu^n|^2 + c_{d,a}(\phi_U) |\nabla \omega^n|^2 \, dx \\
& \quad = - \int_\Omega \phi_U \nabla \mu_U \cdot \bu^n \, dx,
    \end{aligned}
\end{equation}
where we have employed the relation \eqref{2D:curlu:w} to simplify the terms involving $\W \bu^n$, $\curl_2 \bu^n$ and $\curl_1 \omega^n$. For the right-hand side of \eqref{diff:ineq:Gal:est:L2:bu:bo} we invoke the definition of $\mu_U$, Korn's inequality (with constant $C_K >0$), Young's inequality for convolutions, the relation $\nabla a = \nabla K \star 1$, and the property that $|\phi_U| < 1$ a.e.~in $\Omega \times (0,T)$ to see that
\begin{align*}
& - (\phi_U \nabla \mu_U, \bu^n) = (\mu_U \nabla \phi_U, \bu^n) \\
& \quad = (\nabla F(\phi_U), \bu^n) + (a \phi_U \nabla \phi_U, \bu^n) - ((K \star \phi_U) \nabla \phi_U, \bu^n) \\
& \quad = 0 - \frac{1}{2}(|\phi_U|^2 \nabla a, \bu^n) + (\phi_U (\nabla K \star \phi_U), \bu^n) \\
& \quad \leq \frac{1}{2} \| \nabla a \| \| \bu^n \| +  \| \nabla K \|_{L^1} \| \phi_U \| \| \bu^n \| \leq \frac{3}{2} C_{K} |\Omega|^{1/2}\| \nabla K \|_{L^1(\R^2)} \| \D \bu^n \| \\
& \quad \leq \eta_* \| \D \bu^n \|^2 + \frac{9C_K^2}{16 \eta_*} |\Omega| \| \nabla K \|_{L^1(\R^2)}^2,
\end{align*}
where $\eta_*$ is the lower bound for $\eta$ from \eqref{gen:ass3}. Hence, returning to \eqref{diff:ineq:Gal:est:L2:bu:bo} we arrive at the differential inequality
\begin{align*}
\frac{d}{dt} \int_\Omega \rho(\phi_U) (|\bu^n|^2 + |\omega^n|^2) \, dx \leq \frac{9C_K^2}{8 \eta_*} |\Omega| \| \nabla K \|_{L^1(\R^2)}^2.
\end{align*}
Integrating the above inequality in time and utilizing the upper bound $\rho^*$ and lower bound $\rho_*$ of $\rho$, we get
\begin{align}\label{LinftyL2:bdd:Gal:app:bu:bo}
\max_{t \in [0,T]} \big ( \| \bu^n (t) \|^2 + \| \omega^n(t) \|^2 \big ) \leq \frac{\rho^*}{\rho_*} ( \| \bu_0 \|^2 + \| \omega_0 \|^2) + \frac{9 C_K^2 |\Omega| T}{8 \eta_* \rho_*} \| \nabla K \|_{L^1(\R^2)}^2 =: M_0^2.
\end{align}
We define the closed ball
\begin{align*}
    S = \Big \{ \bm{v} \in C^0([0,T];\mathbb{V}_{\bu,n}) \, : \, \| \bm{v} \|_{C^0([0,T];\mathbb{V}_{\bu,n})} \leq M_0 \Big \}.
\end{align*}
Then, $\Lambda : S \to S$, $\Lambda(\bm{U}) = \bu^n$ is well-defined by the estimate \eqref{LinftyL2:bdd:Gal:app:bu:bo}. Next, we multiply \eqref{equ:bu:Gal:app:2Dstr:fixedpoint} by $(a_k^n)'(t)$  and then summing over $k = 1, \dots, n$ to get
\begin{align*}
\rho_* \| \pd_t \bu^n \|^2 & \leq - (\rho(\phi_U) (\bm{U} \cdot \nabla) \bu^n, \pd_t \bu^n) - (2 \eta(\phi_U) \D \bu^n, \nabla \pd_t \bu^n) \\
& \quad - (2 \eta_r(\phi_U) \W \bu^n, \nabla \pd_t \bu^n) + \frac{\overline{\rho}_1 - \overline{\rho}_2}{2} ((\nabla \mu_U \cdot \nabla) \bu^n, \pd_t \bu^n) \\
& \quad - (\phi_U \nabla \mu_U, \pd_t \bu^n)  + (2 \curl_1(\eta_r(\phi_U) \omega^n), \pd_t \bu^n) \\
& \leq \rho^* \|\bm{U} \|_{L^4} \| \nabla \bu^n \|_{L^4} \| \pd_t \bu^n \| + 2 (\eta^* \| \D \bu^n \| +  \eta_r^* \| \W \bu^n \|) \| \nabla \pd_t \bu^n \| \\
& \quad + |\tfrac{\overline{\rho}_1 - \overline{\rho}_2}{2}| \| \nabla \mu_U \| \| \nabla \bu^n \|_{L^4} \| \pd_t \bu^n \|_{L^4} + \| \nabla \mu_U \| \| \pd_t \bu^n \| + 2 \eta_r^* \| \omega^n \| \| \nabla \pd_t \bu^n \|.
\end{align*}
Exploiting the inverse estimates \eqref{inv:est:Gal:app}, we find that 
\begin{align*}
\rho_* \| \pd_t \bu^n \|^2 & \leq \rho^* C_n \| \bm{U} \|_{L^4} \| \bu^n \| \| \pd_t \bu^n \| + 2(\eta^* + \eta_r^*) C_n^2 \| \bu^n \| \| \pd_t \bu^n \| \\
& \quad + |\tfrac{\overline{\rho}_1 + \overline{\rho}_2}{2} | C_n^2 \| \nabla \mu_U \| \| \bu^n \| \| \pd_t \bu^n \| + \| \nabla \mu_U \| \| \pd_t \bu^n \| + 2 \eta_r^* C_n \| \omega^n \| \| \pd_t \bu^n \| \\
& \leq C_n \| \pd_t \bu^n \| \Big [ \big ( 1 + \| \bm{U} \|_{L^4} + \| \nabla \mu_U \| \big ) \| \bu^n \| + \| \nabla \mu_U \| + \| \omega^n \| \Big ] \\
& \leq \frac{\rho_*}{2} \| \pd_t \bu^n \|^2 + C_n \Big [ \big (1 + \| \bm{U} \|_{L^4}^2 + \| \nabla \mu_U \|^2 \big )M_0^2 + \| \nabla \mu_U \|^2 + M_0^2 \Big ].
\end{align*}
Integrating in time over $(0,T)$ and using \eqref{str:sol:CH:grad:mu:L2L2} we deduce that
\begin{equation}\label{L2V:bdd:Gal:app:time:d:bu}
\begin{aligned}
\int_0^T \| \pd_t \bu^n \|^2 \, dt & \leq \frac{2}{\rho_*} C_n \Big ( \Big ( T + \| \bm{U} \|_{L^2(0,T;L^4)}^2 \Big ) (M_0^2 + 1) + M_0^2 \Big ) \\
& \leq \frac{2}{\rho_*} C_n \Big ( T (M_0^2 + 1)^2 + M_0^2 \Big ) =: M_1^2.
\end{aligned}
\end{equation}
Similarly, multiplying \eqref{equ:bo:Gal:app:2Dstr:fixedpoint} by $(b_k^n)'(t)$ and then summing over $k = 1, \dots, n$ yields
\begin{align*}
    \rho_* \| \pd_t \omega^n \|^2 & \leq \rho^* \| \bm{U} \|_{L^4} \| \nabla \omega^n \|_{L^4} \| \pd_t \omega^n \| + c_{d,a}^* \| \nabla \omega^n \| \| \nabla \pd_t \omega^n \| \\
    & \quad + |\tfrac{\overline{\rho}_1 - \overline{\rho}_2}{2}| \| \nabla \mu_U \| \| \nabla \omega^n \|_{L^4} \| \pd_t \omega^n \|_{L^4} + 2 \eta_r^* \| \nabla \bu^n \| \|\pd_t \omega^n \| + 4 \eta_r^* \| \omega^n \| \| \pd_t \omega^n \| \\
    & \leq \frac{\rho_*}{2} \| \pd_t \omega^n \|^2 + C_n \Big [ \Big ( 1 + \| \bm{U} \|_{L^4}^2 + \| \nabla \mu_U \|^2 \Big ) M_0^2 \Big ].
\end{align*}
Integrating in time over $(0,T)$ leads to 
\begin{align}\label{L2V:bdd:Gal:app:time:d:bo}
\int_0^T \| \pd_t \omega^n \|^2 \, dt \leq M_1^2.
\end{align}
Hence, we see that $\Lambda : S \to \widetilde{S}$ with $\widetilde{S} := \{ \bm{v} \in S \, : \, \| \pd_t \bm{v} \|_{L^2(0,T;\mathbb{V}_{\bu,n})} \leq M_1\}$. By Aubin--Lions lemma, on the finite dimensional space $\mathbb{V}_{\bu,n}$ we have the compact embedding $C^0([0,T];\mathbb{V}_{\bu,n}) \cap H^1(0,T;\mathbb{V}_{\bu,n}) \Subset C^0([0,T];\mathbb{V}_{\bu,n})$, which implies that $\Lambda : S \to S$ is a compact operator.

To apply Schauder's fixed point theorem, it remains to establish the continuity of $\Lambda$. Consider a sequence $\{\bm{U}_k\}_{k \in \N} \subset S$ such that $\bm{U}_k \to \widehat{\bm{U}}$ strongly in $C^0([0,T];\mathbb{V}_{\bu,n})$. Let $(\phi_k, \mu_k)$ and $(\widehat{\phi}, \widehat{\mu})$ denote the strong solution pair to the nonlocal convective Cahn--Hilliard equation associated to velocity $\bm{v} = \bm{U}_k$ and $\widehat{\bm{U}}$, respectively.  We define the differences
\[
\bm{U} := \bm{U}_k - \widehat{\bm{U}}, \quad \bu := \bu^k - \widehat{\bu}, \quad \omega := \omega^k - \widehat{\omega}, \quad \phi := \phi_k - \widehat{\phi}, \quad \mu := \mu_k - \widehat{\mu}
\]
and use the notation $f_k := f(\phi_k)$, $\widehat{f} := f(\widehat{\phi})$ and $f := f_k - \widehat{f}$ for $f \in \{ \rho, \eta, \eta_r, c_{d,a}\}$. The differences satisfy
\begin{subequations} \begin{alignat}{2}
\notag 0 & = ( \rho_k \pd_t \bu, \bm{v}) + (\rho \pd_t \widehat{\bu}, \bm{v}) + (\rho_k (\bm{U}_k \cdot \nabla) \bu^k - \hat{\rho}(\widehat{\bm{U}} \cdot \nabla) \widehat{\bu}, \bm{v}) \\
\notag & \quad  + (2 \eta_k \D \bu, \nabla \bm{v})  + (2 \eta \D \widehat{\bu}, \nabla \bm{v}) + (2 \eta_{r,k} \W \bu, \nabla \bm{v}) + (2 \eta_r \W \widehat{\bu}, \nabla \bm{v}) \\
\notag &\quad - \tfrac{\overline{\rho}_1 - \overline{\rho}_2}{2} (( \nabla \mu_k \cdot \nabla) \bu^k - (\nabla \widehat{\mu} \cdot \nabla) \widehat{\bu}, \bm{v}) - (\mu_k \nabla \phi_k - \widehat{\mu} \nabla \widehat{\phi}, \bm{v}) \\
\label{Gal:app:diff:est:equ:bu} & \quad - (2 \curl_1(\eta_{r,k} \omega + \eta_r \widehat{\omega}), \bm{v}) \\
\notag 0 & =
(\rho_k \pd_t \omega, z) + (\rho \pd_t \widehat{\omega}, z) + (\rho_k \bm{U}_k \cdot \nabla \omega^k - \hat{\rho} \widehat{\bm{U}} \cdot \nabla \widehat{\omega}, z) \\
\notag & \quad + (c_{d,a,k} \nabla \omega, \nabla z) + (c_{d,a} \nabla \widehat{\omega}, \nabla z) - \tfrac{\overline{\rho}_1 - \overline{\rho}_2}{2} (\nabla \mu_k \cdot \nabla \omega + \nabla \mu \cdot \nabla \widehat{\omega}, z), \\
\label{Gal:app:diff:est:equ:bo}  
& \quad - (2 \eta_{r,k} \curl_2 \bu + 2 \eta_r \curl_2 \widehat{\bu}, z) + (4 \eta_{r,k} \omega + 4 \eta_r \widehat{\omega}, z).
\end{alignat}
\end{subequations}
holding for all $\bm{v} \in \mathbb{V}_{\bu,n}$ and $z \in \mathbb{V}_{\omega,n}$. Choosing $\bm{v} = \bu = \bu^k - \widehat{\bu}$ and $z = \omega = \omega^k - \widehat{\omega}$ in \eqref{Gal:app:diff:est:equ:bu} and \eqref{Gal:app:diff:est:equ:bo}, respectively, summing the resulting identities yields
\begin{equation}\label{Gal:app:fixedpoint:cont:identity}
    \begin{aligned}
& \frac{1}{2} \frac{d}{dt} \int_\Omega \rho_k (|\bu|^2 + |\omega|^2) \, dx + \int_\Omega 2 \eta_k |\D \bu|^2 + 2 \eta_{r,k} |\W \bu|^2 + c_{d,a,k} |\nabla \omega|^2 + 4 \eta_{r,k} |\omega|^2 \, dx \\
& \quad = \int_\Omega \tfrac{\overline{\rho}_1 - \overline{\rho}_2}{4} \pd_t \phi_k |\bu|^2  - \tfrac{\overline{\rho}_1 - \overline{\rho}_2}{2} \phi \pd_t \widehat{\bu} \cdot \bu - (\rho_k (\bm{U}_k \cdot \nabla ) \bu^k - \hat{\rho} (\widehat{\bm{U}} \cdot \nabla) \widehat{\bu}) \cdot \bu \, dx \\
& \qquad - \int_\Omega 2 \eta \D \widehat{\bu} : \nabla \bu + 2 \eta_r \W \widehat{\bu} : \nabla \bu -  \tfrac{\overline{\rho}_1 - \overline{\rho}_2}{2} (( \nabla \mu_k \cdot \nabla) \bu^k - (\nabla \widehat{\mu} \cdot \nabla) \widehat{\bu}) \cdot \bu \, dx \\
& \qquad + \int_\Omega (\mu_k \nabla \phi_k - \widehat{\mu} \nabla \widehat{\phi}) \cdot \bu + 2 \curl_1(\eta_{r,k} \omega) \cdot \bu + 2 \curl_1 (\eta_r \widehat{\omega}) \cdot \bu \, dx \\
& \qquad + \int_\Omega \tfrac{\overline{\rho}_1 - \overline{\rho}_2}{4} \pd_t \phi_k |\omega|^2 - \tfrac{\overline{\rho}_1 - \overline{\rho}_2}{2} \phi \pd_t \widehat{\omega} \omega - (\rho_k (\bm{U}_k \cdot \nabla \omega^k) - \hat{\rho} (\widehat{\bm{U}} \cdot \nabla \widehat{\omega})) \omega \, dx \\
& \qquad - \int_\Omega c_{d,a} \nabla \widehat{\omega} \cdot \nabla \omega - 4 \eta_r \widehat{\omega} \omega  - \tfrac{\overline{\rho}_1 - \overline{\rho}_2}{2} ( \nabla \mu_k \cdot \nabla \omega^k - \nabla \widehat{\mu} \cdot \nabla \widehat{\omega}) \omega dx \\
& \qquad + \int_\Omega 2\eta_{r,k} \curl_2 \bu \omega + 2 \eta_r \curl_2 \widehat{\bu} \omega \, dx \\
& \quad =: A_1 + \cdots + A_{17}.
    \end{aligned}
\end{equation}
Comparing with \cite{CCGAGMG}, the new terms arising due to micropolar effects are $A_5$, $A_8, \dots, A_{17}$. Recalling the inverse estimates \eqref{inv:est:Gal:app}, the following interpolation inequalities valid in two spatial dimensions:
\begin{align}\label{GN:ineq}
\| f \|_{L^4} \leq C \| f \|^{1/2} \| f \|_{H^1}^{1/2}, \quad \| f \|_{L^\infty} \leq C \| f \|^{1/2} \| f \|_{H^2}^{1/2},
\end{align}
as well as the estimates in  \eqref{CH:uni:est:Gal:diff:2Dstr} for $(\phi_k, \widehat{\phi}, \mu_k, \widehat{\mu})$, and the estimates \eqref{LinftyL2:bdd:Gal:app:bu:bo} and \eqref{L2V:bdd:Gal:app:time:d:bu} for $\bu^k$ and $\omega^k$, along with the fact that $\bm{U}_k, \widehat{\bm{U}}, \bu^k, \widehat{\bu} \in S \subset C^0([0,T];\mathbb{V}_{\bu,n})$, we have
\begin{align*}
|A_1| & \leq C \| \pd_t \phi_k \| \| \bu \|_{L^4}^2 \leq C_n \| \pd_t \phi_k \| \| \bu \|^2, \\
|A_2| & \leq C \| \phi \|_{L^4} \| \pd_t \widehat{\bu} \| \| \bu \|_{L^4} \leq C_n \| \pd_t \widehat{\bu} \|^2 \| \bu \|^2 + C_n \| \phi \|_{L^4}^2, \\
|A_3| & = | \tfrac{\overline{\rho}_1 - \overline{\rho}_2}{2} ( \phi (\bm{U}_k \cdot \nabla) \bu^k, \bu) + (\hat{\rho} (\bm{U} \cdot \nabla) \bu^k, \bu) + (\hat{\rho} (\widehat{\bm{U}} \cdot \nabla) \bu, \bu)| \\
& \leq C \| \phi \|_{L^4} \| \bm{U}_k \|_{L^\infty} \| \nabla \bu^k \| \| \bu \|_{L^4} + C \| \bm{U} \|_{L^4} \| \nabla \bu^k \| \| \bu \|_{L^4} + C \| \widehat{\bm{U}} \|_{L^4} \| \nabla \bu \| \| \bu \|_{L^4} \\
& \leq C_n M_0^2 \| \phi \|_{L^4} \| \bu \|_{L^2} + C_n M_0 \| \bm{U} \| \| \bu \| + C_n M_0 \| \bu \|^2 \\
& \leq C_n ( 1 + M_0^2) \| \bu \|^2 + C_n M_0^2 ( \| \phi \|_{L^4}^2 + \| \bm{U} \|^2 ), \\
|A_4| & \leq C\| \phi \|_{L^4} \| \D \widehat{\bu} \| \| \nabla \bu \|_{L^4} \leq C_n M_0^2 \| \bu \|^2 + C_n \| \phi \|_{L^4}^2, \\
|A_5| & \leq C_n M_0^2 \| \bu \|^2 + C_n \| \phi \|_{L^4}^2, \\
|A_6|& = |\tfrac{\overline{\rho}_1 - \overline{\rho}_2}{2} ((\nabla \mu_k \cdot \nabla) \bu + (\nabla \mu \cdot \nabla) \widehat{\bu}, \bu) | \\
& = |\tfrac{\overline{\rho}_1 - \overline{\rho}_2}{2} (( \mu_k \nabla \bu, \nabla \bu) + (\mu_k \bu, \Delta \bu)) + \tfrac{\overline{\rho}_1 - \overline{\rho}_2}{2} ( ( \mu \nabla \widehat{\bu}, \nabla \bu) + (\mu \bu, \Delta \widehat{\bu}))| \\
& \leq C \| \mu_k \| \| \nabla \bu \|_{L^4}^2 + C \| \mu_k \| \| \bu \|_{L^\infty} \| \Delta \bu \| + C \| \mu \| \| \nabla \widehat{\bu} \|_{L^4} \| \nabla \bu \|_{L^4} + C \| \mu \| \| \bu \|_{L^\infty} \| \Delta \widehat{\bu} \| \\
& \leq C_n (1 + M_0^2) \| \bu \|^2 + C_n \| \mu \|^2, \\
|A_7|& = |(-\phi \nabla \mu_k + \mu \nabla \widehat{\phi}, \bu)| \leq \| \phi \|_{L^4} \| \nabla \mu_k \| \| \bu \|_{L^4} + \| \mu \| \| \nabla \widehat{\phi} \| \| \bu \|_{L^\infty} \\
& \leq C_n \| \bu \|^2 + C_n ( \| \phi \|_{L^4}^2 + \| \mu \|^2), \\
|A_8| & = 2|(\eta_{r,k} \omega, \curl_2 \bu)| \leq C \| \omega \|^2 + C_n \| \bu \|^2, \\
|A_9| & = 2|(\eta_r \widehat{\omega}, \curl_2 \bu)| \leq C \| \phi \|_{L^4} \| \widehat{\omega} \|_{L^4} \| \curl_2 \bu \| \leq C_n \| \bu \|^2 + C_n \| \phi \|_{L^4}^2.
\end{align*}
The estimates for $A_{10}$, $A_{11}$, $A_{12}$, $A_{13}$ and $A_{15}$ proceed analogously to $A_{1}$, $A_{2}$, $A_{3}$, $A_{4}$ and $A_{6}$, respectively:
\begin{align*}
& |A_{10} + A_{11} + A_{12} + A_{13} + A_{15}| \\
& \quad \leq C_n (1 + \| \pd_t \phi_k \| + \| \pd_t \widehat{\omega} \|^2 ) \| \omega \|^2 + C_n \| \phi \|_{L^4}^2 + C_n \| \bm{U} \|^2 + C_n \| \mu \|^2.
\end{align*}
For the remaining terms, we have
\begin{align*}
|A_{14}| & \leq C \| \phi \|_{L^4} \| \widehat{\omega} \| \| \omega \|_{L^4} \leq C_n (\| \phi \|_{L^4}^2 + \| \omega \|^2), \\
|A_{16}| & \leq C \| \curl_2 \bu \| \| \omega \| \leq C_n ( \| \bu \|^2 + \| \omega \|^2), \\
|A_{17}| & \leq C \| \phi \|_{L^4} \| \curl_2 \widehat{\bu} \| \| \omega \|_{L^4} \leq C_n ( \| \phi \|_{L^4}^2 + \| \omega \|^2 ).
\end{align*}
Then, from \eqref{Gal:app:fixedpoint:cont:identity} we deduce the following differential inequality
\begin{align*}
\frac{d}{dt} \int_\Omega \rho_k (|\bu|^2 + |\omega|^2) \, dx \leq G_1(t) \int_\Omega \rho_k (|\bu|^2 + |\omega|^2) \, dx + G_2(t),
\end{align*}
where
\[
G_1(t) := C_n \Big ( 1 + \| \pd_t \phi_k \|^2 + \| \pd_t \widehat{\bu} \|^2 + \| \pd_t \widehat{\omega} \|^2 \Big ), \quad G_2(t) := C_n \Big ( 1 + \| \phi \|_{L^4}^2 + \| \mu \|^2 + \| \bm{U} \|^2 \Big )
\]
satisfy $G_1, G_2 \in L^1(0,T)$ thanks to \eqref{CH:uni:est:Gal:diff:2Dstr}, \eqref{L2V:bdd:Gal:app:time:d:bu} and \eqref{L2V:bdd:Gal:app:time:d:bo}. Gronwall's inequality and the lower bound on $\rho$ provide the estimate
\begin{align*}
\max_{t \in [0,T]} \Big ( \| \bu(t) \|^2 + \| \omega(t) \|^2 \Big ) \leq \frac{1}{\rho_*} \exp \Big (\int_0^T G_1(r) \, dr \Big ) \int_0^T G_2(r) \, dr.
\end{align*}
Then, as $\bm{U} = \bm{U}_k - \widehat{\bm{U}} \to \bm{0}$ in $C^0([0,T];\mathbb{V}_{\bu,n})$, it holds that the right-hand side of the above inequality tends to zero as $k \to \infty$. Hence, we deduce that $\bu = \bu^k - \widehat{\bu} \to \bm{0}$ in $C^0([0,T];\mathbb{V}_{\bu,n})$ and $\omega = \omega^k - \widehat{\omega} \to 0$ in $C^0([0,T];\mathbb{V}_{\omega,n})$, leading to the continuity of the mapping $\Lambda$.  Hence, for each $n \in \N$, we infer via Schauder's fixed point theorem the existence of an approximate solution quadruple $(\bu_n, \omega_n, \phi_n, \mu_n)$ to \eqref{equ:Gal:app:2Dstr}.
\end{proof}

\subsection{Uniform estimates}
In the sequel we use the symbol $C$ to denote positive constants independent of $n \in \N$, whose values may change line by line and even within the same line.

Integrating \eqref{equ:CH:Gal:app:2Dstr} over $\Omega$ yields the conservation of mass
\begin{align} \label{conser:mass:phi:uni:n}
\overline{\phi_{n}}(t) = \overline{\phi_{0}}, \quad \forall t \in [0,T].
\end{align}
Choosing $\bm{v} = \bu_n$ in \eqref{equ:bu:Gal:app:2Dstr}, $z = \omega_n$ in \eqref{equ:bo:Gal:app:2Dstr}, while testing \eqref{equ:CH:Gal:app:2Dstr} with $\mu_n$ and testing \eqref{equ:mu:Gal:app:2Dstr} with $\pd_t\phi_n$, then upon summing we obtain the energy identity
\begin{align*}
& \frac{d}{dt} \int_\Omega \frac{\rho(\phi_n)}{2} (|\bu_n|^2 + |\omega_n|^2) + F(\phi_n) + \frac{1}{2} a|\phi_n|^2 - \phi_n (K \star \phi_n) \, dx \\
& \quad + \int_\Omega 2 \eta(\phi_n) |\D \bu_n|^2 + 4 \eta_{r}(\phi_n) |\tfrac{1}{2} \curl_2 \bu_n - \omega_n|^2 + c_{d,a}(\phi_n)|\nabla \omega_n|^2 + |\nabla \mu_n|^2 \, dx = 0.
\end{align*}
From this we deduce that
\begin{align}\label{uni:n:bdd:bu:bo:LinfL2capL2H1:mu:L2H1homo}
\| \bu_n \|_{L^\infty(0,T;\HH_\sigma) \cap L^2(0,T;\HH^1_{\sigma})} + \| \omega_n \|_{L^\infty(0,T;L^2) \cap L^2(0,T;H^1)} + \| \nabla \mu_n \|_{L^2(Q)} \leq C.
\end{align}
Thanks to the embedding $L^\infty(0,T;\HH_\sigma) \cap L^2(0,T;\HH^1_{\sigma}) \subset L^4(0,T;L^4_{\sigma})$ in two spatial dimensions, we can invoke Lemma \ref{sec2prere:str:sol:reg:CH} to obtain analogous uniform estimates to \eqref{str:sol:CH:grad:mu:Linfty:L2:est}, \eqref{str:sol:tdphi:L2L2:CH:grad:mu:L2H1:est} and \eqref{str:sol:phi:mu:F:bdds} for $(\phi_n, \mu_n)$.

Next, we choose $\bm{v} = \pd_t \bu_n$ in \eqref{equ:bu:Gal:app:2Dstr} and $z = \pd_t \omega_n$ in \eqref{equ:bo:Gal:app:2Dstr}, so that upon summing leads to (cf.~\cite[(3.31)]{ChanLamstrsol})
\begin{equation}\label{equ:bu:plus:bo:H1:est:2Dstr}
\begin{aligned}
& \frac{d}{dt} \int_\Omega \eta(\phi_n) |\D \bu_n|^2 + 2\eta_r(\phi_n) |\tfrac{1}{2} \curl_2 \bu_n - \omega_n|^2 + c_{d,a}(\phi_n) |\nabla \omega_n|^2 \, dx  \\
& \qquad + \int_\Omega \rho(\phi_n) (|\pd_t \bu_n|^2  + |\pd_t \omega_n|^2 ) \, dx \\
& \quad = - (\rho(\phi_n) (\bu_n \cdot \nabla) \bu_n, \pd_t \bu_n) + (\eta'(\phi_n) \pd_t \phi_n, |\D \bu_n|^2) + (\eta_r'(\phi_n) \pd_t \phi_n, |\W \bu_n|^2) \\
& \qquad + \tfrac{\overline{\rho}_1 - \overline{\rho}_2}{2} ((\nabla \mu_n \cdot \nabla) \bu_n, \pd_t \bu_n) + (\mu_n \nabla \phi_n, \pd_t \bu_n) - (\rho(\phi_n) \bu_n \cdot \nabla \omega_n, \pd_t \omega_n)   \\
& \qquad + (c_{d,a}'(\phi_n) \pd_t \phi_n, |\nabla \omega_n|^2) + \tfrac{\overline{\rho}_1 - \overline{\rho}_2}{2} ( \nabla \mu_n \cdot \nabla \omega_n, \pd_t \omega_n) \\
& \qquad - (2 \eta_r'(\phi_n) \pd_t \phi_n \curl_2 \bu_n, \omega_n) + (2 \eta_r'(\phi_n) \pd_t \phi_n, |\omega_n|^2) \\
& \quad =: B_1 + \cdots + B_{10}.
\end{aligned}
\end{equation}
The derivation of the above identity relies on the following relation, see also \eqref{2D:curlu:w}:
\begin{align*}
& (2\eta_r(\phi_n) \W \bu_n, \W \pd_t \bu_n) - (2 \curl_1(\eta_r(\phi_n) \omega_n), \pd_t \bu_n) \\
& \qquad - (2 \eta_r(\phi_n) \curl_2 \bu_n, \pd_t \omega_n) + (4 \eta_r(\phi_n) \omega_n, \pd_t \omega_n) \\
& \quad = \frac{1}{2} \frac{d}{dt} \int_\Omega \eta_r(\phi_n) \Big [ |\curl_2 \bu_n|^2 - 4 \omega_n \curl_2 \bu_n  + 4 |\omega_n|^2] \, dx \\
& \qquad - (2 \eta_r'(\phi_n) \pd_t \phi_n, \tfrac{1}{2} |\W \bu_n|^2 - \omega_n \curl_2 \bu_n + |\omega_n|^2).
\end{align*}
Similarly, choosing $\bm{v} = \bm{A}_S \bu_n$ in \eqref{equ:bu:Gal:app:2Dstr}, where we recall that there exists a pressure $p_n \in C^0([0,T];H^1(\Omega))$ from the regularity theory involving the Stokes operator $\bm{A}_S$ such that $\bm{A}_S \bu_n = - \Delta \bu_n + \nabla p_n$, leads to (cf.~\cite[(3.36)]{ChanLamstrsol})
\begin{equation}\label{equ:bu:H2:est:2Dstr}
\begin{aligned}
& ((\eta(\phi_n) + \eta_r(\phi_n))\bm{A}_S \bu_n, \bm{A}_S \bu_n) \\
& \quad = - (\rho(\phi_n) \pd_t \bu_n, \bm{A}_S \bu_n) - (\rho(\phi_n) (\bu_n \cdot \nabla) \bu_n, \bm{A}_S \bu_n) \\
& \qquad + \tfrac{\overline{\rho}_1 + \overline{\rho}_2}{2} ((\nabla \mu_n \cdot \nabla) \bu_n, \bm{A}_S \bu_n) + (\mu_n \nabla \phi_n, \bm{A}_S \bu_n) \\
& \qquad + (2 \eta'(\phi_n) \nabla \phi_n \D \bu_n, \bm{A}_S \bu_n) + (2 \eta_r'(\phi_n) \nabla \phi_n \W \bu_n, \bm{A}_S \bu_n) \\
& \qquad + (2 \curl_1(\eta_r(\phi_n) \omega_n), \bm{A}_S \bu_n) - (p_n \eta'(\phi_n) \nabla \phi_n, \bm{A}_S \bu_n) \\
& \qquad - (p_n \eta_r'(\phi_n) \nabla \phi_n, \bm{A}_S \bu_n) \\
& \quad =: B_{11} + \cdots + B_{19},
\end{aligned}
\end{equation}
while choosing $z = - \Delta \omega_n$ in \eqref{equ:bo:Gal:app:2Dstr} leads to (cf.~\cite[(3.37)]{ChanLamstrsol})
\begin{equation}\label{equ:bo:H2:est:2Dstr}
\begin{aligned}
& (c_{d,a}(\phi_n) \Delta \omega_n, \Delta \omega_n) \\
& \quad =  (\rho(\phi_n) \pd_t \omega_n, \Delta \omega_n) + (\rho(\phi_n) \bu_n \cdot \nabla \omega_n, \Delta \omega_n) \\
& \qquad - \tfrac{\overline{\rho}_1 - \overline{\rho}_2}{2} (\nabla \mu_n \cdot \nabla \omega_n, \Delta \omega_n) - (c_{d,a}'(\phi_n) \nabla \phi_n \cdot \nabla \omega_n, \Delta \omega_n) \\
& \qquad - (2 \eta_r(\phi_n) \curl_2 \bu_n, \Delta \omega_n) + (4 \eta_r(\phi_n) \omega_n, \Delta \omega_n) \\
& \quad =: B_{20} + \cdots + B_{25}.
\end{aligned}
\end{equation}
We define the quantity
\begin{align*}
H_n(t) := \int_\Omega \eta(\phi_n) |\D \bu_n|^2 + 2 \eta_r(\phi_n) |\tfrac{1}{2} \curl_2 \bu_n - \omega_n|^2 + c_{d,a}(\phi_n) |\nabla \omega_n|^2 \,dx
\end{align*}
and let $\beta_1$ and $\beta_2$ be two positive constants yet to be determined. Multiplying \eqref{equ:bu:H2:est:2Dstr} by $\beta_1$, \eqref{equ:bo:H2:est:2Dstr} by $\beta_2$ and adding the resulting equalities to \eqref{equ:bu:plus:bo:H1:est:2Dstr}, and using the lower bounds of $\eta$, $\eta_r$, $c_{d,a}$ and $\rho$, we arrive at
\begin{equation}\label{equ:H2est:2Dstr:combined}
    \begin{aligned}
& \frac{d}{dt} H_n(t) + \rho_* (\| \pd_t \bu_n \|^2 + \| \pd_t \omega_n \|^2) + \beta_1 (\eta_* + \eta_{r,*}) \| \bm{A}_S \bu_n \|^2 + \beta_2 c_{d,a,*} \| \Delta \omega_n \|^2 \\
& \quad \leq B_1 + \cdots + B_{10} + \beta_1 ( B_{11} + \cdots + B_{19}) + \beta_2 (B_{20} + \cdots + B_{25}).
    \end{aligned}
\end{equation}
Exploiting the uniform estimates \eqref{str:sol:phi:mu:F:bdds}, \eqref{uni:n:bdd:bu:bo:LinfL2capL2H1:mu:L2H1homo}, the bound $|\phi_n| \leq 1$ a.e.~in $Q$, the interpolation inequalities \eqref{GN:ineq}, Korn's inequality, Poincar\'e's inequality, as well as the $L^4$-estimate \eqref{pressure:L4} for the pressure $p_n$ from Lemma \eqref{L4:pre:estim}, we now estimate the right-hand side of \eqref{equ:H2est:2Dstr:combined} term-by-term:
\begin{align*}
|B_1| & \leq C \| \bu_n \|_{L^4} \| \nabla \bu_n \|_{L^4} \| \pd_t \bu_n \| \leq C \| \D \bu_n \|\| \bm{A}_S \bu_n \|^{1/2}\| \pd_t \bu_n \| \\
& \leq \frac{\rho_*}{8} \| \pd_t \bu_n \|^2 + \frac{\beta_1 \eta_*}{22} \| \bm{A}_S \bu_n \|^2 + C \| \D \bu_n \|^4,  \\
|B_2| & \leq C \| \pd_t \phi_n \| \| \D \bu_n \|_{L^4}^2 \leq C \| \pd_t \phi_n \| \| \D\bu_n \| \| \bm{A}_S \bu_n \| \\
& \leq \frac{\beta_1 \eta_*}{22} \| \bm{A}_S \bu_n \|^2 + C \| \pd_t \phi_n \|^2 \| \D \bu \|^2, \\
|B_3| & \leq C \| \pd_t \phi_n \| \| \nabla \bu_n \|_{L^4}^2 \leq \frac{\beta_1 \eta_*}{22} \| \bm{A}_S \bu_n \|^2 + C \| \pd_t \phi_n \|^2 \| \D \bu_n \|^2, \\
|B_4| & \leq C \| \nabla \mu_n \|_{L^4} \| \nabla \bu_n \|_{L^4} \| \pd_t \bu_n \| \leq C \| \nabla \mu_n \|_{H^1}^{1/2} \| \D \bu_n \|^{1/2} \| \bm{A}_S \bu_n \|^{1/2} \| \pd_t \bu_n \| \\
& \leq \frac{\rho_*}{8} \| \pd_t \bu_n \|^2 + \frac{\beta_1 \eta_*}{22} \| \bm{A}_S \bu_n \|^2 + C \|\nabla \mu_n \|_{H^1}^2 \| \D \bu_n \|^2, \\
|B_5| & = |(\phi_n \nabla \mu_n, \pd_t \bu_n)| \leq \frac{\rho_*}{8} \| \pd_t \bu_n \|^2 + C \| \nabla \mu_n \|^2.
\end{align*}
Estimating $B_6$, $B_7$ and $B_8$ can proceed analogously as for $B_1$, $B_2$ and $B_4$:
\begin{align*}
|B_6| & \leq \frac{\rho_*}{8} \| \pd_t \omega_n \|^2 + \frac{\beta_2 c_{d,a,*}}{14} \| \Delta \omega_n \|^2 + C \|\nabla \omega_n \|^4, \\
|B_7| & \leq \frac{\beta_2 c_{d,a,*}}{14} \| \Delta \omega_n \|^2 + C \| \pd_t \phi_n \|^2 \| \nabla \omega_n \|^2, \\
|B_8| & \leq \frac{\rho_*}{8} \| \pd_t \omega_n \|^2 + \frac{\beta_2 c_{d,a,*}}{14} \|\Delta \omega_n \|^2 + C \| \nabla \mu_n \|_{H^1}^2 \| \nabla \omega_n \|^2.
\end{align*}
Meanwhile, we see that
\begin{align*}
|B_9| & \leq C \| \pd_t \phi_n \| \| \nabla \bu_n \|_{L^4} \| \omega_n \|_{L^4} \leq C \| \pd_t \phi_n \| \| \D \bu_n \|^{1/2} \| \bm{A}_S \bu_n \|^{1/2} \|\nabla \omega_n \| \\
& \leq \frac{\beta_1 \eta_*}{22} \| \bm{A}_S \bu_n \|^2 + C \Big ( 1 + \| \pd_t \phi_n \|^2 \Big ) \Big ( \| \D \bu_n \|^2 + \| \nabla \omega_n \|^2 \Big ), \\
|B_{10}| & \leq C \| \pd_t \phi_n \| \| \omega_n \|_{L^4}^2 \leq C \Big ( 1+ \| \pd_t \phi_n \|^2 \Big ) \| \nabla \omega_n \|^2, \\
|\beta_1 B_{11}| & \leq \frac{\rho_*}{8} \|\pd_t \bu_n \|^2 + \frac{2 \beta_1^2 (\rho^*)^2}{\rho_*} \| \bm{A}_S \bu_n \|^2.
\end{align*}
Similar to $B_1$, $B_4$ and $B_5$, we have
\begin{align*}
|\beta_1 B_{12}| & \leq C\| \bm{A}_S \bu_n \| \| \bu_n \|_{L^4} \| \nabla \bu_n \|_{L^4} \leq C \| \bm{A}_S \bu_n \|^{3/2} \| \D \bu_n \| \\
& \leq \frac{\beta_1 \eta_*}{22} \| \bm{A}_S \bu_n \|^2 + C \| \D \bu_n \|^4, \\
|\beta_1 B_{13}| & \leq C \| \nabla \mu_n \|_{L^4} \|\nabla \bu_n \|_{L^4} \| \bm{A}_S \bu_n \| \leq C \| \nabla \mu_n \|_{H^1}^{1/2} \| \D \bu_n \|^{1/2} \| \bm{A}_S \bu_n \|^{3/2} \\
& \leq \frac{\beta_1 \eta_*}{22} \| \bm{A}_S \bu_n \|^2 + C \| \nabla \mu_n \|_{H^1}^2 \| \D \bu_n \|^2, \\
|\beta_1 B_{14}| & = |\beta_1(\phi_n \nabla \mu_n, \bm{A}_S \bu_n )| \leq \frac{\beta_1 \eta_*}{22} \| \bm{A}_S \bu_n \|^2 + C \| \nabla \mu_n \|^2,
\end{align*}
while
\begin{align*}
|\beta_1 (B_{15} + B_{16})| & \leq C \| \nabla \phi_n \|_{L^4} \| \D \bu_n \|_{L^4} \| \bm{A}_S \bu_n \| \leq C \| \nabla \phi_n \|_{H^1}^{1/2} \| \D \bu_n \|^{1/2} \| \bm{A}_S \bu_n \|^{3/2} \\
& \leq \frac{\beta_1 \eta_*}{22} \| \bm{A}_S \bu_n \|^2 + C \| \nabla \phi_n \|_{H^1}^2 \| \D \bu_n \|^2.
\end{align*}
For $B_{17}$ we use the product rule
\begin{align*}
|\beta_1 B_{17}| & = |\beta_1 (2 \eta_r(\phi_n) \curl_1 \omega_n + \eta_r'(\phi_n) \omega_n \curl_1 \phi_n , \bm{A}_S \bu_n)| \\
& \leq C \| \bm{A}_S \bu_n \| (\| \nabla \omega_n \| + \| \omega_n \|_{L^4} \| \nabla \phi_n \|_{L^4}) \\
& \leq \frac{\beta_1 \eta_*}{22} \| \bm{A}_S \bu_n \|^2 + C (1 +  \| \nabla \phi_n \|_{H^1}^2) \| \nabla \omega_n \|^2,
\end{align*}
while for $B_{18}$ and $B_{19}$ we use the pressure estimate \eqref{pressure:L4}:
\begin{align*}
|\beta(B_{18} + B_{19})| & \leq C \| p_n \|_{L^4} \| \bm{A}_S \bu_n \| \| \nabla \phi_n \|_{L^4} \leq C \| \nabla \bu_n \|^{1/2} \| \bm{A}_S \bu_n \|^{3/2} \| \nabla \phi_n \|_{H^1}^{1/2}\\
& \leq \frac{\beta_1 \eta_*}{22} \| \bm{A}_S \bu_n \|^2 + C \| \nabla \phi_n \|_{H^1}^2 \| \D \bu_n \|^2.
\end{align*}
We observe that estimating $B_{20}$, $B_{21}$, $B_{22}$ and $B_{23}$ proceed analogously as for $B_{11}$, $B_{12}$, $B_{13}$ and $B_{15}$:
\begin{align*}
|\beta_2 B_{20}| & \leq \frac{\rho_*}{4} \| \pd_t \omega_n \|^2 + \frac{ \beta_2^2 (\rho^*)^2}{\rho_*} \| \Delta \omega_n \|^2, \\
|\beta_2 B_{21}| & \leq  C \| \Delta \omega_n \| \| \bu_n \|_{L^4} \| \nabla \omega_n \|_{L^4} \leq C \| \Delta \omega_n \|^{3/2} \| \D \bu_n \| \\
& \leq \frac{\beta_2 c_{d,a,*}}{14} \| \Delta \omega_n \|^2 + C \| \D \bu_n \|^4, \\
|\beta_2 B_{22}| & \leq \frac{\beta_2 c_{d,a,*}}{14} \| \Delta \omega_n \|^2 + C \| \nabla \mu_n \|_{H^1}^2 \| \nabla \omega_n \|^2, \\
|\beta_2 B_{23}| & \leq \frac{\beta_2 c_{d,a,*}}{14} \| \Delta \omega_n \|^2 + C \| \nabla \phi_n \|_{H^1}^2 \| \nabla \omega_n \|^2.
\end{align*}
Similar to $B_9$ and $B_{10}$, we deduce that
\begin{align*}
|\beta_2 (B_{24} + B_{25})| & \leq \frac{\beta_2 c_{d,a,*}}{14} \| \Delta \omega_n \|^2 + C (\| \D \bu_n \|^2 + \| \nabla \omega_n \|^2).
\end{align*}
Hence, we infer from \eqref{equ:H2est:2Dstr:combined} the differential inequality
\begin{align*}
& \frac{d}{dt} H_n(t) + \frac{\rho_*}{2} (\| \pd_t \bu_n \|^2 + \| \pd_t \omega_n \|^2) \\
& \qquad + \beta_1 \Big ( \frac{1}{2} - \frac{2 \beta_1 (\rho^*)^2}{\rho_*} \Big ) \| \bm{A}_S \bu_n \|^2 + \beta_2 \Big (\frac{1}{2} - \frac{\beta_2 (\rho^*)^2}{\rho_*} \Big ) \| \Delta \omega_n \|^2 \\
& \quad \leq C \Big ( 1 + \| \pd_t \phi_n \|^2 + \| \D \bu_n \|^2 + \| \nabla \omega_n \|^2 + \| \nabla \mu_n \|_{H^1}^2 \Big ) H_n(t) + C \| \nabla \mu_n \|^2 \\
& \quad =: G_n(t) H_n(t) + J_n(t).
\end{align*}
By the analogoue of \eqref{str:sol:tdphi:L2L2:CH:grad:mu:L2H1:est} for $(\phi_n, \mu_n)$, it holds that $G_n, J_n \in L^1(0,T)$ uniformly in $n$.  Hence, by Gronwall's inequality,
\begin{align}\label{strsol:diffineq}
H_n(t) \leq \Big ( H_n(0) + \int_0^T J_n(s) \, ds \Big ) \exp \Big ( \int_0^T G_n(r) \, dr \Big ), \quad \forall t \in [0,T].
\end{align}
By properties of the projection operators, it holds that
\[
|H_n(0)| \leq C \Big (\| \bu_0 \|_{\HH^1_{\sigma}}^2 + \| \omega_0 \|_{H^1}^2 \Big ),
\]
and so we obtain the uniform estimates:
\begin{equation} \label{high:est:bu:bo:uni:n:2Dstr}
\begin{aligned}
& \| \bu_{n} \|_{L^{\infty}(0,T;\HH^1_{\sigma}) \cap L^2(0,T;\HH^2_{\sigma}) \cap H^1(0,T;\HH_\sigma)} \\
& \quad + \| \omega_n \|_{L^\infty(0,T;H^1(\Omega)) \cap L^2(0,T;H^2(\Omega)) \cap H^1(0,T;L^2(\Omega))} \leq C.
\end{aligned}
\end{equation}

\subsection{Passing to the limit}
The above uniform estimates \eqref{str:sol:CH:grad:mu:Linfty:L2:est}, \eqref{str:sol:tdphi:L2L2:CH:grad:mu:L2H1:est}, \eqref{str:sol:phi:mu:F:bdds}, \eqref{conser:mass:phi:uni:n}, \eqref{uni:n:bdd:bu:bo:LinfL2capL2H1:mu:L2H1homo} and \eqref{high:est:bu:bo:uni:n:2Dstr} enable us to obtain a quadruple of limit functions $(\bu, \omega, \phi, \mu)$ such that along a nonrelabelled subsequence $n \to \infty$
\begin{align*}
\bu_n & \to \bu \text{ weakly* in } L^\infty(0,T;\HH^1_{\sigma}) \cap L^2(0,T;\HH^2{\div}) \cap H^1(0,T;\HH_\sigma), \\
\bu_n & \to \bu \text{ strongly in } L^2(0,T;\HH^1_{\sigma}), \\
\omega_n & \to \omega \text{ weakly* in } L^\infty(0,T;H^1(\Omega)) \cap L^2(0,T;H^2(\Omega)) \cap H^1(0,T;L^2(\Omega)), \\
\omega_n & \to \omega \text{ strongly in } L^2(0,T;H^1(\Omega)), \\
\phi_n & \to \phi \text{ weakly* in } L^\infty(0,T;H^1(\Omega)) \cap H^1(0,T;L^2(\Omega)) \cap W^{1,\infty}(0,T;H^1(\Omega)^*), \\
\phi_n & \to \phi \text{ weakly in } L^{2p/(p-2)}(0,T;W^{1,p}(\Omega)) , \quad \forall p \in (2,\infty), \\
 \phi_n & \to \phi \text{ strongly in } C^0([0,T];L^p(\Omega)) \quad \forall p \in [2,\infty), \\
\mu_n & \to \mu \text{ weakly* in } L^\infty(0,T;H^1(\Omega)) \cap L^2(0,T;H^2(\Omega)) \cap H^1(0,T;H^1(\Omega)^*), \\
\mu_n & \to \mu \text{ strongly in } L^2(0,T;H^1(\Omega)),\\
F'(\phi_n) & \to F'(\phi) \text{ weakly* in } L^\infty(0,T;H^1(\Omega)), 
\end{align*}
along with $|\phi| < 1$ a.e.~in $Q$. These convergences are sufficient to pass to the limit $n \to \infty$ in \eqref{equ:Gal:app:2Dstr}. To recover the pressure, we argue as in \cite{ChanLamstrsol,CHNS2Dstrsol} to deduce the existence of $p \in L^2(0,T;H^1_{(0)}(\Omega))$ such that 
\begin{align*}
\nabla p & = - \rho(\phi) (\pd_t \bu + (\bu \cdot \nabla) \bu) + \rho'(\phi) (\nabla \mu \cdot \nabla) \bu \\
& \quad + \div ( 2\eta(\phi) \D \bu + 2 \eta_r(\phi) \W \bu) + \mu \nabla \phi + 2 \curl_1(\eta_r(\phi) \omega) \text{ a.e.~in Q}.
\end{align*}
Hence, the limit quntuple $(\bu, \omega, p, \phi, \mu)$ constitutes a strong solution to nMAGG model as depicted in Theorem \ref{thm:strwellposed}.

\begin{remark}[Strong solutions under periodic boundary conditions]
The proof of Theorem \ref{thm:strwellposed} under periodic boundary conditions simplifies slightly in the following sense: we can choose $\{\bm{X}_j\}_{j \in \N}$ as the eigenfunctions of the Stokes operator augmented by the constant function, and instead of choosing $\bm{v} = \bm{A}_S \bu_n$ in \eqref{equ:bu:Gal:app:2Dstr} to obtain the identity \eqref{equ:bu:H2:est:2Dstr}, it is sufficient to take $\bm{v} = - \Delta \bu_n$, which in turn leads to the absence of the terms $B_{18}$ and $B_{19}$ involving the pressure $p_n$.
\end{remark}



\subsection{Uniqueness and continuous dependence}
We consider two sets of initial data $\{(\bu_0^i, \omega_0^i, \phi_0^i)\}_{i=1,2}$ with $\| \phi_0^i \|_{L^\infty(\Omega)} < 1$ for $i \in \{1,2\}$, and the strong solutions $\{(\bu_i, \omega_i, p_i, \phi_i, \mu_i)\}_{i=1,2}$ emanating from them. Denoting $(\bu, \omega, p, \phi, \mu)$ as their differences, we further set $g_i = g(\phi_i)$ for $g \in \{ \rho, \eta, \eta_r, c_{d,a}\}$ and see that the differences satisfy
\begin{subequations}
\begin{alignat}{2}
\label{2Dstrsol:uni:diff:equ:bu} & \rho_1 \pd_t \bu + (\rho_1 - \rho_2) \pd_t \bu_2 + \rho_1 (\bu_1 \cdot \nabla) \bu + \rho_1 (\bu \cdot \nabla) \bu_2 + (\rho_1 - \rho_2) (\bu_2 \cdot \nabla) \bu_2 \\
\notag & \qquad - \tfrac{\overline{\rho}_1 - \overline{\rho}_2}{2} ((\nabla \mu_1 \cdot \nabla) \bu + (\nabla \mu \cdot \nabla)\bu_2) - 2\div ( \eta_1 \D \bu + \eta_{r,1} \W \bu) \\
\notag & \qquad - 2 \div ( (\eta_1 - \eta_2) \D \bu_2 + (\eta_{r,1} - \eta_{r,2}) \W \bu_2) + \nabla p \\
\notag & \quad = \mu_1 \nabla \phi + \mu \nabla \phi_2 + 2 \curl_1( \eta_{r,1} \omega + (\eta_{r,1} - \eta_{r,2}) \omega_2), \\
\label{2Dstrsol:uni:diff:equ:bo} & \rho_1 \pd_t \omega + (\rho_1 - \rho_2) \pd_t \omega_2 + \rho_1 \bu_1 \cdot \nabla \omega + \rho_1 \bu \cdot \nabla \omega_2 + (\rho_1 - \rho_2) \bu_2 \cdot \nabla \omega_2 \\
\notag & \qquad - \tfrac{\overline{\rho}_1 - \overline{\rho}_2}{2} (\nabla \mu_1 \cdot \nabla \omega + \nabla \mu \cdot \nabla \omega_2) - \div (c_{d,a,1} \nabla \omega + (c_{d,a,1} - c_{d,a,2}) \nabla \omega_2)  \\
\notag & \quad = 2 \eta_{r,1} (\curl_2 \bu - 2 \omega) + 2 (\eta_{r,1} - \eta_{r,2})(\curl_2 \bu_2 - 2 \omega_2), \\
\label{2Dstrsol:uni:diff:equ:CH} & \pd_t \phi + \bu_1 \cdot \nabla \phi + \bu \cdot \nabla \phi_2 = \Delta \mu, \\
\label{2Dstrsol:uni:diff:equ:mu} & \mu = F'(\phi_1) - F'(\phi_2) + a \phi - K \star \phi,
\end{alignat}
\end{subequations}
a.e.~in $Q$.  Observing the following identities
\begin{subequations}\label{uniq:id}
\begin{alignat}{2}
\label{uniq:id:1} 0 & = \int_\Omega -\frac{|\bu|^2}{2} \pd_t \rho_1 + \rho_1 (\bu_1 \cdot \nabla) \bu \cdot \bu - \frac{\overline{\rho}_1 - \overline{\rho}_2}{2} (\nabla \mu_1 \cdot \nabla ) \bu \cdot \bu \, dx, \\
0 & = \int_\Omega -\frac{|\omega|^2}{2} \pd_t \rho_1 + \rho_1 (\bu_1 \cdot \nabla \omega)\omega - \frac{\overline{\rho}_1 - \overline{\rho}_2}{2} (\nabla \mu_1 \cdot \nabla \omega) \omega \, dx, \\
0 & = \int_\Omega (\nabla \mu \cdot \nabla) \bu_2 \cdot \bu + \mu \Delta \bu_2 \cdot \bu + \mu \nabla \bu_2 : \nabla \bu \, dx, \\
0 & = \int_\Omega (\nabla \mu \cdot \nabla \omega_2) \omega + \mu \Delta \omega_2 \omega + \mu \nabla \omega_2 \cdot \nabla \omega \, dx,
\end{alignat}
\end{subequations}
along with 
\begin{equation}\label{uniq:id3}
\begin{aligned}
& (\mu_1 \nabla \phi + \mu \nabla \phi_2 , \bu) = - (\nabla \mu_1, \phi \bu) - (\nabla \mu, \phi_2 \bu) \\
& \quad = - ((a + F''(\phi_1) \nabla \phi_1, \phi \bu) - (\phi_1 \nabla a, \phi \bu) + ( \nabla K \star \phi_1, \phi \bu) \\
& \qquad - ((a + F''(\phi_1)) \nabla \phi, \phi_2 \bu) - ((F''(\phi_1) - F''(\phi_2)) \nabla \phi_2, \phi_2 \bu) \\
& \qquad- (\phi \nabla a, \phi_2 \bu) + (\nabla K \star \phi, \phi_2 \bu), 
\end{aligned}
\end{equation}
and 
\begin{equation}\label{uniq:id4}
\begin{aligned}
(\nabla \mu, \nabla \phi) & = (a + F''(\phi_1), |\nabla \phi|^2) + (F''(\phi_1) - F''(\phi_2), \nabla \phi_2 \cdot \nabla \phi) \\
& \quad - (\nabla K \star \phi, \nabla \phi) - (\phi \nabla a, \nabla \phi),
\end{aligned}
\end{equation}
we then obtain after testing \eqref{2Dstrsol:uni:diff:equ:bu} by $\bu$, \eqref{2Dstrsol:uni:diff:equ:bo} by $\omega$, \eqref{2Dstrsol:uni:diff:equ:CH} by $\phi$ and summing:
\begin{equation}\label{2Dstrsol:uni:ineq}
\begin{aligned}
& \frac{1}{2} \frac{d}{dt} \int_\Omega \rho_1 (|\bu|^2 + |\omega|^2) + |\phi|^2 \, dx \\
& \qquad + \int_\Omega 2 \eta_1 |\D \bu|^2 + 4 \eta_{r,1} |\tfrac{1}{2} \curl_2 \bu - \omega|^2 + c_{d,a,1}|\nabla \omega|^2 + (a + F''(\phi_1)) |\nabla \phi|^2 \, dx  \\
& \quad = - ((\rho_1 - \rho_2)\pd_t \bu_2, \bu) - (\rho_1 (\bu \cdot \nabla) \bu_2, \bu) - ((\rho_1 - \rho_2) (\bu_2 \cdot \nabla) \bu_2, \bu) \\
& \qquad - \tfrac{\overline{\rho}_1 - \overline{\rho}_2}{2} ( \mu \Delta \bu_2, \bu) - \tfrac{\overline{\rho}_1 - \overline{\rho}_2}{2} (\mu \nabla \bu_2, \nabla \bu) - 2 ((\eta_1 - \eta_2) \D \bu_2, \nabla \bu) \\
& \qquad - 2((\eta_{r,1} - \eta_{r,2}) \W \bu_2, \nabla \bu)  - ((a + F''(\phi_1)) \nabla \phi_1, \phi \bu) \\
& \qquad - (\phi_1 \nabla a, \phi \bu) + ( \nabla K \star \phi_1, \phi \bu) - ((a + F''(\phi_1)) \nabla \phi, \phi_2 \bu) \\
& \qquad - ((F''(\phi_1) - F''(\phi_2)) \nabla \phi_2, \phi_2 \bu) - (\phi \nabla a, \phi_2 \bu)  \\
& \qquad + (\nabla K \star \phi, \phi_2 \bu) + 2 ((\eta_{r,1} - \eta_{r,2}) \omega_2, \curl_2 \bu) \\
& \qquad - ((\rho_1 - \rho_2) \pd_t \omega_2, \omega) - (\rho_1 \bu \cdot \nabla \omega_2, \omega) - ((\rho_1 - \rho_2) \bu_2 \cdot \nabla \omega_2, \omega) \\
& \qquad - \tfrac{\overline{\rho}_1 - \overline{\rho}_2}{2} (\mu \Delta \omega_2, \omega) -\tfrac{\overline{\rho}_1 - \overline{\rho}_2}{2} (\mu \nabla \omega_2, \nabla \omega) + ((c_{d,a,1} - c_{d,a,2}) \nabla \omega_2, \nabla \omega) \\
& \qquad + 2((\eta_{r,1} - \eta_{r,2}) (\curl_2 \bu_2 - 2 \omega_2), \omega) - (\bu \cdot \nabla \phi_2, \phi) \\
& \qquad - (F''(\phi_1) - F''(\phi_2), \nabla \phi_2 \cdot \nabla \phi) + (\nabla K \star \phi, \nabla \phi) + (\phi \nabla a, \nabla \phi) \\
& \quad =: D_{1} + \cdots + D_{26}.
\end{aligned}
\end{equation}
Recalling the positive constant $\hat{\alpha}$ in \eqref{sing:pot:ass5}, we can arguing similarly as in Section 6 of \cite{CCGAGMG} for $D_1, \dots, D_6$:
\begin{align*}
|D_1|& \leq C \| \phi \|_{L^4} \| \pd_t \bu_2 \| \| \bu \|_{L^4} \leq C\| \pd_t \bu_2 \| \| \phi \|^{1/2} ( \| \phi \|^{1/2} + \| \nabla \phi \|^{1/2}) \| \bu \|^{1/2} \| \D \bu \|^{1/2} \\
& \leq \frac{\eta_*}{10} \| \D \bu \|^2 + \frac{\hat{\alpha}}{30} \| \nabla \phi \|^2 + C( 1+ \| \pd_t \bu_2 \|^2) ( \| \phi \|^2 + \| \bu \|^2), \\
|D_2| & \leq C \| \bu \|_{L^4} \| \nabla \bu_2 \| \| \bu \|_{L^4} \leq \frac{\eta_*}{10} \| \D \bu \|^2 + C \| \bu \|^2, \\
|D_3|& \leq C \| \phi \|_{L^4} \| \bu_2 \|_{L^\infty} \| \nabla \bu_2 \| \| \bu \|_{L^4} \leq C(\| \phi \| + \| \nabla \phi \|) \| \bu_2 \|_{H^2}^{1/2} \| \bu \|^{1/2} \| \D \bu \|^{1/2} \\
& \leq \frac{\eta_*}{10} \| \D \bu \|^2 + \frac{\hat{\alpha}}{30} \| \nabla \phi \|^2 + C (1 +\| \bu_2 \|_{H^2}^{2}) (\| \bu \|^2 + C \| \phi \|^2), \\
|D_4| & \leq C|( (F'(\phi_1) - F'(\phi_2)) \Delta \bu_2, \bu) + (a \phi \Delta \bu_2, \bu) - ((K \star \phi ) \Delta \bu_2, \bu)| \\
& \leq C \| \phi \|_{L^4} \| \Delta \bu_2 \| \| \bu \|_{L^4} \\
& \leq C \| \phi \|^{1/2}( \| \phi \|^{1/2} + \| \nabla \phi \|^{1/2}) \| \Delta \bu_2 \| \| \bu \|^{1/2} \| \D \bu \|^{1/2} \\
& \leq \frac{\eta_*}{10} \| \D \bu \|^2 + \frac{\hat{\alpha}}{30} \| \nabla \phi \|^2 + C( 1 + \| \Delta \bu_2 \|^2)( \| \phi \|^2 + \| \bu \|^2), \\
|D_5| & \leq C |((F'(\phi_1) - F'(\phi_2)) \nabla \bu_2, \nabla \bu) + (a \phi \nabla \bu_2, \nabla\bu) - ((K \star \phi) \nabla \bu_2, \nabla \bu)| \\
& \leq C \| \phi \|_{L^4} \| \nabla \bu_2 \|_{L^4} \| \nabla \bu \|  \leq C \| \phi \|^{1/2} ( \| \phi \|^{1/2} + \| \nabla \phi \|^{1/2}) \| \nabla \bu_2 \|_{L^4}  \| \D \bu \| \\
& \leq \frac{\eta_*}{10} \| \D \bu \|^2 + \frac{\hat{\alpha}}{30} \| \nabla \phi \|^2 + C( 1 + \| \nabla \bu_2 \|_{L^4}^4) \| \phi \|^2, \\
|D_6| & \leq C \| \phi \|_{L^4} \| \D \bu_2 \|_{L^4} \| \nabla \bu \| \leq \frac{\eta_*}{10} \| \D \bu \|^2 + \frac{\hat{\alpha}}{30} \| \nabla \phi \|^2 + C( 1 + \| \nabla \bu_2 \|_{L^4}^4) \| \phi \|^2.
\end{align*}
For $D_7$ we proceed in the same fashion as $D_6$:
\[
|D_7| \leq C \| \phi \|_{L^4} \| \W \bu_2 \|_{L^4} \| \nabla \bu \| \leq \frac{\eta_*}{10}\| \D \bu \|^2 + \frac{\hat{\alpha}}{30} \| \nabla \phi \|^2 + C( 1 + \| \nabla \bu_2 \|_{L^4}^4) \| \phi \|^2.
\]
By the strict separation property of $\phi_i$, we see that
\begin{align*}
|D_8| & \leq C \|\nabla \phi_1 \| \| \bu \|_{L^4} \| \phi \|_{L^4} \leq C \| \nabla \phi_1 \| \| \phi \|^{1/2} ( \| \phi \|^{1/2} + \| \nabla \phi \|^{1/2}) \| \bu \|^{1/2} \| \D \bu \|^{1/2} \\
& \leq \frac{\eta_*}{10} \| \D \bu \|^2 + \frac{\hat{\alpha}}{30} \| \nabla \phi \|^2 + C (\| \phi \|^2 + \| \bu \|^2), \\
|D_9| + |D_{10}| & \leq C (\| \phi \|^2 + \| \bu \|^2), \\
|D_{11}| & \leq \frac{\hat{\alpha}}{30} \| \nabla \phi \|^2 + C \| \bu \|^2, \\
|D_{12}| & \leq C \| \phi \|_{L^4} \| \nabla \phi_2 \| \| \bu \|_{L^4} \leq C \| \phi \|^{1/2}( \| \phi \|^{1/2} + \| \nabla \phi \|^{1/2}) \| \bu \|^{1/2} \| \D \bu \|^{1/2} \\
& \leq \frac{\eta_*}{10} \| \D \bu \|^2 + \frac{\hat{\alpha}}{30} \| \nabla \phi \|^2 + C (\| \phi \|^2 + \| \bu \|^2), \\
|D_{13}| + |D_{14}| & \leq C (\| \phi \|^2 + \| \bu \|^2).
\end{align*}
We note that the estimation of $D_{16}, \dots, D_{21}$ proceed analogously as for $D_{1}, \dots, D_6$, and so 
\begin{align*}
 & |D_{16} + \cdots + D_{21}| \\
 & \quad \leq \frac{c_{d,a,*}}{2} \| \nabla \omega \|^2 + \frac{\hat{\alpha}}{30} \| \nabla \phi \|^2 + C(1 + \| \pd_t \omega_2 \|^2 + \| \Delta \omega_2 \|^2 + \| \nabla \omega_2 \|_{L^4}^4) (\| \phi \|^2 + \| \omega \|^2).
\end{align*}
For $D_{23}, \dots, D_{26}$ we have
\begin{align*}
|D_{23}|& = |(\bu \cdot \nabla \phi, \phi_2)| \leq \frac{\hat{\alpha}}{30} \| \nabla \phi \|^2 + C \| \bu \|^2, \\
|D_{24}| & \leq C \| \phi \|_{L^4} \| \nabla \phi_2 \|_{L^4} \| \nabla \phi \| \leq C \| \phi \|^{1/2}( \| \phi \|^{1/2} + \| \nabla \phi \|^{1/2}) \| \nabla \phi_2 \|_{L^4} \| \nabla \phi \| \\
& \leq \frac{\hat{\alpha}}{30} \| \nabla \phi \|^2 + C( 1 + \| \nabla \phi_2 \|_{L^4}^4) \| \phi \|^2, \\
|D_{25} + D_{26}| & \leq \frac{\hat{\alpha}}{30} \|\nabla \phi \|^2 + C \| \phi \|^2.
\end{align*}
Lastly, we look at the remaining terms $D_{15}$ and $D_{22}$:
\begin{align*}
|D_{15}| & \leq C \| \phi \|_{L^4} \| \omega_2 \|_{L^4} \| \nabla \bu \| \leq C \| \phi \|^{1/2}( \| \phi \|^{1/2} + \| \nabla \phi \|^{1/2}) \| \omega_2 \|_{L^4} \| \D \bu \| \\
& \leq \frac{\eta_*}{10} \| \D \bu \|^2 + \frac{\hat{\alpha}}{30} \| \nabla \phi \|^2 + C (1 + \| \omega_2 \|_{L^4}^4) \| \phi \|^2, \\
|D_{22}| & \leq C \| \phi \|_{L^4} ( \| \nabla \bu_2 \|_{L^4} + \| \omega_2 \|_{L^4}) \| \omega \|\\
&\leq C \| \phi \|^{1/2}( \| \phi \|^{1/2} + \| \nabla \phi \|^{1/2})( \| \nabla \bu_2 \|_{L^4} + \| \omega_2 \|_{L^4}) \| \omega \| \\
& \leq \frac{\hat{\alpha}}{30} \| \nabla \phi \|^2 + C(1 + \| \nabla \bu_2 \|_{L^4}^4 + \| \omega_2 \|_{L^4}^4) ( \| \phi \|^2 + \| \omega \|^2).
\end{align*}
Consequently, from \eqref{2Dstrsol:uni:ineq} we infer the differential inequality
\begin{align*}
& \frac{1}{2} \frac{d}{dt} \int_\Omega \rho_1( |\bu|^2 + |\omega|^2) + |\phi |^2 \, dx \\
& \qquad + \int_\Omega \eta_* | \D \bu |^2 + 4 \eta_{r,*} |\tfrac{1}{2} \curl_2 \bu - \omega|^2 + \frac{c_{d,a,*}}{2} |\nabla \omega|^2 + \frac{\hat{\alpha}}{2} |\nabla \phi |^2 \, dx \\
& \quad \leq C \mathcal{M}(t) \int_\Omega \frac{\rho_1}{2} (| \bu |^2 + | \omega |^2) + \| \phi \|^2 \, dx
\end{align*}
where
\[
\mathcal{M}(t) := (1 + \| \pd_t \bu_2 \|^2 + \| \bu \|_{H^2}^2 + \| \nabla \bu_2 \|_{L^4}^4 +\| \pd_t \omega_2 \|^2 +  \| \omega_2 \|_{H^2}^2 + \| \nabla \omega_2 \|_{L^4}^4 +  \| \nabla \phi_2 \|_{L^4}^4)(t).
\]
As $\mathcal{M} \in L^1(0,T)$, it follows by Gronwall's inequality that we obtain continuous dependence of the strong solutions on the initial data, and subsequently the uniqueness of strong solutions is established.

\section{Consistency estimates with nonlocal nonpolar models}\label{sec:consistency}
In this section we consider $\eta_r(\cdot)$ to be a finite positive scalar $\eta_r > 0$ and aim to derive estimates for the difference between the strong solutions of the nMAGG model and of the other nonlocal nonpolar models in terms of $\eta_r$.
\subsection{Consistency with nonlocal AGG solutions}
The well-posedness of global strong solutions to the nonlocal AGG (nAGG) model in two spatial dimensions (consisting of \eqref{2D:nonlocal:model:equ:div:0}, \eqref{2D:nonlocal:model:equ:bu} with $\eta_r = 0$, \eqref{2D:nonlocal:model:equ:CH} with $m(\phi) = 1$ and \eqref{2D:nonlocal:model:equ:mu}) has been established in Theorem 1.4 of \cite{CCGAGMG}. We remark that the regularities of nAGG strong solutions are the same as those in Theorem~\ref{thm:strwellposed}. The consistency estimates between nMAGG strong solutions and nAGG strong solutions are formulated as follows.
\begin{thm}\label{thm:nonpolar:nAGG}
For fixed $T> 0$ let $(\bu_a,p_a, \phi_a, \mu_a)$ denote the unique strong solution on $[0,T]$ to the nAGG model corresponding to initial data $(\bu_a(0), \phi_a(0)) = (\bu_0, \phi_0)$, and let $(\bu_w, \omega_w, p_w, \phi_w, \mu_w)$ denote the unique strong solution on $[0,T]$ to the nMAGG model \eqref{nonlocal:model:equ:2D} with $m(\phi) = 1$ and $\eta_r(\cdot) = \eta_r$ corresponding to the initial data $(\bu_w(0), \omega_w(0), \phi_w(0)) = (\bu_0, 0, \phi_0)$. Then, there exists a constant $C>0$ depending on the norms of the initial data, the terminal time $T$ and the parameters of the systems, but independent of $\eta_r$, such that
\begin{align}\label{nAGG:consist}
\sup_{t \in (0,T]} \Big ( \| \bu_w(t) - \bu_a(t) \|^2 + \| \omega_w(t) \|^2 + \| \phi_w(t) - \phi_a(t) \|^2 \Big ) \leq C \eta_r.
\end{align}
\end{thm}

\begin{proof}
Returning to the setting of the proof of continuous dependence and uniqueness, we set $(\bu_1, \omega_1, p_1, \phi_1, \mu_1) = (\bu_w, \omega_w, p_w, \phi_w,\mu_w)$ and $(\bu_2, \omega_2, p_2, \phi_2, \mu_2) = (\bu_a, 0, p_a, \phi_a, \mu_a)$. Then, the differences $(\bu, \omega_w, p, \phi, \mu)$ satisfy
\begin{subequations}
\begin{alignat}{2}
\label{nAGG:diff:equ:bu} & \rho_w \pd_t \bu + (\rho_w - \rho_a) \pd_t \bu_a + \rho_w (\bu_w \cdot \nabla) \bu + \rho_w (\bu \cdot \nabla) \bu_a + (\rho_w - \rho_a) (\bu_a \cdot \nabla) \bu_a \\
\notag & \qquad - \tfrac{\overline{\rho}_1 - \overline{\rho}_2}{2} ((\nabla \mu_w \cdot \nabla) \bu + (\nabla \mu \cdot \nabla)\bu_a) - 2\div ( \eta_w \D \bu + (\eta_w - \eta_a) \D \bu_a) + \nabla p \\
\notag & \quad = 2 \eta_r \div (\W \bu_w) + \mu_w \nabla \phi + \mu \nabla \phi_a + 2 \eta_r \curl_1 \omega_w, \\
\label{nAGG:diff:equ:bo} & \rho_w \pd_t \omega_w + \rho_w \bu_w \cdot \nabla \omega_w - \tfrac{\overline{\rho}_1 - \overline{\rho}_2}{2} \nabla \mu_w \cdot \nabla \omega_w - \div (c_{d,a}(\phi_w) \nabla \omega_w)  \\
\notag & \quad = 2 \eta_{r} (\curl_2 \bu_w - 2 \omega_w), \\
\label{nAGG:uni:diff:equ:CH} & \pd_t \phi + \bu_w \cdot \nabla \phi + \bu \cdot \nabla \phi_a = \Delta \mu, \\
\label{nAGG:uni:diff:equ:mu} & \mu = F'(\phi_w) - F'(\phi_a) + a \phi - K \star \phi.
\end{alignat}
\end{subequations}
We note that the analogues of the identities \eqref{uniq:id}, \eqref{uniq:id3} and \eqref{uniq:id4} remain valid in our setting, and hence upon testing \eqref{nAGG:diff:equ:bu} by $\bu$, \eqref{nAGG:diff:equ:bo} by $\omega_w$, \eqref{nAGG:uni:diff:equ:CH} by $\phi$ and summing leads to 
\begin{equation}\label{nAGG:uni:ineq}
    \begin{aligned}
& \frac{1}{2} \frac{d}{dt} \int_\Omega \rho_w (|\bu|^2 + |\omega_w|^2) + |\phi|^2 \, dx \\
& \qquad + \int_\Omega 2 \eta_a |\D \bu|^2 + 4 \eta_{r} |\tfrac{1}{2} \curl_2 \bu_w - \omega_w|^2 + c_{d,a}(\phi_w)|\nabla \omega_w|^2 + (a + F''(\phi_w)) |\nabla \phi|^2 \, dx  \\
& \quad = - ((\rho_w - \rho_a)\pd_t \bu_a, \bu) - (\rho_w (\bu \cdot \nabla) \bu_a, \bu) - ((\rho_w - \rho_a) (\bu_a \cdot \nabla) \bu_a, \bu) \\
& \qquad - \tfrac{\overline{\rho}_1 - \overline{\rho}_2}{2} ( \mu \Delta \bu_a, \bu) -\tfrac{\overline{\rho}_1 - \overline{\rho}_2}{2} (\mu \nabla \bu_a, \nabla \bu) - 2 ((\eta_w - \eta_a) \D \bu_a, \nabla \bu) \\
& \qquad + 2 \eta_r( \W \bu_w, \nabla \bu_a) - ((a + F''(\phi_w)) \nabla \phi_w, \phi \bu) - (\phi_w \nabla a, \phi \bu)\\
& \qquad + ( \nabla K \star \phi_w, \phi \bu) - 2 \eta_r( w_w, \curl_2 \bu_a) - ((a + F''(\phi_w)) \nabla \phi, \phi_a \bu) \\
& \qquad - ((F''(\phi_w) - F''(\phi_a)) \nabla \phi_a, \phi_a \bu) - (\phi \nabla a, \phi_a \bu)  + (\nabla K \star \phi, \phi_a \bu) \\
& \qquad - (\bu \cdot \nabla \phi_a, \phi) - (F''(\phi_w) - F''(\phi_a), \nabla \phi_a \cdot \nabla \phi) + (\nabla K \star \phi, \nabla \phi) + (\phi \nabla a, \nabla \phi) \\
& \quad =: E_{1} + \cdots + E_{19}.
    \end{aligned}
\end{equation}
We note that $E_1, \dots, E_6$ are analogous to $D_1, \dots, D_6$, while $E_8, E_9, E_{10}, E_{12}, \dots, E_{19}$ are analogous to $D_8, \dots, D_{14}, D_{23}, \dots, D_{26}$. The difference lies in the new terms $E_7$ and $E_{11}$, which we estimate as follows:
\[
|E_7 + E_{11}| = 2\eta_r|(\curl_2 \bu_a, \tfrac{1}{2}\curl_2 \bu_w -  \omega_w)| \leq 2 \eta_r \| \tfrac{1}{2} \curl_2 \bu_w - \omega_w \|^2 + C \eta_r \| \nabla \bu_a \|^2.
\]
Then, revisiting the estimates for the proof of continuous dependence and uniqueness, and adjusting the prefactors slightly, we infer from \eqref{nAGG:uni:ineq} the following differential inequality
\begin{align*}
& \frac{1}{2} \frac{d}{dt} \int_\Omega \rho_w( |\bu|^2 + |\omega_w|^2) + |\phi |^2 \, dx \\
& \qquad + \int_\Omega \eta_* | \D \bu |^2 + 2 \eta_{r} |\tfrac{1}{2} \curl_2 \bu_w - \omega_w|^2 + c_{d,a,*} |\nabla \omega_w|^2 + \frac{\hat{\alpha}}{2} |\nabla \phi |^2 \, dx \\
& \quad \leq C \widetilde{\mathcal{M}}(t) \int_\Omega \frac{\rho_w}{2} (| \bu |^2 + | \omega_w |^2) + \| \phi \|^2 \, dx + C \eta_r,
\end{align*}
where
\[
\widetilde{\mathcal{M}} := (1 + \| \pd_t \bu_a \|^2 + \| \bu_a \|_{H^2}^2 + \| \nabla \bu_a \|_{L^4}^4 + \| \nabla \phi_a \|_{L^4}^4)(t) \in L^1(0,T).
\]
Hence, \eqref{nAGG:consist} follows from the application of Gronwall's inequality.
\end{proof}

\subsection{Consistency with nonlocal Model H solutions}
Let $\overline{\rho}>0$ be a fixed constant and we introduce the nonlocal Model H (nModelH) in two spatial dimensions to be
\begin{subequations}\label{nonlocal:modelH}
\begin{alignat}{2}
& \div \bu_h = 0, \\[1ex]
& \overline{\rho} (\pd_t \bu_h + (\bu_h \cdot \nabla) \bu_h) - \div( 2\eta(\phi_h)\D\bu_h) + \nabla p_h = \mu_h \nabla \phi_h, \\[1ex]
& \partial_{t}\phi_h + \bu_h \cdot \nabla\phi_h = \Delta \mu_h, \\[1ex] 
& \mu_h = a\phi_h - K \star \phi_h + F'(\phi_h),
\end{alignat}
\end{subequations}
which can be obtained by setting $\overline{\rho}_1 = \overline{\rho}_2 = \overline{\rho}$ in the nAGG model. We establish a consistency estimate between strong solutions of the nAGG model and of the nModelH that is analogous to Theorem 1.10 of \cite{CCGAGMG} but under stronger norms. Then, combining with Theorem \ref{thm:nonpolar:nAGG} we infer the following result.
\begin{thm}\label{thm:nonpolar:nModelH}
For fixed $T> 0$ let $(\bu_h,p_h, \phi_h, \mu_h)$ denote the unique strong solution on $[0,T]$ to nModelH corresponding to initial data $(\bu_h(0), \phi_h(0)) = (\bu_0, \phi_0)$, and let $(\bu_w, \omega_w, p_w, \phi_w, \mu_w)$ denote the unique strong solution on $[0,T]$ to the nMAGG model \eqref{nonlocal:model:equ:2D} with $m(\phi) = 1$ and $\eta_r(\cdot) = \eta_r$ corresponding to the initial data $(\bu_w(0), \omega_w(0), \phi_w(0)) = (\bu_0, 0, \phi_0)$. Then, there exists a constant $C>0$ depending on the norms of the initial data, the terminal time $T$ and the parameters of the systems, but independent of $\eta_r$, $|\overline{\rho}_1 - \overline{\rho}_2|$ and $|\tfrac{1}{2}(\overline{\rho}_1 + \overline{\rho}_2) - \overline{\rho}|$, such that
\begin{equation}\label{nH:consist}
\begin{aligned}
& \sup_{t \in (0,T]} \Big ( \| \bu_w(t) - \bu_h(t) \|^2 + \| \omega_w(t) \|^2 + \| \phi_w(t) - \phi_h(t) \|^2 \Big ) \\
& \quad \leq C \Big (\eta_r +|\tfrac{1}{2}(\overline{\rho}_1 + \overline{\rho}_2) - \overline{\rho}|^2 + |\overline{\rho}_1 - \overline{\rho}_2|^2 \Big ).
\end{aligned}
\end{equation}
\end{thm}

\begin{proof}
It suffices to provide an estimate between strong solutions of the nAGG model and of the nModelH in terms of $|\overline{\rho}_1 - \overline{\rho}_2|$ and $|\frac{1}{2}(\overline{\rho}_1 + \overline{\rho}_2) - \overline{\rho}|$. Returning to the setting of the proof of continuous dependence and uniqueness, we set $(\bu_1, \omega_1, p_1, \phi_1, \mu_1) = (\bu_a, 0, p_a, \phi_a, \mu_a)$ and $(\bu_2, \omega_2, p_2, \phi_2, \mu_2) = (\bu_h, 0, p_h, \phi_h, \mu_h)$. Then, dropping all terms involving $\omega$ we see that the differences $(\bu, p, \phi, \mu)$ satisfy
\begin{subequations}
\begin{alignat}{2}
\label{nH:diff:equ:bu} & \rho_a \pd_t \bu + (\rho_a - \overline{\rho}) \pd_t \bu_h + \rho_a (\bu_a \cdot \nabla) \bu + \rho_a (\bu \cdot \nabla) \bu_h + (\rho_a - \overline{\rho}) (\bu_h \cdot \nabla) \bu_h \\
\notag & \qquad - \tfrac{\overline{\rho}_1 - \overline{\rho}_2}{2} (\nabla \mu_a \cdot \nabla) \bu_a  - 2\div ( \eta_a \D \bu + (\eta_a - \eta_h) \D \bu_h) + \nabla p \\
\notag & \quad = \mu_a \nabla \phi + \mu \nabla \phi_h, \\
\label{nH:uni:diff:equ:CH} & \pd_t \phi + \bu_a \cdot \nabla \phi + \bu \cdot \nabla \phi_h = \Delta \mu, \\
\label{nH:uni:diff:equ:mu} & \mu = F'(\phi_a) - F'(\phi_h) + a \phi - K \star \phi.
\end{alignat}
\end{subequations}
Similar to \eqref{uniq:id:1} we have the identity
\begin{align*}
0 = \int_\Omega -\frac{|\bu|^2}{2} \pd_t \rho_a + \rho_a(\bu_a \cdot \nabla) \bu \cdot \bu - \frac{\overline{\rho}_1 - \overline{\rho}_2}{2} (\nabla \mu_a \cdot \nabla) (\bu_a - \bu_h) \cdot \bu\, dx,
\end{align*}
while recalling \eqref{uniq:id3} and \eqref{uniq:id4}, we obtain upon testing \eqref{nH:diff:equ:bu} by $\bu$, \eqref{nH:uni:diff:equ:CH} by $\phi$ and then summing
\begin{equation}\label{nH:uni:ineq}
    \begin{aligned}
& \frac{1}{2} \frac{d}{dt} \int_\Omega \rho_a |\bu|^2 + |\phi|^2 \, dx + \int_\Omega 2 \eta_a |\D \bu|^2 + (a + F''(\phi_a)) |\nabla \phi|^2 \, dx  \\
& \quad = - ((\rho_a - \overline{\rho})\pd_t \bu_h, \bu) - (\rho_a (\bu \cdot \nabla) \bu_h, \bu) - ((\rho_a - \overline{\rho}) (\bu_h \cdot \nabla) \bu_h, \bu) \\
& \qquad - \tfrac{\overline{\rho}_1 - \overline{\rho}_2}{2} ( \mu_a \Delta \bu_h, \bu) - \tfrac{\overline{\rho}_1 - \overline{\rho}_2}{2} (\mu_a \nabla \bu_h, \nabla \bu) - 2 ((\eta_a - \eta_h) \D \bu_h, \nabla \bu) \\
& \qquad - ((a + F''(\phi_a)) \nabla \phi_a, \phi \bu) - (\phi_a \nabla a, \phi \bu) + ( \nabla K \star \phi_a, \phi \bu) \\
& \qquad - ((a + F''(\phi_a)) \nabla \phi, \phi_h \bu) - ((F''(\phi_a) - F''(\phi_h)) \nabla \phi_h, \phi_h \bu) \\
& \qquad - (\phi \nabla a, \phi_h \bu)  + (\nabla K \star \phi, \phi_h \bu) - (\bu \cdot \nabla \phi_h, \phi) \\
& \qquad - (F''(\phi_a) - F''(\phi_h), \nabla \phi_h \cdot \nabla \phi) + (\nabla K \star \phi, \nabla \phi) + (\phi \nabla a, \nabla \phi) \\
& \quad =: F_{1} + \cdots + F_{17}.
    \end{aligned}
\end{equation}
We note that $F_1, \dots, F_6$ are analogous to $D_1, \dots, D_6$, while $F_7, \dots, F_{17}$ are analogous to $D_8, \dots , D_{14}, D_{23} \dots, D_{26}$, although $F_1$, $F_3$, $F_4$ and $F_5$ have to be treated differently. We begin estimating as follows:
\begin{align*}
 |F_2 + F_6|& \leq \frac{\eta_*}{3} \| \D \bu \|^2 + \frac{\hat{\alpha}}{4} \| \nabla \phi \|^2 + C( 1 + \| \bu_h \|_{H^2}^2)( \| \phi \|^2 + \| \bu \|^2), \\
 |F_{7} + \cdots + F_{17}| & \leq \frac{\eta_*}{3} \| \D \bu \|^2 + \frac{\hat{\alpha}}{4} \| \nabla \phi \|^2 + C(1 + \| \nabla \phi_h \|_{L^4}^4)( \| \phi \|^2 + \| \bu \|^2).
\end{align*}
For $F_1$ and $F_3$ we use the explicit form for $\rho(\phi_a) = \frac{\overline{\rho}_1 - \overline{\rho}_2}{2} \phi_a + \frac{\overline{\rho}_1 + \overline{\rho}_2}{2}$ to estimate
\begin{align*}
 |F_1| & \leq \Big ( |\tfrac{1}{2} (\overline{\rho}_1 + \overline{\rho}_2) - \overline{\rho})| + \tfrac{1}{2} \| \phi_a \|_{L^\infty} |\overline{\rho}_1 - \overline{\rho}_2| \Big ) \| \pd_t \bu_h\| \| \bu \| \\
 & \leq C \| \pd_t \bu_h \|^2 \| \bu \|^2 + C (|\tfrac{1}{2} (\overline{\rho}_1 + \overline{\rho}_2) - \overline{\rho})|^2 + |\overline{\rho}_1 - \overline{\rho}_2|^2), \\
 |F_3| & \leq \Big ( |\tfrac{1}{2} (\overline{\rho}_1 + \overline{\rho}_2) - \overline{\rho})| + \tfrac{1}{2} \| \phi_a \|_{L^\infty} |\overline{\rho}_1 - \overline{\rho}_2| \Big )  \| \bu_h \|_{L^4} \| \nabla \bu_h \|_{L^4} \| \bu \| \\
 & \leq C \| \nabla \bu_h \|_{L^4}^4 \| \bu \|^2 + C (|\tfrac{1}{2} (\overline{\rho}_1 + \overline{\rho}_2) - \overline{\rho})|^2 + |\overline{\rho}_1 - \overline{\rho}_2|^2).
\end{align*}
For $F_4$ and $F_5$ we have
\begin{align*}
|F_4 + F_5| & \leq C|\overline{\rho}_1 - \overline{\rho}_2| \| F'(\phi_a) + a \phi_a - K \star \phi_a \|_{L^\infty} (\| \Delta \bu_h \| \|\bu\| + \|\nabla \bu_h \| \| \D \bu \|) \\
& \leq \frac{\eta_*}{3} \| \D \bu \|^2 + C \| \bu \|^2 + C| \overline{\rho}_1 - \overline{\rho}_2|^2 \| \bu_h \|_{H^2}^2 .
\end{align*}
Hence, we deduce from \eqref{nH:uni:ineq} the differential inequality
\begin{align*}
& \frac{1}{2} \frac{d}{dt} \int_\Omega \rho_a |\bu|^2 + |\phi|^2 \, dx + \int_\Omega \eta_a|\D \bu|^2 + \frac{\hat{\alpha}}{2} |\nabla \phi|^2 \, dx \\
& \quad \leq C (1 + \| \bu_h \|_{H^2}^2 + \| \nabla \phi_h \|_{L^4}^4 + \| \pd_t \bu_h |^2) ( \| \bu \|^2 + \| \phi \|^2) \\
& \qquad + C(1 + \| \bu_h \|_{H^2}^2) (|\tfrac{1}{2} (\overline{\rho}_1 + \overline{\rho}_2) - \overline{\rho})|^2 + |\overline{\rho}_1 - \overline{\rho}_2|^2).
\end{align*}
By Gronwall's inequality we arrive at
\begin{align*}
& \sup_{t \in (0,T]} \Big ( \| \bu_a(t) - \bu_h(t) \|^2 + \| \phi_a(t) - \phi_h(t) \|^2 \Big ) + \int_0^T \| \D (\bu_a - \bu_h) \|^2 + \| \nabla (\phi_a - \phi_h) \|^2 \, dt \\
& \quad \leq C (|\tfrac{1}{2} (\overline{\rho}_1 + \overline{\rho}_2) - \overline{\rho})|^2 + |\overline{\rho}_1 - \overline{\rho}_2|^2).
\end{align*}
Together with \eqref{nAGG:consist} and the triangle inequality we infer the consistency estimate \eqref{nH:consist}.
\end{proof}

\section*{Acknowledgments}
\noindent KFL gratefully acknowledges the support by the Research Grants Council of the Hong Kong Special Administrative Region, China [Project No.: HKBU 14303420].

\end{document}